\documentclass{amsart}
\usepackage{amssymb, amsmath,amsthm }
\usepackage{relsize}
\usepackage{hyperref}
\usepackage{stmaryrd}
\usepackage[all]{xy}
\usepackage{color}
\usepackage{enumitem}
\usepackage{tikz-cd}
\newcommand{\gen}{\mathrm{gen}}

\newcommand{\rig}{\mathrm{rig}}
\newtheorem{prop}{Proposition}[section]
\newtheorem{defi}[prop]{Definition}

\newtheorem{lem}[prop]{Lemma}
\newtheorem{cor}[prop]{Corollary}
\newtheorem{thm}[prop]{Theorem}
\theoremstyle{remark}

\newtheorem{remar}[prop]{Remark}

\DeclareMathOperator{\Mod}{Mod}

\DeclareMathOperator{\Ban}{Ban}

\newcommand{\op}{\mathrm{op}}

\DeclareMathOperator{\Hom}{Hom}

\DeclareMathOperator{\End}{End}

\DeclareMathOperator{\Ker}{Ker}

\DeclareMathOperator{\id}{id}

\DeclareMathOperator{\Ext}{Ext}

\DeclareMathOperator{\wtimes}{\widehat{\otimes}}

\newcommand{\cHom}[1]{\Hom^{\cont}_{#1}}
\DeclareMathOperator{\mSpec}{m-Spec}
\DeclareMathOperator{\Frac}{Frac}

\def\rank{\mathop{\mathrm{rank}}\nolimits}

\newcommand{\tr}{\operatorname{tr}}

\DeclareMathOperator{\Spec}{Spec}

\DeclareMathOperator{\supp}{Supp}
\DeclareMathOperator{\ann}{ann}

\DeclareMathOperator{\cosoc}{cosoc}

\DeclareMathOperator{\Ind}{Ind}
\DeclareMathOperator{\cInd}{c-Ind}

\DeclareMathOperator{\Sym}{Sym}

\DeclareMathOperator{\Irr}{Irr}
\DeclareMathOperator{\Spf}{Spf}

\newcommand{\inv}{\mathrm{inv}}
\newcommand{\invt}{\mathrm{inv.t}}
\newcommand{\PC}{\mathrm{PC}}
\newcommand{\rE}{\mathfrak r(E)}
\newcommand{\rM}{\mathfrak r(M)}
\newcommand{\EB}{E_{\BB}}
\newcommand{\PB}{P_{\BB}}
\newcommand{\ZB}{Z_{\BB}}
\newcommand{\MB}{M_{\BB}}

\newcommand{\msm}{\Mod^{\sm}}
\newcommand{\mlf}{\Mod^{\lfin}}
\newcommand{\mpro}{\Mod^{\pro}}

\newcommand{\br}[1]{\llbracket #1\rrbracket}
\newcommand{\sm}{\mathrm{sm}}
\newcommand{\lfin}{\mathrm{l.fin}}
\newcommand{\adm}{\mathrm{adm}}

\newcommand{\pro}{\mathrm{pro}}
\newcommand{\aBan}{\Ban^{\adm}}

\newcommand{\cont}{\mathrm{cont}}

\DeclareMathOperator{\GL}{GL}
\DeclareMathOperator{\SL}{SL}
\DeclareMathOperator{\Sp}{Sp}

\DeclareMathOperator{\Ord}{Ord}

\DeclareMathOperator{\Gal}{Gal}
\DeclareMathOperator{\Art}{Art}

\newcommand{\gal}{\mathcal G_{\Qp}}

\newcommand{\rhobar}{\bar{\rho}}
\newcommand{\ps}{\mathrm{ps}}
\newcommand{\cV}{\check{\mathbf{V}}}

\newcommand{\tf}{\mathrm{tf}}

\newcommand{\Qtbar}{\overline{\mathbb{Q}}_3}

\newcommand{\Qp}{\mathbb{Q}_p}
\newcommand{\Zp}{\mathbb{Z}_p}
\newcommand{\Qpbar}{\overline{\mathbb{Q}}_p}

\newcommand{\Fp}{\mathbb F_p}

\newcommand{\unif}{\varpi}

\newcommand{\unifmatrix}{\left(\smallmatrix \unif & 0 \\ 0 & \unif \endsmallmatrix\right)}

\DeclareMathOperator{\dualcat}{\mathfrak C}

\newcommand{\BB}{\mathfrak B}
\newcommand{\CC}{\mathfrak C}
\newcommand{\TT}{\mathfrak T}
\newcommand{\mm}{\mathfrak m}

\newcommand{\rr}{\mathfrak r}
\newcommand{\pp}{\mathfrak p}
\newcommand{\qq}{\mathfrak q}

\newcommand{\GG}{\mathcal G}

\newcommand{\ab}{\mathrm{ab}}
\newcommand{\OO}{\mathcal O}

\newcommand{\cT}{\mathcal T}

\newcommand{\QQ}{\mathbb Q}
\newcommand{\RR}{\mathbb R}
\newcommand{\ZZ}{\mathbb Z}
\newcommand{\NN}{\mathbb{N}}

\newcommand{\VV}{\mathbf V}

\newcommand{\Eins}{\mathbf 1}

\newcommand{\md}{\mathrm m}

\newcommand{\fg}{\mathrm{fg}}
\newcommand{\fl}{\mathrm{fl}}
\newcommand{\CH}{\mathrm{CH}}

\title[Finiteness properties]{Finiteness properties of the category of mod $p$ representations of $\GL_2(\Qp)$}
\author{Vytautas Pa\v{s}k\={u}nas}
\author{Shen-Ning Tung}
\date{\today.}
\begin{document} 
\maketitle

\begin{abstract} We establish Bernstein-centre type of results for the category of mod $p$ 
representations of $\GL_2(\Qp)$. We treat all the remaining open cases, which occur  when
$p$ is $2$ or $3$. Our arguments carry over for all primes $p$.
 This allows 
us to remove the restrictions on the residual representation at $p$ in Lue Pan's recent proof of the Fontaine--Mazur conjecture for 
Hodge--Tate representations of $\Gal(\overline{\mathbb Q}/\mathbb{Q})$ with equal Hodge--Tate weights. 
\end{abstract}
\tableofcontents
\section{Introduction}

Recently,  Lue Pan has given a new proof of the Fontaine--Mazur conjecture for Hodge--Tate 
representations of $\Gal(\overline{\mathbb Q}/\mathbb{Q})$ with equal Hodge--Tate weights, 
\cite{luepan}. One of the ingredients in his proof are the Bernstein-centre type of results for the
 category of mod $p$ representations of $\GL_2(\Qp)$ proved in \cite{image}, \cite{2adicANT}. 
 This caused him to impose some restrictions on the Galois representations at $p$, when $p=2$
  and $p=3$. In this paper we remove these restrictions by proving the required 
  finiteness results in these remaining cases, see Theorem \ref{application}.   

Let $L$ be a finite extension of $\Qp$ with ring of integers $\OO$ and residue field $k$. Let $G=\GL_2(\Qp)$ and let $\Mod^{\sm}_{G}(\OO)$ be the category of smooth $G$-representations on $\OO$-torsion modules. Let $\Mod^{\lfin}_{G}(\OO)$ be the full subcategory of $\Mod^{\sm}_{G}(\OO)$ consisting of representations, which are equal to the union of  their subrepresentations of finite length. This is equivalent to the requirement that every finitely generated subrepresentation is of finite length and we call such representations \textit{locally finite}. We fix  a character $\zeta:Z\rightarrow \OO^{\times}$ of the centre of $G$, and let $\Mod^{\lfin}_{G, \zeta}(\OO)$ be the full subcategory of $\Mod^{\lfin}_{G}(\OO)$ consisting of representations with central character $\zeta$. The category $\Mod^{\lfin}_{G, \zeta}(\OO)$ is locally finite and by general results of Gabriel \cite{MR232821} decomposes as a product $$\Mod^{\lfin}_{G, \zeta}(\OO)\cong \prod_{\BB} \Mod^{\lfin}_{G, \zeta}(\OO)_{\BB}$$ 
of indecomposable subcategories, called blocks. Moreover, each block is anti-equivalent to the category of pseudo-compact modules over a pseudo-compact ring $E_{\BB}$. The centre of the ring $E_{\BB}$, which we denote by $Z_{\BB}$, is naturally isomorphic to the centre of the category  $\Mod^{\lfin}_{G, \zeta}(\OO)_{\BB}$, which by definition is the ring of natural 
transformations of the identity functor. This means that $Z_{\BB}$ acts functorially on every object in $\Mod^{\lfin}_{G, \zeta}(\OO)$. The finiteness result in the title is an analogue of a result 
of Bernstein in the theory of smooth representations of $p$-adic groups on $\mathbb{C}$-vector spaces, 
\cite[Proposition 3.3]{bernstein}: 

\begin{thm}\label{intro_finite} The ring $\ZB$ is noetherian and $\EB$ is a finitely generated $\ZB$-module. 
\end{thm}

To prove the theorem we use in an essential way the direct connection between the  $\GL_2(\Qp)$ representations and the representations of the absolute Galois group of $\Qp$, which we denote by $\gal$, discovered by Colmez in \cite{MR2642409}, via his celebrated Montreal functor $\cV$, which we review in section \ref{CMF}.

For each fixed block $\Mod^{\lfin}_{G, \zeta}(\OO)_{\BB}$ there is a finite extension $L'$ of $L$ with ring of integers $\OO'$, such that $\Mod^{\lfin}_{G, \zeta}(\OO)_{\BB}\otimes_\OO \OO'$ decomposes into a finite product of indecomposable subcategories, each of which remain indecomposable after a further extension of scalars. Such absolutely indecomposable blocks have been classified in \cite{MR3444235} and they correspond to semi-simple representations $\rhobar: \gal\rightarrow \GL_2(k)$, which are either absolutely irreducible or a direct sum of two characters. This bijection realizes the semi-simple mod $p$ local Langlands correspondence established in a visionary paper of Breuil, \cite{MR2018825}. 

If $\rhobar$ is absolutely irreducible then $\Mod^{\lfin}_{G, \zeta}(\OO)_{\BB_{\rhobar}}$ contains only one irreducible object $\pi$, satisfying $\cV(\pi^{\vee})\cong \rhobar$, where $\vee$ denotes Pontryagin dual. 
Moreover, $\pi$ is not a subquotient of any parabolically induced representation; such representations are called \textit{supersingular}.

If $\rhobar=\chi_1 \oplus \chi_2$, where $\chi_1, \chi_2: \gal\rightarrow k^{\times}$ are characters, then the irreducible objects in $\Mod^{\lfin}_{G, \zeta}(\OO)_{\BB_{\rhobar}}$ are the irreducible subquotients of the representation 
$$ (\Ind^G_B \chi_1\otimes \chi_2 \omega^{-1})_{\sm}\oplus (\Ind^G_B \chi_2\otimes \chi_1 \omega^{-1})_{\sm},$$ 
where we consider $\chi_1$, $\chi_2$ as characters of $\Qp^{\times}$ via the Artin map $\Art_{\Qp}: \Qp^{\times} \rightarrow \gal^{\ab}$ and $\omega(x)= x |x|\pmod{p}$ for all $x\in \Qp^{\times}$ corresponds to the cyclotomic character modulo $p$, see section \ref{section:blocks} for an explicit list. 

All the blocks, except when $p$ is either $2$ or $3$ and $\rhobar=\chi\oplus \chi\omega$, have been well understood in \cite{image}, \cite{2adicANT}. These \textit{exceptional blocks} are the main focus of this paper, however our arguments work for all $p$ and all blocks.

   The action of  $Z_{\BB_{\rhobar}}$ induces 
   a functorial ring homomorphism 
$$c_\tau: Z_{\BB_{\rhobar}} \rightarrow \End_{G}(\tau),$$ 
for every object $\tau$ in $\Mod^{\lfin}_{G, \zeta}(\OO)_{\BB_{\rhobar}}$. Since $\cV$ is a functor, it induces a ring homomorphism 
$$\End_{G}(\tau)\rightarrow \End^{\cont}_{\gal}(\cV(\tau^{\vee}))^{\op},\quad  \varphi\mapsto \cV(\varphi).$$ 
We denote the action of $\gal$ on $\cV(\tau^{\vee})$ by $\rho_{\cV(\tau^{\vee})}$.

Let $R^{\ps, \zeta\varepsilon}_{\tr \rhobar}$ be the universal deformation ring parameterising pseudorepresentations lifting $\tr \rhobar$ with determinant $\zeta\varepsilon$, where $\varepsilon$ is the $p$-adic cyclotomic character. We may evaluate the universal pseudorepresentation 
$T: \gal \rightarrow R^{\ps, \zeta\varepsilon}_{\tr\rhobar}$ at $g\in \gal$  to obtain an element $T(g)\in R^{\ps, \zeta\varepsilon}_{\tr\rhobar}$.

\begin{thm}\label{intro_explain} There is an $\OO$-algebra homomorphism 
\begin{equation}\label{main_point}
\upsilon:R^{\ps, \zeta\varepsilon}_{\tr \rhobar}\rightarrow Z_{\BB_{\rhobar}},
\end{equation}
 satisfying the following compatibility property with the Colmez's functor.
For all $\tau\in\Mod^{\lfin}_{G, \zeta}(\OO)_{\BB_{\rhobar}}$  and all $g\in \gal$ we have
$$\cV(c_{\tau}(\upsilon(T(g))))= \rho_{\cV(\tau^{\vee})}(g) + \rho_{\cV(\tau^{\vee})}(g^{-1}) \zeta\varepsilon(g),$$
in $\End_{\gal}^{\cont}(\cV(\tau^{\vee}))$.
\end{thm}

The construction of the map \eqref{main_point} is the main point of the paper. Outside the exceptional cases \eqref{main_point}
has been established in \cite{image}, \cite{2adicANT} using  a different argument to the present paper.  Our main result is 

\begin{thm}\label{main_intro} 
The map \eqref{main_point} makes $Z_{\BB_{\rhobar}}$ and $E_{\BB_{\rhobar}}$ into finitely generated $R^{\ps, \zeta\varepsilon}_{\tr\rhobar}$-modules.
\end{thm}

Since $R^{\ps, \zeta\varepsilon}_{\tr\rhobar}$ is known to be noetherian by the work of Chenevier \cite{che_durham}, Theorem \ref{main_intro} implies Theorem \ref{intro_finite}. Moreover, Theorems \ref{main_intro},  \ref{intro_explain} are sufficient to remove the restrictions in Lue Pan's paper. Further, we can reprove most of the results concerning Banach space representations in \cite{image}, see section \ref{main_banach} and Corollary \ref{irr_irr}.

To give a flavour of the results on Banach space representations, we will explain the following special case. Let $\Ban^{\adm}_{G, \zeta}(L)$ be the category of admissible unitary $L$-Banach space representations of $G$ with central character $\zeta$. This category is abelian, \cite{st_iw}.  By \cite{CDP}, Colmez's functor induces a bijection between  the equivalence classes of absolutely irreducible non-ordinary $\Pi \in \Ban^{\adm}_{G, \zeta}(L)$ and absolutely irreducible Galois representations $\rho: \gal\rightarrow \GL_2(L)$. We show that there are no extensions between $\Pi$ and other irreducible representations in $\Ban^{\adm}_{G, \zeta}(L)$. Hence $\Ban^{\adm}_{G, \zeta}(L)^{\fl}_{\Pi}$ is a direct summand of $\Ban^{\adm}_{G, \zeta}(L)^{\fl}$, where the superscript $\fl$ indicates finite length and subscript $\Pi$ indicates that all the irreducible subquotients are isomorphic to $\Pi$. We show in Corollary \ref{irr_irr} that $\cV$ induces an anti-equivalence of categories between $\Ban^{\adm}_{G, \zeta}(L)^{\fl}_{\Pi}$ and the category of modules of finite length over $R^{\zeta\varepsilon}_{\rho}$, the universal deformation ring of $\rho$ parameterizing deformations of $\rho$ with determinant 
equal to $\zeta \varepsilon$ to local artinian $L$-algebras. Such results were known before for $p\ge 5$, \cite{image} and under assumptions on the reduction modulo $p$ of a $G$-invariant lattice in $\Pi$ if $p=2$ or $p=3$, \cite{2adicANT}.

Let $R^{\ps, \zeta\varepsilon}_{\tr \rhobar}\br{\gal}/J$ be the largest quotient of $R^{\ps, \zeta\varepsilon}_{\tr \rhobar}\br{\gal}$ such that the Cayley--Ha\-mil\-ton 
theorem holds for the universal pseudorepresentation with determinant $\zeta\varepsilon$ lifting $\tr \rhobar$. Such  algebras have been studied by Bella\"{i}che--Chenevier \cite{bel_che}, Chenevier \cite{che_durham}; they play a key role in this paper. The subscript $\tf$ will indicate the maximal $\OO$-torsion free quotient.
Our second main result asserts that 

\begin{thm}\label{intro_main2} 
The essential image of $\Mod^{\lfin}_{G, \zeta}(\OO)_{\BB_{\rhobar}}$ under $\cV$ is anti-equivalent to the category of pseudo-compact 
$(R^{\ps, \zeta\varepsilon}_{\tr \rhobar}\br{\gal}/J)_{\tf}$-modules. The map \eqref{main_point} induces an isomorphism  
\begin{equation}\label{intro_invertp}
R^{\ps, \zeta\varepsilon}_{\tr \rhobar}[1/p]\overset{\cong}{\longrightarrow}Z_{\BB_{\rhobar}}[1/p].
\end{equation}
Moreover, if $p\neq 2$ then $Z_{\BB_{\rhobar}}= (R^{\ps, \zeta\varepsilon}_{\tr\rhobar_{\BB}})_{\tf}$ and if $p=2$ then the cokernel of \eqref{main_point} is killed by $2$.
\end{thm}

\begin{cor}\label{intro_tf_red} 
$Z_{\BB_{\rhobar}}$ is a complete local noetherian $\OO$-algebra  with residue field 
$k$. It is $\OO$-torsion free and $Z_{\BB_{\rhobar}}[1/p]$ is normal. 
\end{cor}

If we are not in  the exceptional cases then Theorem \ref{main_intro} is proved in \cite{image}, \cite{2adicANT} essentially by first computing $\EB$ and then its centre $\ZB$. Moreover, it is proved there that  \eqref{main_point} is an isomorphism and  $R^{\ps, \zeta\varepsilon}_{\tr \rhobar}\br{\gal}/J$ is $\OO$-torsion free.  The argument in this paper, sketched in section \ref{sketch} below,  is different: we do not compute $\EB$. 

Our original strategy for proving \eqref{intro_invertp} in this paper was to use  an argument of Gabber in Appendix to \cite{gabber}. We showed that  $R^{\ps, \zeta\varepsilon}_{\tr \rhobar}[1/p]$ is normal, $\ZB[1/p]$ is reduced, and \eqref{main_point} induces a bijection on maximal spectra $\mSpec \ZB[1/p]\rightarrow \mSpec R^{\ps, \zeta\varepsilon}_{\tr \rhobar}[1/p]$ and 
isomorphism of the residue fields. However, this is replaced by a different argument in the final version, which also proves the first part of Theorem \ref{intro_main2}. One important ingredient in the proof are results of Colmez, Dospinescu and VP \cite{CDP}, \cite{CDP2}, which imply that the universal framed deformation of $\rhobar$ with determinant $\zeta \varepsilon$ lies in the image of $\cV$.  We show in the appendix \ref{appendix} that  $(R^{\ps, \zeta\varepsilon}_{\tr\rhobar}\br{\gal}/J)_{\tf}$ acts faithfully on this representation using the theory of Cayley--Hamilton algebras, \cite{Pro87}.  

The normality of the ring $R^{\ps, \zeta\varepsilon}_{\tr \rhobar}[1/p]$ is proved in the appendix, where we show if the generic fibre of the framed deformation ring $R^{\square}_{\rhobar}[1/p]$ is normal then $R^{\ps, \zeta\varepsilon}_{\tr \rhobar}[1/p]$ and the corresponding rigid analytic space $(\Spf R^{\ps, \zeta\varepsilon}_{\tr \rhobar})^{\rig}$ are normal. The same applies to the deformation rings without the fixed determinant condition. In fact, we prove this statement not only for $\gal$,  but for any pro-finite group satisfying Mazur's
finiteness condition at $p$. We then show that $R^{\square}_{\rhobar}[1/p]$ is normal for all $2$-dimensional semi-simple representations of $\gal$; the hard cases are precisely those corresponding to the exceptional blocks, if $p=2$ then the assertion has been shown in \cite{CDP2}, if $p=3$ then we give a proof in the appendix based on B\"ockle's explicit description of the framed deformation ring in \cite{boeckle}. We note that the argument of \cite{CDP2} has been generalized by Iyengar \cite{Iy} showing that when $\rhobar$ is the trivial $d$-dimensional 
representation of the absolute Galois group of a $p$-adic field, containing $4$-th root of unity if $p=2$, then $R^{\square}_{\rhobar}[1/p]$ is normal, so our result also apply in this setting. We expect\footnote{This has now been proved in \cite[Corollary 4.22]{BIP}.} the rings $R^{\square}_{\rhobar}[1/p]$ to be normal for any $d$-dimensional representation $\rhobar$ of $\GG_F$, where $F$ is a finite extension of $\Qp$.

\subsection{A sketch of the proof}\label{sketch} 
We will now explain the construction of the map \eqref{main_point}. To fix ideas we will discuss a special case $p=2$,  $\rhobar=\Eins \oplus \omega$ and $\zeta=\Eins$. Since the cyclotomic character is trivial modulo $2$, $\rhobar=\Eins \oplus \Eins$. The corresponding block has $2$ irreducible representations: the trivial  $\Eins$ and the smooth Steinberg representation $\Sp$. Instead of working with representations on $\OO$-torsion modules it is more convenient to use Pontryagin duality and work with representations of $G$ on compact $\OO$-modules. We denote by $\dualcat(\OO)_{\BB}$ the category  anti-equivalent to $\Mod^{\lfin}_{G, \zeta}(\OO)_{\BB}$ under the Pontryagin duality.

Let $P_{\Eins^{\vee}}$ and $P_{\Sp^{\vee}}$ be projective envelopes of $\Eins^{\vee}$ and $\Sp^{\vee}$ in $\dualcat(\OO)_{\BB}$, respectively. Then $P_{\BB}:= P_{\Eins^{\vee}}\oplus P_{\Sp^{\vee}}$ is a projective generator of $\dualcat(\OO)_{\BB}$ and by results of Gabriel \cite{MR232821} the category $\dualcat(\OO)_{\BB}$ is equivalent to the category of pseudo-compact modules of $E_{\BB}:=\End_{\dualcat(\OO)}(\PB)$. The equivalence is induced by the functors
$$N\mapsto \Hom_{\dualcat(\OO)}(\PB, N),\quad  \md\mapsto \md \wtimes_{\EB} \PB.$$
The centre $\ZB$ of the category $\dualcat(\OO)_{\BB}$ is naturally isomorphic to the centre of the ring $\EB$. 

Colmez's functor kills all objects on which $\SL_2(\Qp)$ acts trivially, and these form a thick subcategory. Thus $\cV$ factors through the quotient category, which we denote by  $\mathfrak Q(\OO)_{\BB}$ and let $\mathcal T: \dualcat(\OO)_{\BB}\rightarrow \mathfrak Q(\OO)_{\BB}$ be the quotient functor. Moreover, $\cV$ induces an equivalence of categories between $\mathfrak Q(\OO)_{\BB}$ and its essential image under $\cV$. To prove this one needs to show that $\cV$ sends non-split extensions to non-split extensions, such arguments are originally due to Colmez, in the case $p=2$ this has been shown by the second author in his thesis, \cite{Tung2020}.

Since $\mathcal T(\Eins^{\vee})=0$, $\mathcal T(\Sp^{\vee})$ is the only irreducible object in $\mathfrak Q(\OO)_{\BB}$ up to isomorphism. Moreover, it is shown in \cite{image} that $\mathcal T(P_{\Sp^{\vee}})$ is a projective envelope of $\mathcal T(\Sp^{\vee})$, and $\mathcal T$ induces an isomorphism
$$ \EB':=\End_{\dualcat(\OO)}(P_{\Sp^{\vee}})\cong\End_{\mathfrak Q(\OO)}(\mathcal TP_{\Sp^{\vee}}).$$ 
Since $\mathcal T(\Sp^{\vee})$ is the only irreducible in $\mathfrak Q(\OO)_{\BB}$, $\mathcal T(P_{\Sp^{\vee}})$ is a projective generator of $\mathfrak Q(\OO)_{\BB}$, and thus $\mathfrak Q(\OO)_{\BB}$ is equivalent to the category of pseudo-compact $\EB'$-modules. This implies that $\mathcal T(P_{\Sp^{\vee}})$ and hence, by equivalence of categories, $\cV(P_{\Sp^{\vee}})$ are flat over $\EB'$. Since $\cV(\Sp^{\vee})$ is a one dimensional representation of $\gal$, in fact the trivial representation with our normalizations, an application of Nakayama's lemma shows that $\cV(P_{\Sp^{\vee}})$ is a free $\EB'$-module of rank $1$. The action of $\gal$ on $\cV(P_{\Sp^{\vee}})$ commutes with the action of $\EB'$ and so induces a homomorphism 
$$\alpha: \OO\br{\gal} \rightarrow \End_{\EB'}(\cV(P_{\Sp^{\vee}}))\cong (\EB')^{\op}.$$ 
We show that this map is surjective. In general the argument is carried out in section \ref{endo} in an abstract setting and then in Proposition \ref{prop:surjgalE} we verify that the conditions of the abstract setting are satisfied. However, in the special case under the consideration  the argument is easier: since $\cV$ induces an equivalence of categories between $\mathfrak Q(\OO)_{\BB}$
and its essential image, the $\EB'$-cosocle and $\OO\br{\gal}$-cosocle of $\cV(P_{\Sp^{\vee}})$ coincide. This implies that there is $v\in\cV(P_{\Sp^{\vee}})$, which is a generator of $\cV(P_{\Sp^{\vee}})$ both as $\EB'$- and as $\OO\br{\gal}$-module. This implies that $\alpha$ is surjective.

We also show in section \ref{section:pseudorepresentations}  that the natural map
$$\beta: \OO\br{\gal}\rightarrow R^{\ps, \zeta\varepsilon}_{\tr \rhobar}\br{\gal}/J$$ 
is surjective, where $R^{\ps, \zeta\varepsilon}_{\tr \rhobar}\br{\gal}/J$ is the largest quotient of $R^{\ps, \zeta\varepsilon}_{\tr \rhobar}\br{\gal}$ such that the Cayley--Hamilton theorem holds for the universal pseudorepresentation with determinant $\zeta\varepsilon$ lifting $\tr \rhobar$. 

The idea is to show that $\Ker \alpha$ contains $\Ker \beta$, since this implies that the action of $\gal$ on $\cV(P_{\Sp^{\vee}})$ induces surjections: 
\begin{equation}\label{intro_surj}
\OO\br{\gal}\twoheadrightarrow R^{\ps, \zeta\varepsilon}_{\tr \rhobar}\br{\gal}/J\twoheadrightarrow \EB'.
\end{equation}
This is proved using the results of Berger--Breuil \cite{MR2642406} on universal unitary completions of locally algebraic principal series and density arguments already used in \cite{CDP} and also in \cite{2018arXiv180307451T}. Morally, the argument should be that $\cV(P_{\Sp^{\vee}})$ injects into the product of all $2$-dimensional crystabelline representations of $\gal$ with mod $p$ reduction isomorphic to $\rhobar$ and determinant $\zeta \varepsilon$, then an element in $\Ker \beta$ would kill this  product and hence $\cV(P_{\Sp^{\vee}})$. In reality the argument is technically a bit more complicated: we also have to consider deformations of such representations to local artinian $L$-algebras, see section \ref{density}.

Wang-Erickson has proved in \cite{WE_alg} that $R^{\ps, \zeta\varepsilon}_{\tr \rhobar}\br{\gal}/J$ is a finitely generated $R^{\ps, \zeta\varepsilon}_{\tr \rhobar}$-module. The surjection \eqref{intro_surj} implies that the image of $R^{\ps, \zeta\varepsilon}_{\tr \rhobar}$ is contained in the centre of $\EB'$, which we denote by $\ZB'$. Moreover, both $\EB'$ and $\ZB'$ are finite $R^{\ps, \zeta\varepsilon}_{\tr \rhobar}$-modules. We show that \eqref{intro_surj} induces an isomorphism 
$$ (R^{\ps, \zeta\varepsilon}_{\tr \rhobar}\br{\gal}/J)_{\tf}\overset{\cong}{\longrightarrow} \EB'$$ 
by showing that  the universal framed deformation of $\rhobar$ with determinant $\zeta \varepsilon$ lies in the image of $\cV$ using \cite{CDP}, \cite{CDP2} and $(R^{\ps, \zeta\varepsilon}_{\tr \rhobar}\br{\gal}/J)_{\tf}$ acts faithfully on it, see Proposition \ref{Ztf}. The assertions about the centre in Theorem \ref{intro_main2} with $\ZB'$ instead of $\ZB$ are proved by studying the centre of $(R^{\ps, \zeta\varepsilon}_{\tr \rhobar}\br{\gal}/J)_{\tf}$. This argument is carried out in the appendix for $d$-dimensional representations of any profinite group, satisfying Mazur's finiteness condition at $p$.

We then transfer this result from $\EB'$ and $\ZB'$ to $\EB$ and $\ZB$. Let $M_{\BB}$ be the kernel of $P_{\BB}\rightarrow (P_{\BB})_{\SL_2(\Qp)}$. We show that 
$$\End_{\dualcat(\OO)}(M_{\BB})\cong \End_{\dualcat(\OO)}(P_{\BB})=E_{\BB}$$ 
by examining various exact sequences and showing that certain $\Ext$-groups vanish. We also show that the cosocle of $M_{\BB}$ is a direct sum of finitely many copies of $\Sp^{\vee}$. Thus $M_{\BB}$ is a quotient of $(\PB')^{\oplus n}$ for some $n\ge 1$. This allows to conclude that $\End_{\dualcat(\OO)}(M_{\BB})$ and its centre are finitely generated $\ZB'$-modules, which finishes the proof of Theorem \ref{main_intro}. The arguments showing finiteness of $\EB$ and $\ZB$ over $\ZB'$ are carried out in section \ref{centre_dash}. Then with some more effort we are able to show that $\ZB'=\ZB$, see Corollary \ref{Zequal}.

\subsection{What is left to do?} Although we believe that our results will suffice for most number theoretic applications, for example \cite{luepan}, to complete the programme started in \cite{image} one would have to compute the ring $\EB$ in the exceptional cases. This will be harder than \cite[Section 10.5]{image}, which is already quite involved. We expect that the map in \eqref{intro_surj} induces  isomorphisms 
$$R^{\ps, \zeta\varepsilon}_{\tr \rhobar}\br{\gal}/J\overset{\cong}{\longrightarrow} \EB', \quad R^{\ps, \zeta\varepsilon}_{\tr \rhobar}\overset{\cong}{\longrightarrow} \ZB$$ 
This is known to hold for all blocks except for the exceptional ones.  Theorem \ref{intro_main2} implies that to prove this result it would be enough to show that $R^{\ps, \zeta\varepsilon}_{\tr \rhobar}\br{\gal}/J$ is $\OO$-torsion free and for the second isomorphism in the case $p=3$ it would be enough to show\footnote{This follows from \cite[Corollary 5.11]{BIP}.} that $R^{\ps, \zeta\varepsilon}_{\tr \rhobar}$ is $\OO$-torsion free. 

It seems likely that using the results of this paper one can remove the restriction on the prime $p$ in Lue Pan's  work \cite{luepan2} on the Fontaine--Mazur conjecture in the residually reducible case, which generalizes the work of Skinner--Wiles. We hope to return to these questions in future work. 

\subsection{Acknowledgements.} VP would like to thank Toby Gee for organizing a seminar on Lue Pan's paper \cite{luepan} and for the invitation to participate. Lue Pan's paper 
was  a major source of motivation to deal with the exceptional cases.  VP would like to thank Patrick Allen and Gebhard B\"ockle for stimulating correspondence. Parts of the paper were written, when SNT was visiting Academia Sinica. SNT would like to thank Carl Wang-Erickson for stimulating correspondence. 
The authors would like to thank Gabriel Dospinescu and Toby Gee for their detailed comments on an
earlier draft. 

\section{Endomorphism rings}\label{endo}

Let $E$ be a pseudo-compact $\OO$-algebra and let $\PC(E)$ be the category of left pseudo-compact $E$-modules, see \cite{brumer}, \cite[\S IV.3]{MR232821}. Let $\Irr(E)$ be the set of equivalence classes of irreducible objects in $\PC(E)$.
 
Let $M$ be in $\PC(E)$. We assume that we are given a continuous $E$-linear action of a profinite group  $\GG$ on $M$, which makes $M$ into a pseudo-compact module over the completed group algebra $\OO\br{\GG}$. The action induces a homomorphism of $\OO$-algebras $\OO\br{\GG}\rightarrow \End_E^{\cont}(M)$. In this section we will study, when this map is surjective and its kernel. 

If $N$ is a pseudo-compact $E$-module, which is finitely generated as an $E$-module, then we may present $N$ as 
$$ \prod_{i\in I} E \rightarrow E^{\oplus n} \rightarrow N \rightarrow 0.$$ 
By applying $\Hom^{\cont}_E(\ast, M)$ we obtain an exact sequence 
$$ 0\rightarrow \Hom^{\cont}_E(N, M)\rightarrow M^{\oplus n} \rightarrow \oplus_{i\in I} M.$$ 
We thus may identify $\Hom^{\cont}_E(N, M)$  with a  closed submodule of $M^{\oplus n}$, which makes $\Hom^{\cont}_E(N, M)$
into a pseudo-compact left $\OO\br{\GG}$-module. 

\begin{lem}\label{confused} Let $N$ be a finitely generated projective  $E$-module and let $\md$ be a right pseudo-compact $\OO\br{\GG}$-module. Then the natural 
map 
\begin{equation}\label{one}
\md \wtimes_{\OO\br{\GG}} \Hom^{\cont}_E(N, M) \rightarrow \Hom^{\cont}_E( N, \md \wtimes_{\OO\br{\GG}} M)
\end{equation}
is an isomorphism. 
\end{lem}

\begin{proof} Since $N$ is finitely generated and projective, we may present it as
$$ F\overset{e}{\longrightarrow} F \rightarrow N \rightarrow 0,$$ 
where $F\cong E^{\oplus n}$ and $e$ is an idempotent. In particular, $N$ is a pseudo-compact $E$-module. The map in \eqref{one} is induced by a continuous bilinear map 
$$(v, \phi)\mapsto [ w\mapsto v \wtimes \phi(w)].$$ 
It is an isomorphism if $N=F$. The general case follows by applying the idempotent $e$ to the isomorphism obtained for $N=F$. 
\end{proof}

\begin{lem}\label{double} Let $\{\rho_i\}_{i\in I}$ be a family of pairwise distinct absolutely irreducible right pseudo-compact $\OO\br{\GG}$-modules. Then the map 
\begin{equation}\label{prod}
\OO\br{\GG}\rightarrow \prod_{i\in I} \End_k(\rho_i)^{\op}
\end{equation}
is surjective. 
\end{lem}

\begin{proof} Since $\GG$ is profinite each $\rho_i$ is a finite dimensional $k$-vector space. Thus $\varphi_i: \OO\br{\GG} \rightarrow \End_{k}(\rho_i)^{\op}$, given by the action, is continuous for the discrete topology on the target. Since $\rho_i$ is absolutely irreducible $\varphi_i$ is surjective. Moreover, $\Ker \varphi_i$ is an open maximal two-sided ideal of  $\OO\br{\GG}$. If $i\neq j$ then $\rho_i\not \cong \rho_j$ and thus $\Ker \varphi_i +\Ker \varphi_j= \OO\br{\GG}$. This implies that for every finite subset $F$ of $I$ the map $\OO\br{\GG}\rightarrow \prod_{i\in F} \End_k(\rho_i)^{\op}$ is surjective.  Thus the image of \eqref{prod} is dense for the product topology on the target. On the other hand \eqref{prod} is a continuous map between pseudo-compact $\OO$-modules and thus its image is closed, which implies surjectivity.
\end{proof} 

If $M$ is in $\PC(E)$ then we let $\rM$ be the intersection of open 
maximal submodules of $M$. Then $\rE$ is a closed two-sided ideal of $E$ and 
$\rM$ is the closure of $\rE M$ inside $M$. 

\begin{prop}\label{surjective} Let us assume that the following hold
\begin{enumerate}
\item $M$ is a finitely generated projective $E$-module;
\item $M/\rM = \cosoc_{\GG} M$;
\item for all  $S\in \Irr(E)$, such that 
$$\rho_S:=\Hom_{E}^{\cont}(M, S)\neq 0,$$ 
$\dim_k \rho_S$ is finite and $\rho_S$ is an absolutely irreducible  representation of $\GG$; 
\item if $S, S'\in \Irr(E)$ and $S\not\cong S'$   then $\Hom_{\GG}(\rho_S, \rho_{S'})=0$.
\end{enumerate}
Then the map $\OO\br{\GG}\rightarrow \End_{E}^{\cont}(M)$ is surjective. 
\end{prop}

\begin{proof} For $S\in \Irr(E)$ we let $\overline{M}_S$ be the smallest quotient of $M$ such that 
$$\rho_S=\Hom_E^{\cont}(M, S)=\Hom_E^{\cont}(\overline{M}_S, S).$$
It follows from  (3) that the above subspaces are finite dimensional. Thus $\overline{M}_S\cong S^{\oplus d}$ such that $d \cdot \dim_k \End_E^{\cont}(S)= \dim_k \rho_S$. Since $S$ is irreducible $\End_E^{\cont}(S)$ is a skew field. It acts on $\rho_S$ and this action commutes with the action of $\GG$. Since $\rho_S$ is absolutely irreducible we conclude that $\End_E^{\cont}(S)=k$ and $\dim_k \rho_S=d$. Thus $\End^{\cont}_E( \overline{M}_S)\cong M_d(k)$ and thus does not have non-trivial two-sided ideals. Hence, the natural right action of $\End_E( \overline{M}_S)$ on $\rho_S$ induces an injective ring homomorphism 
\begin{equation}\label{iso}
\End_E^{\cont}( \overline{M}_S) \rightarrow  \End_k(\rho_S)^{\op},
\end{equation}
which has to be surjective as both $k$-vector spaces have dimension equal to $d^2$.

The isomorphism $M/\rM \cong \prod_{S\in \Irr(E)} \overline{M}_S$ induces an isomorphism 
$$ \End_E^{\cont}( M /\rM) \cong \prod_{S\in \Irr(E)} \End_E^{\cont}(\overline{M}_S).$$
Since $\rho_S \not \cong \rho_{S'}$ if $S\not \cong S'$,  it follows from Lemma \ref{double} together with the isomorphism \eqref{iso} that the action of $\OO\br{\GG}$ on  $M /\rM$ induces a 
surjection 
\begin{equation}\label{surj1}
\OO\br{\GG} \twoheadrightarrow \End_E^{\cont}(M /\rM).
\end{equation}
Since $M$ is projective and $M/\rM$ is pro-semisimple we have 
\begin{equation}\label{surj2}
 \End^{\cont}_E(M) \twoheadrightarrow \Hom^{\cont}_E(M, M/\rM)\cong \End_E^{\cont}( M/\rM).
\end{equation} 
If $\md$ is an irreducible right pseudo-compact $\OO\br{\GG}$-module and $\mathfrak a$ is its annihilator then $\OO\br{\GG}/\mathfrak a$ is a finite dimensional simple $k$-algebra, and thus $\OO\br{\GG}/\mathfrak a$ is semi-simple as a left $\OO\br{\GG}$-module, and thus $(\OO\br{\GG}/\mathfrak a) \wtimes_{\OO\br{\GG}} M$ is semi-simple as a left $\OO\br{\GG}$-module. Hence, the surjection $M\twoheadrightarrow  (\OO\br{\GG}/\mathfrak a) \wtimes_{\OO\br{\GG}} M$ factors through as
$$ M\twoheadrightarrow \cosoc_{\GG} M \twoheadrightarrow (\OO\br{\GG}/\mathfrak a) \wtimes_{\OO\br{\GG}} M.$$ 
Moreover, the maps become isomorphisms after applying $\md\wtimes_{\OO\br{\GG}}$. Thus 
\begin{equation}\label{dazed}
\md \wtimes_{\OO\br{\GG}} M \cong \md \wtimes_{\OO\br{\GG}} \cosoc_{\GG} M\cong  \md \wtimes_{\OO\br{\GG}} M/\rM,
\end{equation} 
as $M/\rM= \cosoc_{\GG} M$ by assumption. Lemma \ref{confused} together with \eqref{dazed} imply that by applying $\md \wtimes_{\OO\br{\GG}}$ to \eqref{surj2} we obtain isomorphisms 
\begin{equation}\label{surj3} \begin{split} 
\md\wtimes_{\OO\br{\GG}} \End^{\cont}_E(M)&\overset{\cong}{\longrightarrow}
\Hom^{\cont}_E(M, \md\wtimes_{\OO\br{\GG}} M)\\
&\overset{\cong}{\longrightarrow}
\Hom^{\cont}_E(M, \md\wtimes_{\OO\br{\GG}} M/\rM)\\
&\overset{\cong}{\longrightarrow}
\md\wtimes_{\OO\br{\GG}} \Hom^{\cont}_E(M, M/\rM)\\
&\overset{\cong}{\longrightarrow}
\md\wtimes_{\OO\br{\GG}} \End^{\cont}_E(M/\rM),
\end{split}
\end{equation}
where the $\GG$-action on $\End^{\cont}_E(M)$ and $\End^{\cont}_E(M/\rM)$ is given by $(g.\varphi)(v):= g (\varphi(v))$. 

If the map $\OO\br{\GG}\rightarrow \End^{\cont}_{E}(M)$ is not surjective then its cokernel is a non-zero left pseudo-compact $\OO\br{\GG}$-module and thus 
will have an irreducible quotient $\md'$. If we let $\md= \Hom_k(\md', k)$ with the right $\OO\br{\GG}$-action then $\md\wtimes_{\OO\br{\GG}} \md'$ is non-zero, as the evaluation map $\md\wtimes_{\OO\br{\GG}} \md'\twoheadrightarrow k$ is non-zero.  By construction the composition 
$$ \md\wtimes_{\OO\br{\GG}} \OO\br{\GG}\rightarrow \md\wtimes_{\OO\br{\GG}} \End^{\cont}_E(M)\twoheadrightarrow \md \wtimes_{\OO\br{\GG}} \md' $$ 
is the zero map. Thus  $\md\wtimes_{\OO\br{\GG}} \OO\br{\GG}\rightarrow \md\wtimes_{\OO\br{\GG}} \End^{\cont}_E(M)$ cannot be surjective. However, the commutative diagram 
\[
  \begin{tikzcd}
   \md\wtimes_{\OO\br{\GG}} \OO\br{\GG}\arrow[r, rightarrow] \arrow[d, "="] &  \md\wtimes_{\OO\br{\GG}} \End^{\cont}_E(M)\arrow[d, "\cong", "\eqref{surj3}"']\\
 \md\wtimes_{\OO\br{\GG}} \OO\br{\GG}\arrow[r, twoheadrightarrow, "\eqref{surj1}" ] & \md\wtimes_{\OO\br{\GG}} \End^{\cont}_E(M/\rM)
  \end{tikzcd}
\]
implies that the top horizontal arrow is surjective, yielding a contradiction.
\end{proof}

We remind the reader that as a consequence of the topological Nakayama's lemma, see Lemma 0.3.3 in Expos\'e $VII_B$ in  \cite{MR2867622}, the following holds: 

\begin{lem} \label{lem:Eproptest}
Let $N$ be a pseudo-compact left $E$-module. Then $N$ is projective in $\PC(E)$ if and only if the functor $\md \mapsto \md\wtimes_E N$ from 
the category of right pseudo-compact $E$-modules to the category of abelian groups is exact. In this case, $N\twoheadrightarrow N/\mathfrak r(N)$ 
is a projective envelope of $N/\mathfrak r(N)$. 
\end{lem} 

\begin{cor} \label{cor:autosurj}
If in addition to the assumptions of Proposition \ref{surjective} we assume that $M/\rM \cong E/\rE$ as $E$-modules
then $M$ is a free $E$-module of rank $1$ and the action of $\OO\br{\GG}$ on $M$ induces a surjection 
$$\OO\br{\GG}\twoheadrightarrow E^{\op},$$
which is uniquely determined up to a conjugation by $E^{\times}$. 
\end{cor}

We will now give a characterisation of the kernel of $\OO\br{\GG}\rightarrow \End_E^{\cont}(M)$ 
in favourable settings. 

Let $\md$ be a finite dimensional $L$-vector space with continuous $\OO$-linear action of $E$ on the right. The image of $E$ in $\End_L(\md)$ is a compact $\OO$-module and thus $E$ stabilises an $\OO$-lattice $\md^0$ in $\md$. The action of $\OO\br{\GG}$ on $M$ induces a continuous left action of $\OO\br{\GG}$ on $\md^0\wtimes_E M$ and hence on 
$$\md \otimes_E M= (\md^0\otimes_E M)[1/p]= (\md^0\wtimes_E M)[1/p].$$

\begin{lem} \label{lem:kergalend}
Let $\{\md_i\}_{i\in I}$ be a family of  finite dimensional $L$-vector spaces with continuous right $\OO$-linear action of $E$. For each $i\in I$, let $\mathfrak a_i$ be the $E$-annihilator of $\md_i$ and let $\mathfrak b_i$ be the $\OO\br{\GG}$-annihilator of $\md_i \otimes_E M$. If $M$ is a free $E$-module of finite rank and $\bigcap_{i\in I} \mathfrak a_i=0$ then 
$$ \Ker(\OO\br{\GG}\rightarrow \End_E^{\cont}(M))= \bigcap_{i\in I} \mathfrak b_i.$$
\end{lem}

\begin{proof} 
For each $i\in I$, let $\mathfrak c_i$ be the $\End_E^{\cont}(M)$-annihilator of $\md_i \otimes_E M$. Since $\mathfrak b_i$ is the preimage of $\mathfrak c_i$ in $\OO\br{\GG}$, it is enough to show that $\bigcap_{i\in I} \mathfrak c_i=0$.

Let $w_1, \ldots, w_n$ be an $E$-basis of $M$. Then we may identify 
$\End_E^{\cont}(M)$ with $M_n(E)$ by mapping $\varphi$ to the matrix $(a_{kj})$, given by 
$$ \varphi(w_k)= \sum_{j=1}^n a_{kj} w_j$$ 
for all $1\le k\le n$. If $v\in \md_i$ then 
$$v \wtimes \varphi(w_k)= \sum_{j=1}^n ( v a_{kj})\wtimes w_j.$$ 
Thus $\varphi$ annihilates $\md_i\otimes_E M$ if and only if $v a_{kj}=0$ for all $v\in \md_i$ and all $1\le k, j\le n$, which is equivalent to $\varphi \in M_n(\mathfrak a_i)$. Since $\bigcap_{i\in I} \mathfrak a_i=0$ we have $\bigcap_{i\in I} M_n(\mathfrak a_i)=0$ and thus $\bigcap_{i\in I} \mathfrak c_i=0$.
\end{proof}

\section{Pseudorepresentations} \label{section:pseudorepresentations} 

Let $\GG$ be a profinite group and let $\rhobar$ be a continuous semi-simple representation of $\GG$ on a $2$-dimensional $k$-vector space. 
We fix a continuous  group homomorphism $\psi: \GG\rightarrow \OO^{\times}$ lifting $\det \rhobar$. Let $D^{\ps, \psi}$ be the functor from the category of augmented artinian $\OO$-algebras with residue field $k$ to the category of sets, which maps $(A, \mm_A)$ to the set of continuous functions $t: \GG \rightarrow A$ satisfying the following conditions:
\begin{itemize} 
\item $t(1)=2$;
\item $t(g) \equiv \tr \rhobar(g)\pmod{\mm_A}, \quad \forall g\in \GG$;
\item $t(gh)=t(hg), \quad \forall g, h\in \GG$;
\item $\psi(g) t(g^{-1} h)-t(g)t(h) +t(gh)=0, \quad \forall g, h\in \GG$.
\end{itemize}
The data $(t, \psi)$ determines a continuous polynomial law of homogeneous degree $2$ on $\GG$, see \cite[Lemma 1.9]{che_durham}. This deformation problem is pro-representable by a local $\OO$-algebra $R^{\ps, \psi}$ with residue field $k$, complete with respect to profinite topology. We denote by $T: \GG\rightarrow R^{\ps, \psi}$ the universal deformation. We extend it $R^{\ps, \psi}$-linearly to a continuous function $T:R^{\ps, \psi}\br{\GG}\rightarrow R^{\ps, \psi}$. The homomorphism $\psi: \GG \rightarrow \OO^{\times}$ induces a continuous $\OO$-algebra homomorphism $\psi:\OO\br{\GG}\rightarrow \OO$, which we extend $R^{\ps, \psi}$-linearly to a continuous $R^{\ps, \psi}$-algebra homomorphism 
$\psi: R^{\ps, \psi}\br{\GG}\rightarrow R^{\ps, \psi}$. Let $J$ be the closed two-sided ideal of $R^{\ps, \psi}\br{\GG}$ generated by $a^2- T(a) a + \psi(a)$ for all $a\in R^{\ps, \psi}\br{\GG}$. 

\begin{prop} \label{prop:galsurj}
The ring homomorphism $\OO\br{\GG}\rightarrow R^{\ps, \psi}\br{\GG}/ J$ is surjective.
\end{prop}

\begin{proof} 
Let $\overline{R}$ and $C$ be the images of $R^{\ps, \psi}$ and $\OO\br{\GG}$ in $R^{\ps, \psi}\br{\GG}/ J$, respectively. Since $R^{\ps, \psi}$,  $\OO\br{\GG}$ and $R^{\ps, \psi}\br{\GG}/ J$ are pseudo-compact $\OO$-modules $\overline{R}$ and $C$ are closed subrings of $R^{\ps, \psi}\br{\GG}/ J$.
It is enough to show that $C$ contains $\overline{R}$, since in this case we deduce that $C$ contains the image of $R^{\ps, \psi}\br{\GG}$, which is equal to $R^{\ps, \psi}\br{\GG}/ J$.

Let $B= \overline{R}\cap C$ and let $\mm_B$ be the intersection of $B$ with the maximal ideal of $\overline{R}$. Then $B$ is a closed subring of $\overline{R}$. This implies that $B$ is complete for the profinite topology. If  $x\in \mm_B$ then $1+x$ has an inverse in $B$ given by the geometric series. Since $B$ is an $\OO$-algebra and the residue field of $\overline{R}$ is $k$, we conclude that $(B, \mm_B)$ is a local ring with residue field $k$. 

Let $\overline{T}$ be the specialisation of $T$ along $R^{\ps, \psi}\twoheadrightarrow \overline{R}$. The relation $g^2 - \overline{T}(g) g +\psi(g)=0$ in $R^{\ps, \psi}\br{\GG}/ J$ implies that $\overline{T}(g)= g + g^{-1} \psi(g)$. Thus $\overline{T}$ takes values in $B$. The universal property of $R^{\ps, \psi}$ implies that there is a continuous homomorphism of $\OO$-algebras $\varphi: R^{\ps, \psi} \rightarrow B$, such that $\varphi(T(g))= \overline{T}(g)$ for all $g\in \GG$. Using the universal property of $R^{\ps, \psi}$ again, we conclude that if we compose $\varphi$ with the inclusion $B \subset \overline{R}$ we get back the surjection $R^{\ps, \psi}\twoheadrightarrow \overline{R}$ that we have started with. Thus $B=\overline{R}$.
 \end{proof}
 
\begin{remar} 
The proposition does not hold if we do not fix the determinant or consider representations $\rhobar$ of dimension bigger than $2$. Counterexamples may be obtained with $\GG=\Zp$ and $\rhobar$ trivial representation of $\GG$ on an $n$-dimensional $k$-vector space, using \cite[Ex.\,1.7(i), 1.11(i)]{che_durham}.
\end{remar}

\section{Representations of $\GL_2(\Qp)$}

Let $G$ be a $p$-adic analytic group and let $Z$ be its center. We let $\msm_G(\OO)$ be the category of smooth representations of $G$ on $\OO$-torsion modules. Pontryagin duality, $\pi \mapsto \pi^{\vee}:=\Hom_{\OO}(\pi, L/\OO)$ equipped with the compact open topology, induces an antiequivalence of categories between $\msm_G(\OO)$ and the category $\mpro_G(\OO)$ of linearly compact $\OO$-modules with a continuous $G$-action \cite[Lemma 2.2.7]{MR2667882}. 
The inverse is given by $M \mapsto M^{\vee}:=\Hom^{\cont}_{\OO}(M, L/\OO)$. 
In particular, if $G$ is compact then $\mpro_G(\OO)$ is the category of linearly compact $\OO\br{G}$-modules, where $\OO\br{G}$ is the completed group algebra. We define $\msm_G(k)$ and $\mpro_G(k)$ the same way with $\OO$ replaced by $k$. Moreover for a continuous character $\zeta : Z \rightarrow \OO^\times$, adding the subscript $\zeta$ in any of the above categories indicates the corresponding full subcategory of $G$-representations on which $Z$ acts by $\zeta$. Denote $\mlf_{G, \zeta}(\OO)$ the full subcategory of $\msm_{G, \zeta}(\OO)$ consisting of representations in $\msm_{G, \zeta}(\OO)$ which are equal to the union of their subrepresentations of finite length. We let $\CC(\OO)$ be the full subcategory of $\mpro_G(\OO)$ antiequivalent to $\mlf_{G, \zeta}(\OO)$ under the Pontryagin duality.

\subsection{Blocks} \label{section:blocks}
From now on, we will assume $G = \GL_2(\Qp)$. Every irreducible object $\pi$ of $\msm_G(\OO)$ is killed by $\varpi$ and hence is an object of $\msm_G(k)$.  

Let $\Irr_{G, \zeta}$ be the set of irreducible representations in $\msm_{G, \zeta}(k)$. We write $\pi \leftrightarrow \pi'$ if $\pi \cong \pi'$ or $\Ext^1_{G, \zeta}(\pi, \pi') \neq 0$ or $\Ext^1_{G, \zeta}(\pi', \pi) \neq 0$, where $\Ext^1_{G, \zeta}(\pi, \pi')$ is the Yoneda extension group of $\pi'$ by $\pi$ in $\msm_{G, \zeta}(k)$. We write $\pi \sim \pi'$ if there exists $\pi_1, \cdots, \pi_n \in \Irr_{G, \zeta}$ such that $\pi \cong \pi_1$, $\pi' \cong \pi_n$ and $\pi_i \leftrightarrow \pi_{i+1}$ for $1 \leq i \leq n-1$. The relation $\sim$ is an equivalence relation on $\Irr_{G, \zeta}$. A block is an equivalence class of $\sim$. 

Barthel-Livn\'{e} \cite{MR1290194} and Breuil \cite{MR2018825} have classified the absolutely irreducible smooth representations $\pi$ admitting a central character. The blocks containing an absolutely irreducible representation have been determined in \cite[Corollary 6.2]{MR3444235}. There are the  following cases:
\begin{enumerate}[label=(\roman*)]
    \item $\BB = \{ \pi \}$ with $\pi$ supersingular;
    \item $\BB = \{ (\Ind_B^G\chi_1 \otimes \chi_2 \omega^{-1})_{\sm},  (\Ind_B^G\chi_2 \otimes \chi_1 \omega^{-1})_{\sm}\}$ with $\chi_2 \chi_1^{-1} \neq \Eins, \omega^{\pm 1}$;
    \item $p>2$ and $\BB = \{ (\Ind_B^G \chi \otimes \chi \omega^{-1})_{\sm} \}$;
    \item $p \geq 5$ and $\BB = \{ \Eins, \Sp, (\Ind^G_B \omega \otimes \omega^{-1})_{\sm} \} \otimes \chi \circ \det$;
    \item $p=3$ and $\BB = \{ \Eins, \Sp, \omega \circ \det, \Sp \otimes \omega \circ \det \} \otimes \chi \circ \det$;
    \item $p=2$ and $\BB = \{ \Eins, \Sp \} \otimes \chi \circ \det$;
\end{enumerate}
where $\chi$, $\chi_1$, $\chi_2: \Qp^{\times} \rightarrow k^{\times}$ are smooth characters and $\omega: \Qp^{\times} \rightarrow k^{\times}$ is the character $\omega(x) = x |x| \pmod{\unif}$, 
and $\Sp$ is the Steinberg representation defined by the exact sequence
\[
0 \rightarrow \Eins \rightarrow (\Ind_B^G \Eins)_{\sm} \rightarrow \Sp \rightarrow 0.
\]
If $\pi\in \Irr_{G, \zeta}$ is not absolutely irreducible then there is a finite extension $k'$ of $k$, such that $\pi\otimes_k k'$ is a finite direct sum of absolutely irreducible representations, see \cite[Proposition 5.11]{image}, so no information is lost by working with absolutely irreducible representations.

Given a block $\BB$, we denote by $\pi_{\BB}$ the direct sum of all representations in $\BB$ and let $\PB$ be a projective envelope of $\pi_{\BB}^{\vee}$ in $\CC(\OO)$. Then $\EB = \End_{\CC(\OO)}(\PB)$ is a pseudo-compact $\OO$-algebra. We denote  the centre of $\EB$ by $\ZB$. 

By \cite[Corollary 5.35]{image}, the category $\CC(\OO)$ decomposes into a direct product of subcategories
\begin{equation} \label{equa:blockdecomp}
    \CC(\OO) \cong \prod_{\BB \in \Irr_{G, \zeta} / \sim} \CC(\OO)_{\BB},
\end{equation}
where the objects of $\CC(\OO)_{\BB}$ are those $M$ in $\CC(\OO)$ such that for every irreducible subquotient $S$ of $M$, $S^{\vee}$ lies in $\BB$. Moreover, the category $\CC(\OO)_{\BB}$ is equivalent to the category of compact right $\EB$-modules and the centre of $\CC(\OO)_{\BB}$ is isomorphic to $\ZB$ \cite[Proposition 5.45]{image}.

\begin{lem}\label{ZB_local} 
If $\BB$ contains an absolutely irreducible representation then $\ZB$ is a local pseudo-compact $\OO$-algebra with residue field $k$. 
\end{lem}

\begin{proof} 
If $\pi$ is absolutely irreducible then $\End_G(\pi)=k$ and thus the action of $\ZB$ on $\pi$ defines a homomorphism of $\OO$-algebras $c_{\pi}: \ZB\rightarrow k$. If $\pi, \pi'\in \BB$ are distinct and there is a non-split extension $0\rightarrow \pi \rightarrow \tau \rightarrow \pi'\rightarrow 0$ then $\End_G(\tau)=k$ and we conclude that $c_{\pi}=c_{\tau}=c_{\pi'}$. Using the transitivity property of the relation $\sim$ on $\Irr_{G, \zeta}$, we conclude that $c_{\pi}=c_{\pi'}$ for all $\pi, \pi'\in \BB$. It follows from the proof of \cite[Proposition IV.4.12]{MR232821} that the Jacobson radical of $\ZB$ consists of elements that kill all the irreducible representations. Thus $\Ker c_{\pi}$ is the maximal ideal of $\ZB$ with residue field $k$. The last assertion follows from the fact that $\ZB$ is closed in $\EB$ and \cite[Proposition IV.4.13]{MR232821}.
\end{proof}

\begin{lem} \label{lem:PSL2}
Let $P$ be projective in $\CC(\OO)$. Then $P^{\SL_2(\Qp)} = 0$.
\end{lem}

\begin{proof}
Let $J$ be the Pontryagin dual of $P$. Then $J$  is injective in $\mlf_{G, \zeta}(\OO)$ and the assertion is equivalent to $J_{\SL_2(\Qp)} = 0$. Let $N$ be the unipotent subgroup $\left(\smallmatrix 1 & \Qp \\ 0 & 1 \endsmallmatrix\right)$. Then by \cite[Proposition 3.6.2]{MR2667883} and \cite[Corollary 3.12]{MR2667892}, we have $J_{N} = 0$. Thus $J_{\SL_2(\Qp)} = 0$ and the lemma follows. Alternatively, one can deduce the statement from Proposition \ref{HMP} below. 
\end{proof}

\begin{prop}\label{HMP} 
Let $P$ and $M$ be  in $\dualcat(\OO)$, such that $P$ is projective and $M/\varpi M$ is of finite length. Then $\Hom_{\dualcat(\OO)}(M, P)=0$. 
\end{prop}

\begin{proof} 
Let $K'$ be a compact open pro-$p$ subgroup of $\SL_2(\Qp)$ such that $K' \cap Z = \{1\}$. Then $P$ is projective in $\mpro_{K'}(\OO)$ by \cite[Corollary 3.10]{MR2667892} and thus $P \cong \prod_{i \in I} \OO \llbracket K' \rrbracket$ for some index set $I$. Thus it is enough to show that $\Hom_{\OO\br{K'}}^{\cont}(M, \OO\br{K'})=0$. Topological Nakayama's lemma for compact $\OO$-modules implies that it is enough to show that $\Hom_{k\br{K'}}^{\cont}(M/\varpi M, k\br{K'})=0$. Since $M/\varpi M$ is of finite length, it is enough to show that $\Hom_{k\br{K'}}(\pi^{\vee}, k\br{K'})=0$ for every $\pi \in \BB$. This follows from \cite[Lemma 5.16]{ludwig}. Note that if $K_n= 1+ M_2(2 p^n \Zp)$ and $K_n'= K_n\cap \SL_2(\Qp)$ then $K_n= K_n'  ( Z\cap K_n)$  and so $\pi^{K_n}= \pi^{K'_n}$ as $Z\cap K_n$ acts trivially on $\pi$, and so the argument in \cite[Lemma 5.16]{ludwig} carries over to the restriction of $\pi$ to $\SL_2(\Qp)$. 
\end{proof}

We will denote by $\Ord_B$ Emerton's functor of ordinary parts, \cite{MR2667882}. 

\begin{lem} \label{lem:OrdJinj} Let $\pi \hookrightarrow J$ be an injective envelope of $\pi\in \BB$ in 
$\Mod^{\lfin}_{G, \zeta}(\OO)$. If $\pi$ is supersingular then $\Ord_B J=0$. If $\pi$ is an irreducible 
subquotient of $(\Ind_{\overline{B}}^G \chi)_{\sm}$ for some character $\chi: T\rightarrow k^{\times}$, 
where $\overline{B}$ is the subgroup of lower triangular matrices, 
 then $\Ord_B J$ is isomorphic to an injective envelope of $\chi$ in $\mlf_{T, \zeta}(\OO)$, and also in $\Mod^{\sm}_{T, \zeta}(\OO)$. 
\end{lem}

\begin{proof}
 Since $J$ is injective and $\Ord_B$ is  adjoint to  parabolic induction, which is an exact functor, 
 $\Ord_B J$ is injective in $\Mod^{\lfin}_{T, \zeta}(\OO)$. Thus it is a direct sum of injective envelopes of characters of $T$. Moreover, 
 \begin{equation}\label{ord}
 \Hom_T(\chi, \Ord_B J)\cong\Hom_G( (\Ind_{\overline{B}}^G \chi)_{\sm}, J).
 \end{equation} 
 Since $J$ is an injective envelope of $\pi$, this group is non-zero if and only if $\pi$ is  a subquotient of 
 $(\Ind_{\overline{B}}^G \chi)_{\sm}$, in which case the dimension of the spaces in \eqref{ord} is equal 
 to the multiplicity with which $\pi$ occurs as a subquotient of $(\Ind_{\overline{B}}^G \chi)_{\sm}$.
 If $\pi$ is supersingular then it does not occur as a subquotient of principal series and thus $\Ord_B J=0$. Otherwise, it follows from \cite{MR1290194} that there is a unique character $\chi$, such that 
 the multiplicity is non-zero, in which case it is equal to $1$. This proves the first assertion. 
  The same (albeit easier) proof as in \cite[Proposition 5.16]{image} implies that 
  every injective object in $\Mod^{\lfin}_{T, \zeta}(\OO)$ is also injective in $\Mod^{\sm}_{T, \zeta}(\OO)$.
\end{proof}

\begin{lem} \label{lem:Jinvfinite} Let $\pi \hookrightarrow J$ be  an injective envelope of $\pi\in \BB$ in $\Mod^{\lfin}_{G, \zeta}(\OO)$. 
If $\pi$ is not a character then $J^{\SL_2(\Qp)}=0$, otherwise 
$(J^{\SL_2(\Qp)})^{\vee}$ is nonzero and finitely generated over $\OO$.
\end{lem}

\begin{proof} If $J^{\SL_2(\Qp)}$ is non-zero then $\pi \cap J^{\SL_2(\Qp)}$ is non-zero as 
$\pi\hookrightarrow J$ is essential. Since $\pi$ is absolutely irreducible, we deduce that 
$\pi$ is a character and  $\pi$ is the $G$-socle of $J^{\SL_2(\Qp)}$.
Note that the action of $G$ on $J^{\SL_2(\Qp)}$ factors through
\[
G / Z \SL_2(\Qp) \cong \Qp^{\times} / (\Qp^{\times})^2 \cong
\begin{cases}
\ZZ / 2 \ZZ \times \ZZ / 2 \ZZ &\text{if } p>2 \\
\ZZ / 2 \ZZ \times \ZZ / 2 \ZZ \times \ZZ / 2 \ZZ &\text{if } p=2.
\end{cases}
\]
It follows that $J^{\SL_2(\Qp)}[\varpi]$ is a finite dimensional $k$-vector space. Dually this 
 implies that $(J^{\SL_2(\Qp)})^{\vee} / \varpi (J^{\SL_2(\Qp)})^{\vee}$ is finite-dimensional over $k$ and the lemma follows from Nakayama's lemma.
\end{proof}

\begin{lem} \label{lem:nullhomOrd} Let $J$ be injective in $\Mod^{\lfin}_{G, \zeta}(\OO)$ and 
$\kappa \in \Mod^{\lfin}_{T, \zeta}(\OO)$ be such that $\kappa^{\vee}$ is a finitely generated $\OO$-module. Then 
$\Hom_T(\Ord_B J, \kappa) = 0$.
\end{lem}

\begin{proof} 
The following is an analogue of Proposition \ref{HMP} in an easier setting.
Let $T_0$ to be an open  torsion-free pro-$p$ subgroup of $T\cap \SL_2(\Qp)$. Since $\Ord_B J$ is injective in $\Mod^{\sm}_{T,\zeta}(\OO)$ by Lemma \ref{lem:OrdJinj} it is injective in
 $\Mod^{\sm}_{T_0}(\OO)$, as restriction to $T_0$ is adjoint to $\cInd_{ZT_0}^{T}$, which is exact. Thus the Pontryagin dual of $\Ord_B J$ is isomorphic to $\prod_{i \in I} \OO \llbracket T_0 \rrbracket$ for some index set $I$. Hence, it suffices to show that 
$$\cHom{\OO \llbracket T_0 \rrbracket}(\kappa^{\vee}, \OO \llbracket T_0 \rrbracket) = 0.$$ 
This is clear as $\OO\br{T_0}$ is isomorphic to the ring of formal power series in one variable and 
$\kappa^{\vee}$ is a finitely generated $\OO$-module.
\end{proof}

\begin{lem} \label{lem:quotprinciple}
Let $\tau$ be a smooth representation of $G$ whose irreducible subquotients consists of characters in $\BB$. Then $\SL_2(\Qp)$ acts trivially on $\tau$ and there is an exact sequence
\begin{equation}\label{quotient_tau}
0 \rightarrow \tau \rightarrow (\Ind^G_B \tau)_{\sm} \rightarrow Q \rightarrow 0.
\end{equation}
Moreover, the irreducible subquotients of $Q$ are twists of $\Sp$ by a character. In particular, 
$\Hom_G(\tau, Q)=0$.
\end{lem}

\begin{proof}
Since $\SL_2(\Qp)$ acts trivially on $\tau$ by \cite[Lemma 1.2.1]{Tung2020}, the unipotent radical of $B$ acts trivially on $\tau$. Thus $(\Ind^G_B \tau |_{B})_{\sm}$ coincides with the parabolic induction. Since the map $\tau \rightarrow (\Ind^G_B \tau)_{\sm}$ defined by $v \mapsto (g \mapsto gv)$ is $G$-equivariant and injective, we obtain the first assertion.

To show the second assertion, we choose an increasing and exhaustive filtration $\{R^j \}_{j \geq 0}$ of $\tau$ such that $R^j / R^{j-1}$ is a character for each $j$. Then we have $R^j / R^{j-1} \hookrightarrow (\Ind^G_B R^j / R^{j-1})_{\sm}$ with quotient isomorphic to a twist of $\Sp$ by a character. Thus the second assertion follows from the exactness of $(\Ind^G_B -)_{\sm}$.
\end{proof}

\begin{lem} \label{lem:nullhomJ} 
Let $J$  be injective in $\Mod^{\lfin}_{G, \zeta}(\OO)$ and let $\tau$ and $Q$  be as in Lemma \ref{lem:quotprinciple}. If $\tau^{\vee}$ is finitely generated over $\OO$ then
$$\Hom_G(J, \tau) = \Hom_G(J, (\Ind^G_B \tau)_{\sm})= \Hom_G(J, Q) = 0.$$
\end{lem}

\begin{proof} 
Since $\tau^{\vee}$ is finitely generated over $\OO$, $\tau[\varpi]$ is a finite $k$-vector space,
thus 
$$(\Ind^G_B \tau )_{\sm}^{\vee}/ \varpi \cong (\Ind^G_B \tau[\varpi])_{\sm}^{\vee}$$ 
is of finite length in $\dualcat(\OO)$ and the assertion follows from Proposition 
\ref{HMP}. 
\end{proof}

\begin{prop} \label{prop:JExt} Let $J$ and $\tau$ be as in Lemma \ref{lem:nullhomJ}. Then
$\Ext^1_{G, \zeta}(J, \tau) = 0$.
\end{prop}

\begin{proof}
Consider an exact sequence $0 \rightarrow \tau \rightarrow I \rightarrow J \rightarrow 0$. Applying $\Ord_B$ to it, we get an exact sequence
\begin{align*}
    0 &\rightarrow \Ord_B \tau \rightarrow \Ord_B I \rightarrow \Ord_B J 
   \rightarrow \RR^1 \Ord_B \tau \rightarrow \RR^1 \Ord_B I \rightarrow \RR^1 \Ord_B J
\end{align*}
of smooth $T$-representations. It is proved in \cite{MR2667892} that the functors $H^i\Ord_B$ in 
\cite{MR2667883} coincide with the derived functors $\RR^i\Ord_B$. Moreover, 
$H^1\Ord_B$ coincides with the $N$-coinvariants twisted by the character $\alpha^{-1}$, where
$\alpha(\bigl ( \begin{smallmatrix} a & 0 \\ 0 & d\end{smallmatrix} \bigr))=\omega(a d^{-1})$, by \cite[Proposition 3.6.2]{MR2667883}.

Since $\SL_2(\Qp)$ acts trivially on $\tau$, we have 
$\Ord_B \tau=0$ and $\RR^1 \Ord_B \tau
\cong \tau \otimes \alpha^{-1}$.
Thus the above exact sequence reduces to
\begin{equation} \label{equa:HomI}
    0 \rightarrow \Ord_B I \rightarrow \Ord_B J \rightarrow \tau \otimes \alpha^{-1}\rightarrow  I_N \otimes \alpha^{-1} \rightarrow 0.
\end{equation}
Since the middle map is zero by Lemma \ref{lem:nullhomOrd}, the map $\tau\hookrightarrow I$ induces an isomorphism $\tau \cong I_N$ of $T$-representations 
and hence
\begin{equation} \label{equa:Frob}
    \Hom_G(I, (\Ind^G_B \tau)_{\sm}) \cong \Hom_G(\tau, (\Ind^G_B \tau)_{\sm}).
\end{equation}
by adjunction formula. 

It follows from Lemma \ref{lem:nullhomJ} that the first $3$ terms in the long exact sequence obtained by applying $\Hom_G(J, -)$ to \eqref{quotient_tau} are  zero. Thus
by applying $\Hom_G(I, -)\rightarrow \Hom_G(\tau,-)$ to \eqref{quotient_tau}
we obtain the following commutative diagram:
\[
  \begin{tikzcd}
  \Hom_G(I, \tau) \arrow[r, hookrightarrow] \arrow[d, hookrightarrow] &\Hom_G(\tau, \tau) \arrow[d, hookrightarrow] \\
  \Hom_G(I, (\Ind^G_B \tau)_{\sm}) \arrow[r, "\sim", "\eqref{equa:Frob}"'] \arrow[d] &\Hom_G(\tau, (\Ind^G_B \tau)_{\sm}) \arrow[d] \\
  \Hom_G(I, Q) \arrow[r, hookrightarrow] &\Hom_G(\tau, Q)
  \end{tikzcd}
\]
with exact columns. Last part of Lemma \ref{lem:quotprinciple} says that
 $\Hom_G(\tau, Q)=0$ and hence $\Hom_G(I, Q)=0$. Thus all the maps in the 
 top square of the diagram are isomorphisms. 
 The preimage of $\id_{\tau}\in \Hom_G(\tau, \tau)$ in $\Hom_G(I, \tau)$ splits 
 the exact sequence $0 \rightarrow \tau \rightarrow I \rightarrow J \rightarrow 0$. This proves the proposition.
\end{proof}

\subsection{Quotient category} \label{section:quotcat}
Let $\TT(\OO)$ be the full subcategory of $\CC(\OO)$ whose objects have trivial $\SL_2(\Qp)$-action. By \cite[Lemma 10.25]{image} (for $p>2$) and \cite[Lemma 1.2.1]{Tung2020} (for $p=2$), $\TT(\OO)$ is a thick subcategory of $\CC(\OO)$ and hence we may consider the quotient category $\mathfrak Q(\OO) := \CC(\OO) / \TT(\OO)$. Let $\cT: \CC(\OO) \rightarrow \mathfrak Q(\OO)$ be the quotient functor; we note that $\cT$ is the identity on objects. It is shown in \cite[\S 10.3]{image} that $\mathfrak Q(\OO)$ is an abelian category with enough projectives and $\cT$ is an exact functor.

For a block $\BB$, we denote 
\[
\PB' = \bigoplus_{\substack{\pi \in \BB \\ \pi^{\SL_2(\Qp)} = 0}} P_{\pi^{\vee}}, \quad
\EB' = \End_{\CC(\OO)}(\PB'), \quad
\ZB' = Z(\EB'),
\]
where $P_{\pi^{\vee}}$ is a projective envelope of $\pi^{\vee}$ in $\CC(\OO)$.

\begin{prop} \label{prop:CEeuqiv}
Let $\BB$ be a block.
\begin{enumerate}
    \item $\cT \PB'$ is a projective object of $\mathfrak Q(\OO)$ and $\EB' \cong \End_{\mathfrak Q(\OO)}(\cT \PB')$.
    \item The functor $M \mapsto \Hom_{\mathfrak Q(\OO)}(\cT \PB', M)$ defines an equivalence of categories between $\mathfrak Q(\OO)_{\BB}$ and the category of pseudo-compact right $\EB'$-modules, with the inverse given by $\md \mapsto \md \wtimes_{\EB'} \cT \PB'$.
\end{enumerate}
\end{prop}

\begin{proof}
See \cite[Lemma 10.27]{image} for the first assertion and \cite[\S IV.4 Theorem 4, Corollary 1, Corollary 5]{MR232821} for the second assertion.
\end{proof}

\subsection{The centre}\label{centre_dash} 
In this section we will prove key results towards showing \ that $\EB$ and $\ZB$ are finite over $\ZB'$. If $\BB$ is of type  (i), (ii) or (iii) then we have $\PB=\PB'$ and thus $\ZB=\ZB'$.

\begin{lem} \label{lem:Endfg}
Let $E$ be a ring with centre $Z$ and let $\md$ be a finitely generated (right) $E$-module. Assume that $Z$ is noetherian and $E$ is (module) finite over $Z$. Then $\End_E(\md)$ and its centre $Z(\End_E(\md))$ are noetherian and finite over $Z$.
\end{lem}

\begin{proof}
Since $\md$ is finitely generated over $E$, there is a surjection $E^{\oplus n} \twoheadrightarrow \md$ for some $n$. This induces an injection $\End_E(\md) \hookrightarrow \Hom_E(E^{\oplus n}, \md)$ and a surjection $M_n(E^{\op}) \cong \End_E(E^{\oplus n}) \twoheadrightarrow \Hom_E(E^{\oplus n}, \md)$ of $Z$-modules. Since $M_n(E^{\op})$ is finite over $E^{\op}$, it is finite and noetherian over $Z$. It follows that $\Hom_E(E^{\oplus n}, \md)$ and thus $\End_E(\md)$, which can be identified with a $Z$-submodule of $\Hom_E(E^{\oplus n}, \md)$, are finite and noetherian over $Z$. This proves the lemma.
\end{proof}

\begin{lem} \label{lem:MSL2}
Let $\MB = \ker(\PB \rightarrow (\PB)_{\SL_2(\Qp)})$. Then the following hold:
\begin{enumerate}
    \item $(\MB)_{\SL_2(\Qp)} = 0$.
    \item For all $\pi \in \BB$, $\Hom_{\CC(\OO)}(\MB, \pi^{\vee})$ is finite dimensional over $k$.
\end{enumerate}
\end{lem}

\begin{proof}
It suffices to consider the block $\BB$ of type $(iv)$, $(v)$ and $(vi)$, otherwise $\MB = \PB$ and the assertion is trivial. Let $J_{\BB}$ be the Pontryagin dual of $\PB$. Then we have an exact sequence $0 \rightarrow J_{\BB}^{\SL_2(\Qp)} \rightarrow J_{\BB} \rightarrow \MB^{\vee} \rightarrow 0$. By applying $\Hom_{\GL_2(\Qp)}(\pi, -)$, we get a long exact sequence
\begin{align*}
    0 &\rightarrow \Hom_{G}(\pi, J_{\BB}^{\SL_2(\Qp)}) \rightarrow \Hom_{G}(\pi, J_{\BB}) \rightarrow \Hom_{G}(\pi, \MB^{\vee}) \\
    &\rightarrow \Ext^1_{G, \zeta}(\pi, J_{\BB}^{\SL_2(\Qp)}) \rightarrow \Ext^1_{G, \zeta}(\pi, J_{\BB}) = 0,
\end{align*}
where $\Ext^1_{G, \zeta}$ is the extension group in $\mlf_{G, \zeta}(\OO)$. 

If $\pi \in \BB$ is a character, then both $\Hom_{G}(\pi, J_{\BB}^{\SL_2(\Qp)})$ and $\Hom_{G}(\pi, J_{\BB})$ are $1$-dimensional over $k$, and $\Ext^1_{G, \zeta}(\pi, J_{\BB}^{\SL_2(\Qp)}) = 0$ by \cite[Lemma 1.2.1]{Tung2020}. Thus $\Hom_G(\pi, \MB^{\vee}) = 0$ and the first assertion follows (see \cite[Lemma III.40]{MR3267142} for another proof).

If $\pi \in \BB$ is  not a character, then $\Hom_{G}(\pi, J_{\BB}^{\SL_2(\Qp)}) = 0$ and $\Hom_{G}(\pi, J_{\BB})$ is $1$-dimensional over $k$. Since $\pi$ is killed by $\varpi$, for the second assertion it is enough to show that $\Ext^1_{G, \zeta}(\pi, (J_{\BB}[\varpi])^{\SL_2(\Qp)})$ is finite dimensional over $k$. This holds, since $(J_{\BB}[\varpi])^{\SL_2(\Qp)}$ is of finite length by Lemma \ref{lem:Jinvfinite}
and if $\chi$ is a character then $\Ext^1_{G, \zeta}(\pi, \chi \circ \det)$ is finite dimensional. 
\end{proof}

\begin{cor} \label{cor:MBfg}
There is a surjection $(\PB')^{\oplus n} \twoheadrightarrow \MB$ for some $n \in \NN$.
\end{cor}

\begin{proof}
This follows from Lemma \ref{lem:MSL2}, which shows that the cosocle of $\MB$ contains no characters and each irreducible representation in $\BB$ appears with finite multiplicity.
\end{proof}

\begin{lem} \label{lem:Extnull}
Both $\Hom_{\CC(\OO)}(\PB / \MB, \PB)$ and $\Ext^1_{\CC(\OO)}(\PB / \MB, \PB)$ are equal to zero.
\end{lem}

\begin{proof}
The first assertion follows since $\SL_2(\Qp)$ acts trivially on $\PB / \MB$ and $\PB^{\SL_2(\Qp)} = 0$ (see Lemma \ref{lem:PSL2}), and the second assertion follows from Proposition \ref{prop:JExt}
applied to $J=P_{\BB}^{\vee}$ and $\tau= J_{\BB}^{\SL_2(\Qp)}$. It follows from Lemma \ref{lem:Jinvfinite} that $\tau^{\vee}$ is a finitely generated $\OO$-module.
\end{proof}

\begin{prop} \label{prop:MBcentre}
There is a natural isomorphism $\EB \cong \End_{\CC(\OO)}(\MB)$. In particular, $\ZB \cong Z(\End_{\CC(\OO)}(\MB))$.
\end{prop}

\begin{proof}
Let $\phi \in \EB = \End_{\CC(\OO)}(\PB)$. Then the composition $\MB \xrightarrow{\phi} \PB \twoheadrightarrow \PB / \MB$ is the zero map since $\SL_2(\Qp)$ acts trivially on $\PB / \MB \cong (\PB)_{\SL_2(\Qp)}$ and $(\MB)_{\SL(\Qp)}= 0$ by Lemma \ref{lem:MSL2}. Thus $\phi$ maps $\MB$ to $\MB$ and restriction to $\MB$ induces a ring homomorphism $\EB \rightarrow \End_{\CC(\OO)}(\MB)$.

Applying the functor $\Hom_{\CC(\OO)}(\MB, -)$ to the exact sequence $0 \rightarrow \MB \rightarrow \PB \rightarrow \PB/\MB \rightarrow 0$, we get the exact sequence
\[
0 \rightarrow \End_{\CC(\OO)}(\MB) \rightarrow \Hom_{\CC(\OO)}(\MB, \PB) \rightarrow \Hom_{\CC(\OO)}(\MB, \PB / \MB).
\]
Since the last term is equal to zero by Lemma \ref{lem:MSL2}, we obtain $\End_{\CC(\OO)}(\MB) \cong \Hom_{\CC(\OO)}(\MB, \PB)$. On the other hand, by applying the functor $\Hom_{\CC(\OO)}(-, \PB)$ to the same short exact sequence, we get the exact sequence
\begin{align*}
    0 &\rightarrow \Hom_{\CC(\OO)}(\PB / \MB, \PB) \rightarrow \End_{\CC(\OO)}(\PB) \rightarrow \Hom_{\CC(\OO)}(\MB, \PB) \\ &\rightarrow \Ext^1_{\CC(\OO)}(\PB / \MB, \PB).
\end{align*}
Since $\Hom_{\CC(\OO)}(\PB / \MB, \PB) = 0$ and $\Ext^1_{\CC(\OO)}(\PB / \MB, \PB) = 0$ by Lemma \ref{lem:Extnull}, we deduce $\End_{\CC(\OO)}(\PB) \cong \Hom_{\CC(\OO)}(\MB, \PB)$ and the proposition follows.
\end{proof}

\begin{cor}\label{surjectiveZB} There is a natural surjective homomorphism 
$$\ZB'\rightarrow Z(\End_{\CC(\OO)}(\MB))\cong \ZB.$$
\end{cor}
\begin{proof} 
Note that we have $(\MB)^{\SL_2(\Qp)} = 0$ by Lemma \ref{lem:PSL2} and $(\MB)_{\SL_2(\Qp)} = 0$ by Lemma \ref{lem:MSL2}. It follows from \cite[Lemma 10.26]{image} that 
the functor $\cT$ induces  an isomorphism
\begin{align*}
    \End_{\CC(\OO)}(\MB) &\cong \End_{\mathfrak Q(\OO)}(\cT \MB), \quad \varphi \mapsto \cT\varphi.
\end{align*}
Since $\ZB'$ is the centre of $\mathfrak Q(\OO)_{\BB}$ it acts on $\cT \MB$ and this action induces a homomorphism $\ZB'\rightarrow Z(\End_{\mathfrak Q(\OO)}(\cT \MB))$, 
which we may compose with an isomorphism above to obtain a homomorphism $\ZB'\rightarrow Z(\End_{\CC(\OO)}(\MB))$. Since $\ZB$ is the centre of $\dualcat(\OO)_{\BB}$ a similar argument 
shows that $\ZB'$ is a $\ZB$-algebra and the surjection $(\PB')^{\oplus n} \twoheadrightarrow \MB$
in Corollary \ref{cor:MBfg} is  $\ZB$-equivariant. It induces a $\ZB$-equivariant surjection 
$(\cT \PB')^{\oplus n} \twoheadrightarrow \cT \MB$. We deduce that the map 
 $$\ZB'\rightarrow Z(\End_{\CC(\OO)}(\MB))$$ 
is a homomorphism of 
$\ZB$-algebras. Proposition \ref{prop:MBcentre} implies that the composition 
$$ \ZB\rightarrow \ZB' \rightarrow Z(\End_{\CC(\OO)}(\MB))\cong\ZB$$
is the identity map, which implies that the homomorphism is surjective. 
\end{proof}

\subsection{Colmez's Montreal functor}\label{CMF}
Let $\gal$ be the absolute Galois group of $\Qp$. We will consider $\zeta$ as a character of $\gal$ via local class field theory, normalized so that the uniformizers correspond to geometric Frobenius. Let $\varepsilon: \gal \rightarrow \Zp^{\times}$ be the $p$-adic cyclotomic character.

Colmez \cite{MR2642409} has defined an exact and covariant functor $\VV$ from the category of smooth, finite length representations of $G$ on $\OO$-torsion modules with a central character to the category of continuous finite-length representations of $\gal$ on $\OO$-torsion modules. This functor is modified in \cite[\S 5.7]{image} to an exact covariant functor
\[
\cV: \CC(\OO) \rightarrow \mpro_{\gal}(\OO)
\]
as follows. Let $M$ be in $\CC(\OO)$, if it is of finite length, we define $\cV(M) := \VV(M^\vee)^\vee(\zeta\varepsilon)$ where $\vee$ denotes the Pontryagin dual. For general $M \in \CC(\OO)$, we may write $M \cong \varprojlim M_i$, with $M_i$ of finite length in $\CC(\OO)$ and define $\cV(M) := \varprojlim \cV(M_i)$. With this normalization, we have 
\begin{itemize}
\item $\cV(\pi^{\vee}) = 0$ if $\pi \cong \chi \circ \det$;
\item $\cV(\pi^{\vee})= \chi_1$ if $\pi \cong (\Ind^G_B \chi_1 \otimes \chi_2)_{\sm}$;
\item $\cV(\pi^{\vee}) = \chi$ if $\pi \cong \Sp \otimes \chi \circ \det$;
\item $\cV(\pi^{\vee}) = \VV(\pi)$ is a $2$-dimensional absolutely irreducible Galois representation if $\pi$ is supersingular.
\end{itemize}
The functor $\cV$ induces a bijection $\BB\mapsto \rhobar_{\BB}$ between blocks containing an absolutely irreducible representation and  equivalence classes of semi-simple representations $\rhobar: \gal\rightarrow \GL_2(k)$, such that all irreducible summands of $\rhobar$ are absolutely irreducible. The representation $\rhobar_{\BB}$ can be described explicitly according to the classification of blocks given in section \ref{section:blocks}: in case (i) $\rhobar_{\BB}=\cV(\pi^{\vee})$ is absolute irreducible, in case (ii) $\rhobar_{\BB}=\chi_1\oplus \chi_2$,  in cases (iii) and (vi) $\rhobar_{\BB}=\chi\oplus \chi$, in cases (iv) and (v) $\rhobar_{\BB}= \chi \oplus \chi \omega$. 

Since the functor $\cV: \CC(\OO) \rightarrow \mpro_{\gal}(\OO)$ kills characters and hence every object in $\TT(\OO)$, it factors through $\cT: \CC(\OO) \rightarrow \mathfrak Q(\OO)$. We denote $\cV: \mathfrak Q(\OO) \rightarrow \mpro_{\gal}(\OO)$ by the same letter.

\begin{prop} \label{prop:DGalequiv}
For each block $\BB$, the functor $\cV$ induces an equivalence of categories between $\mathfrak Q(\OO)_{\BB}$ and its essential image in $\mpro_{\gal}(\OO)$.
\end{prop}

\begin{proof}
This is due to \cite{image} for cases (i)-(iv), \cite[Proposition 2.8]{2018arXiv180307451T} for case (v) and \cite[Proposition 1.3.2]{Tung2020} for case (iv).
\end{proof}

\begin{prop} \label{prop:surjgalE}
The map $\OO\br{\gal} \rightarrow \End_{\EB'}^{\cont}(\cV(\PB'))$
is surjective. Moreover, if $\BB$ is supersingular then $\End_{\EB'}^{\cont}(\cV(\PB'))\cong M_2((\EB')^{\op})$, and $\End_{\EB'}^{\cont}(\cV(\PB'))\cong(\EB')^{\op}$, otherwise. 
\end{prop}

\begin{proof}
It suffices to show that $M=\cV(\PB')$ satisfies all four conditions in Proposition \ref{surjective}. We note that the functor $\md \mapsto \md \wtimes_{\EB'} \cV(\PB')$ is exact by the equivalence of categories in Proposition \ref{prop:CEeuqiv} (2), the exactness of $\cV$ and the isomorphism
\[
\cV(\md \wtimes_{\EB'} \PB') \cong \md \wtimes_{\EB'} \cV(\PB')
\]
see the proof of \cite[Lemma 5.53]{image}. Thus by Lemma \ref{lem:Eproptest}, $\cV(\PB')$ is a projective $\EB'$-module. Let $\rr$ be the Jacobson radical of $\EB'$. 
Since $$(\EB'/\rr)\wtimes_{\EB'}\cV(\PB')\cong \cV((\EB'/\rr) \wtimes_{\EB'} \PB')$$ is either $1$- or $2$-dimensional $k$-vector space, topological Nakayama's lemma implies that $\cV(\PB')$ is a finitely generated $\EB'$-module, and thus part (1) of Proposition \ref{surjective} holds. Proposition \ref{prop:DGalequiv} implies that part (2) of Proposition \ref{surjective} holds. 

If $\md$ is a right pseudo-compact $\EB'$-module  then  \cite[Lemma 2.4]{brumer} implies that
\begin{equation}\label{dual}
\Hom_{\OO}^{\cont}( \md \wtimes_{\EB'} \cV(\PB') , k) \cong  \Hom_{\EB'}^{\cont}( \cV(\PB'), \Hom_{\OO}^{\cont}(\md, k)).
\end{equation}
Since in our situation $\EB'$ is a compact $\OO$-algebra, the irreducible (left or right) $\EB'$-modules are finite dimensional vector spaces with discrete topology, thus the map $\md\mapsto \md^*:=\Hom_k(\md, k)$  induces a bijection between irreducible left and irreducible right $\EB'$-modules. Moreover, if $\md$ is an irreducible right $\EB'$-module then it follows from \eqref{dual} that $\rho_{\md^*}\cong (\md \wtimes_{\EB'} \cV(\PB'))^*$, thus parts (3) and (4) of Proposition \ref{surjective} follow from Proposition \ref{prop:DGalequiv}.

If $\BB$ is supersingular then $\BB$ contains only one irreducible $\pi$, which is not a character. Thus $\PB=\PB'$, $\EB=\EB'$ is a local ring and $k\wtimes_{\EB} \PB= \pi^{\vee}$. Thus $ k\wtimes_{\EB} \cV(\PB)\cong \cV(\pi^{\vee})\cong \rhobar_{\BB}$ is an absolutely irreducible $2$-dimensional representation. Since $\EB$ is a local ring, we deduce that $\cV(\PB)$ is a free $\EB$-module of rank $2$. Thus $\End^{\cont}_{\EB}(\cV(\PB))\cong M_2((\EB')^{\op})$.

If $\BB$ is of type (iii) of (vi) then the block in the quotient category contains only one irreducible object and Colmez's functor maps it to a $1$-dimensional $\gal$-representation. The same argument as in the supersingular case shows that $\EB'$ is a local ring and $\cV(\PB')$ is a free $\EB'$-module of rank $1$, and thus $\End_{\EB'}^{\cont}(\cV(\PB'))= (\EB')^{\op}$.

If $\BB$ is of type (ii), (iv) or (v) then $\mathfrak Q(\OO)_{\BB}$ contains exactly two irreducible objects and Colmez's functor sends them to distinct $1$-dimensional $\gal$-representations $\chi_1$, $\chi_2$.  It follows from Corollary \ref{cor:autosurj} that $\cV(\PB')$ is a free $\EB'$-module of rank $1$, and thus $\End_{\EB'}^{\cont}(\cV(\PB'))= (\EB')^{\op}$.
\end{proof}

\subsection{Banach space representations}\label{bsp}
Let $\aBan_{G, \zeta}(L)$ be the category of admissible unitary $L$-Banach space representations 
\cite[\S 3]{st_iw} on which $Z$ acts by the character $\zeta$. We note that $\aBan_{G, \zeta}(L)$ is an abelian category \cite[Theorem 3.5]{st_iw}. Any $\Pi \in \aBan_{G, \zeta}(L)$ has an open, bounded and $G$-invariant lattice $\Theta$ and $\Theta \otimes_\OO k$ is an admissible smooth $k$-representation of $G$. Let $\Theta^d = \Hom_\OO(\Theta, \OO)$ be the Schikhof dual of $\Theta$ endowed with the topology of pointwise convergence. Then $\Theta^d$ is an object of $\mpro_G(\OO)$, \cite[Lemma 4.4]{image}. If $\Theta^d$ is in $\CC(\OO)$ then $\Xi^d$ is in $\CC(\OO)$ for every open bounded $G$-invariant lattice $\Xi$ in $\Pi$, since $\Theta$ and $\Xi$ are commensurable and $\CC(\OO)$ is closed under subquotients \cite[Lemma 4.6]{image}.

If $\Pi \in \aBan_{G, \zeta}(L)$, then we let
\[
\cV(\Pi) = \cV(\Theta^d) \otimes_{\OO} L,
\]
where $\Theta$ is any open bounded $G$-invariant lattice in $\Pi$. Then $\cV$ is exact and contravariant on $\aBan_{G, \zeta}(L)$.

\section{Density}\label{density}

\subsection{Capture} \label{section:capture}
Let $G = \GL_2(\Qp)$ and $K = \GL_2(\Zp)$. Write $Z$ for the center of $G$ and $Z(K)$ for the center of $K$. Let $\psi: Z(K) \rightarrow \OO^{\times}$ be a continuous character. We identify $Z$ with $\Qp^{\times}$ and $Z(K)$ with $\Zp^{\times}$ via the map  $\left(\smallmatrix x & 0 \\ 0 & x \endsmallmatrix\right) \mapsto x$. 

\begin{lem} \label{lem:capture}
Let $\{V_i\}_{i \in I}$ be a family of continuous representations of $K$ on finite dimensional $L$-vector spaces with central character $\psi$, and let $M \in \mpro_{K, \psi}(\OO)$ be $\OO$-torsion free. The following conditions are equivalent.
\begin{enumerate}
    \item The smallest quotient $M \twoheadrightarrow Q$ such that $\Hom^{\cont}_{\OO\br{K}}(Q, V_i^*) \cong \Hom^{\cont}_{\OO\br{K}}(M, V_i^*)$ for all $i \in I$ is equal to $M$.
    \item The intersection of the kernels of all $\phi \in \Hom^\cont_{\OO\br{K}}(M, V^*_i)$ for each $i \in I$ is equal to zero.
    \item The image of the evaluation map
    \[
    \bigoplus_{i \in I} \Hom_K(V_i, \Pi(M)) \otimes_L V_i \rightarrow \Pi(M)
    \]
    is a dense subspace, where $\Pi(M) := \Hom^{\cont}_{\OO}(M, L)$ is an L-Banach space equipped with the supremum norm.
\end{enumerate}
\end{lem}

\begin{proof}
See \cite[Lemma 2.7, Lemma 2.10]{CDP}.
\end{proof}

\begin{defi}
We say that $\{V_i\}_{i \in I}$ captures $M$ if it satisfies one of the equivalent conditions above.
\end{defi}

Since $1+p\Zp$ (resp. $1+4\ZZ_2$) is a free pro-$p$ group of rank 1 if $p>2$ (resp. $p=2$), there is a smooth non-trivial character $\chi: \Zp^\times \rightarrow L^\times$ and a continuous character $\eta_0 : \Zp^{\times} \rightarrow L^\times$ such that $\psi = \chi \eta_0^2$. Let $e$ be the smallest integer such that $\chi$ is trivial on $1+p^e\Zp$. Let 
\[
J = \left(\smallmatrix \Zp^\times & \Zp \\ p^e \Zp & \Zp^\times \endsmallmatrix\right),
\]
and let $\chi \otimes \Eins: J \rightarrow L^\times$ be the character which sends $\left(\smallmatrix a & b \\ c & d \endsmallmatrix\right)$ to $\chi(a)$. Then the representation $\tau=\Ind^{K}_{J} (\chi \otimes \Eins)$ is a principal series type. That is, for an irreducible smooth $\overline{L}$-representation $\pi$ of $G$, we have $\Hom_{L[K]}(\tau, \pi) \neq 0$ if and only if $\pi \cong (\Ind^G_B \psi_1 \otimes \psi_2)_\sm$, where $B$ is a Borel subgroup and $\psi_1|_{\Zp^\times} = \chi$ and $\psi_2|_{\Zp^\times} = \Eins$ \cite[\S A 2.2]{MR1944572}.

\begin{prop} \label{prop:principalcapture}
The family
\[
\{\Ind^{K}_{J} (\chi \otimes \Eins) \otimes \Sym^{2a} L^2 \otimes (\det)^{-a} \otimes \eta \eta_0 \circ \det \}_{a \in \NN , \eta},
\]
where $\eta$ runs over all the characters with $\eta^2 = \Eins$, captures every projective object in $\mpro_{K, \psi^{-1}}(\OO)$.
\end{prop}
\begin{proof}
See \cite[Proposition 2.7]{2adicANT}.
\end{proof}
We will denote the family of representations in the above proposition by $\{V_i\}_{i\in I}$. We
note that each $V_i$ is a twist of a locally algebraic representation by a unitary character $\eta_0$, which might not be locally algebraic.  However, twisting by its inverse will get us to a locally algebraic 
situation, which is sufficient for all arguments that follow. 

\subsection{Locally algebraic vectors in $\Pi(P)$} \label{section:localg}
Let $\zeta: Z \rightarrow \OO^{\times}$ be a continuous character and $\psi = \zeta |_{K}$. Let $P$ be a projective object in $\CC(\OO)$ and $E = \End_{\CC(\OO)}(P)$. In particular, $P$ is a torsionfree compact linear-topological $\OO$-module. Let 
\[
\Pi(P) := \cHom{\OO}(P, L)
\]
with the topology induced by the supremum norm. Then we have $E[1/p] \cong \End^{\cont}_{G}(\Pi(P))$.

If $V$ is a continuous representation of $K$ on a finite dimensional $L$-vector space, then
\begin{equation} \label{equa:dual}
    \Hom_K(V, \Pi(P)) \cong \cHom{\OO \br{K}} (P, V^*).
\end{equation}
Since $P$ is projective in $\mpro_{K, \psi^{-1}}(\OO)$ by \cite[Corollary 3.10]{MR2667892}, the family of finite dimensional $K$-representations associated to $\psi$ in Proposition \ref{prop:principalcapture}, which we denote by $\{V_i\}_{i \in I}$, captures $P$. We view $V_i$ as a representation of $KZ$ by letting $\unifmatrix$ act by $\zeta(\varpi)$.

\begin{prop} \label{prop:Hecke}
For each $i \in I$, we let $A_i := \End_G(\cInd^G_{KZ} V_i)$. Then
\begin{enumerate}
    \item $A_i$ is isomorphic to $L[T]$;
    \item $\cInd^G_{KZ} V_i$ is flat over $A_i$.
\end{enumerate}
\end{prop}

\begin{proof} We may write $V_i= \Ind^{K}_{J} (\chi \otimes \Eins)\otimes_L W_i$, where the action of $KZ$ on $W_i$ extends to an action of $G$, see Proposition \ref{prop:principalcapture}. Then 
$$ \cInd^G_{KZ} V_i \cong \cInd^G_{KZ}(\Ind^{K}_{J} (\chi \otimes \Eins))\otimes_L W_i.$$
Here we view $\Ind^{K}_{J} (\chi \otimes \Eins)$ as a representation of $KZ$ by letting $\unifmatrix$ act on $\Ind^{K}_{J} (\chi \otimes \Eins)$ by $\zeta(\varpi) \zeta^{-1}_{W_i}( \varpi)$,
where $\zeta_{W_i}$ is the central character of $W_i$. Since the restriction of $W_i$  to any compact open subgroup of $G$ remains absolutely irreducible, the above isomorphism induces an isomorphism of $L$-algebras $$\End_G(\cInd^G_{KZ} V_i) \cong \End_G(\cInd^G_{KZ}(\Ind^{K}_{J} (\chi \otimes \Eins))).$$ Thus we may assume that $W_i$ is the trivial representation. 

By \cite{MR1711578}, the $K$-type $\Ind^{K}_{J} (\chi \otimes \Eins)$ is a $G$-cover of the $K_M$-type $\chi \otimes \Eins$, where $M = \Qp^{\times} \times \Qp^{\times}$ and $K_M = \Zp^{\times} \times \Zp^{\times}$. Thus there is an algebra isomorphism
\[
j_M: \End_{M}(\cInd^M_{K_M} (\chi \otimes \Eins)) \xrightarrow{\sim} \End_G(\cInd^G_J(\chi \otimes \Eins))
\]
such that for each $f \in \End_{M}(\cInd^M_{K_M} \chi \otimes \Eins)$, we have $\supp(j_M f) = J \supp(f) J$. It follows that
\[
L[T] \xrightarrow{\sim} \End_{M}(\cInd^{M}_{K_MZ} (\chi \otimes \Eins)) \xrightarrow[j_M]{\sim} A_i,
\]
where $T$ maps to an element in $A_i$ supported at $JZ \left(\smallmatrix p & 0 \\ 0 & 1 \endsmallmatrix\right) JZ$. Here we view $\chi \otimes \Eins$ as a representation of $K_M Z$ by letting $\unifmatrix$ act by $\zeta(\varpi)$. This shows the first assertion.

To prove the second assertion, it suffices to show that $\cInd^G_{KZ} V_i$ is torsion free since $A_i$ is a PID. After tensoring with $\overline{L}$, this is equivalent to $\cInd^G_{KZ} V_i$ has no $T - \lambda$ torsion, which is easily seen using the fact that the functions in $\cInd^G_{KZ} V_i$ are compactly supported.
\end{proof}

In particular, Frobenius reciprocity gives
\begin{equation} \label{equa:Frobenius}
    \Hom_K(V_i, \Pi(P)) \cong \Hom_G(\cInd^G_{KZ} V_i, \Pi(P)).
\end{equation}
Hence, $\Hom_K(V_i, \Pi(P))$ is naturally an $A_i$-module and we may transport the action of $A_i$ onto $\cHom{\OO \br{K}} (P, V_i^*)$ via (\ref{equa:dual}).

If $V$ is a continuous representation of $K$ on a finite dimensional $L$-vector space and if $\Theta$ is an open, bounded $K$-invariant lattice in $V$, let $|\cdot|$ be the norm on $V^*$ given by $|\ell| := \sup_{v \in \Theta} |\ell(v)|$, so that $\Theta^d = \Hom_\OO(\Theta, \OO)$ is the unit ball in $V^*$ with respect to $|\cdot|$. The topology on $\cHom{\OO\br{K}}(P, V^*)$ is given by the norm $\lVert \phi \rVert := \sup_{v \in P} |\phi(v)|$, and $\cHom{\OO\br{K}}(P, \Theta^d)$ is the unit ball in this Banach space.

\begin{prop} \label{prop:lfindense}
For all $i \in I$ the submodule
\[
\cHom{\OO \br{K}} (P, V_i^*)_{\lfin} := \{\phi \in \cHom{\OO \br{K}} (P, V_i^*): \ell_{A_i}(A_i \phi) < \infty \}
\]
is dense in $\cHom{\OO \br{K}} (P, V_i^*)$, where $\ell_{A_i}(A_i \phi)$ is the length of $A_i \phi$ as an $A_i$-module.
\end{prop}

\begin{proof}
See \cite[Proposition 2.19]{CDP}.
\end{proof}

\begin{prop} \label{prop:completion} Let $\mm$ be a maximal ideal of $A_i$ and let $\Pi$ be a completion of $A_i / \mm^n \otimes_{A_i} \cInd^G_{KZ} V_i$ with respect to a $G$-invariant norm. Then $\Pi$ is the universal unitary completion of $A_i / \mm^n \otimes_{A_i} \cInd^G_{KZ} V_i$. Moreover, the action of $A_i$ on $A_i / \mm^n \otimes_{A_i} \cInd^G_{KZ} V_i$ extends to a continuous action of $A_i$ on $\Pi$.
\end{prop}

\begin{proof}
Let $\Pi^{\circ}$ be a $G$-invariant $\OO$-lattice of $\Pi$. Then 
\[
\Theta := \Pi^{\circ} \cap (A_i / \mm^n \otimes_{A_i} \cInd^G_{KZ} V_i)
\]
is a $G$-invariant $\OO$-lattice of $A_i / \mm^n \otimes_{A_i} \cInd^G_{KZ} V_i$. By \cite[Proposition 1.17]{MR2181093}, it suffices to show that $\Theta$ is of finite type over $\OO[G]$.

By Proposition \ref{prop:Hecke} (i), we have $\kappa(\mm) := A_i / \mm \cong L[T] / f(T)$, where $f(T) \in L[T]$ is an irreducible polynomial, is a finite extension of $L$. Define a finite, increasing, exhaustive filtration $\{ R^j \}_{n \geq j \geq 0}$ of $A_i / \mm^n \otimes_{A_i} \cInd^G_{KZ} V_i$ by $G$-invariant $A_i$-submodules $R^j = \mm^{n-j} / \mm^n \otimes_{A_i} \cInd^G_{KZ} V_i$. Then we have 
\[
R^j / R^{j-1} \cong \kappa(\mm) \otimes_{A_i} \cInd^G_{KZ} V_i
\]
for each $j$ and $\{ \Theta^j := \Theta \cap R^j\}_{n \geq j \geq 0}$ is a finite, increasing, exhaustive filtration of $\Theta$ such that $\Theta^j$ is a $G$-invariant $\OO$-lattice of $R^j$ for each $j$. Moreover, $\Theta^j / \Theta^{j-1}$ gives rise to a $G$-invariant $\OO$-lattice of $R^j / R^{j-1} \cong A_i / \mm \otimes_{A_i} \cInd^G_{KZ} V_i$, and thus is finitely generated over $\OO[G]$ by the proof of \cite[Theorem 4.3.1]{MR2642406}. This implies that $\Theta$ is finitely generated over $\OO[G]$ and the first assertion follows. 

If $\phi \in A_i$ then $\phi(\Theta) \subset \varpi^n \Theta$ for some $n\in \ZZ$, as $\Theta$ is finitely generated over $\OO[G]$. This implies the second assertion.
\end{proof}

We denote the universal unitary completion of $A_i / \mm^n \otimes_{A_i} \cInd^G_{KZ} V_i$ by $\Pi_{i,\mm, n}$. If $n=1$ then $\Pi_{i,\mm, 1}$ is the universal unitary completion of $\kappa(\mm) \otimes_{A_i} \cInd^G_{KZ} V_i$ studied in \cite[Theorem 4.3.1]{MR2642406} and \cite[Proposition 2.2.1]{MR2667890}.

\begin{cor} \label{cor:completion}
 If $\phi \in \cHom{\OO \br{K}} (P, V_i^*)_{\lfin}$ is such that $A_i \phi \cong A_i / \mm^n$ for  a maximal ideal $\mm$ of $A_i$, then $\phi$ induces an injection $\Pi_{i,\mm, n} \hookrightarrow \Pi(P)$. Moreover, $\Pi_{i,\mm, n}$ admits a filtration of length $n$ such that each graded piece is isomorphic to the universal unitary completion $\Pi_{i,\mm, 1}$ of $\kappa(\mm) \otimes_{A_i} \cInd^G_{KZ} V_i$.
\end{cor}

\begin{proof}
The assumption $A_i \phi \cong A_i / \mm^n$ implies that 
$A_i / \mm^n \otimes_{A_i} \cInd^G_{KZ} V_i$ injects into $\Pi(P)$. 
The first assertion follows immediately from Proposition \ref{prop:completion}.
To show the second assertion, we let $\{R^j\}_{n \geq j \geq 0}$ be the filtration of 
$A_i / \mm^n \otimes_{A_i} \cInd^G_{KZ} V_i$ defined in Proposition \ref{prop:completion}.
Let $\Pi^j$ be the closure of $R^j$ in $\Pi_{i,\mm, n}$; then $\Pi^n= \Pi_{i,\mm,n}$.
Since $\mm R^j = R^{j-1}$ we have $\mm \Pi^j = \Pi^{j-1}$ and hence $\Pi^j= \mm^{n-j} \Pi_{i,\mm,n}$.
If $\Pi^j= \Pi^{j-1}$ then $\Pi^j=\mm \Pi^j$ and hence $\Pi^j= \mm^j \Pi^j= \mm^n \Pi_{i,\mm,n}=0$. Since $\Pi^j\neq 0$ for $1\le j\le n$ we conclude
that $\Pi^j\neq \Pi^{j-1}$ for $1\le j\le n$. Moreover, $\Pi^j / \Pi^{j-1}$ 
 contains $R^j / R^{j-1} \cong A_i / \mm \otimes_{A_i} \cInd^G_{KZ} V_i$ as a dense subspace, 
 and thus is isomorphic to $\Pi_{i,\mm, 1}$. This proves the corollary.
\end{proof}

Note that the image of any $\phi \in \Hom_G(A_i / \mm^n \otimes_{A_i} \cInd^G_{KZ} V_i, \Pi(P))$ is isomorphic to $A_i / \mm^k \otimes_{A_i} \cInd^G_{KZ} V_i$ for some $0 \leq k \leq n$. Hence it induces an injection $\Pi_{i, \mm, k} \hookrightarrow \Pi(P)$ by Corollary \ref{cor:completion}.

\begin{prop} \label{prop:density}
Let $P$ be a projective object in $\CC(\OO)$. Then the image of the evaluation map
\[
\bigoplus_{i \in I} \bigoplus_{\mm, n} \cHom{G}(\Pi_{i,\mm, n}, \Pi(P)) \otimes_L \Pi_{i,\mm, n} \rightarrow \Pi(P),
\]
where $\mm$ runs through maximal ideals of $A_i$ and $n \in \NN$, is a dense subspace.
\end{prop}

\begin{proof}
Let $\Pi$ be the closure of the image of the evaluation map and $M$ the image of $P$ under $\Hom^{\cont}_{L}(\Pi(P), L) \twoheadrightarrow \Hom^{\cont}_{L}(\Pi, L)$. Then we have
\begin{align*}
    \cHom{\OO \br{K}} (P, V_i^*)_{\lfin} 
    &\cong \oplus_{\mm} \cHom{\OO \br{K}} (P, V_i^*)[\mm^{\infty}] \\
    &\cong \oplus_{\mm} \Hom_K(V_i, \Pi(P))[\mm^{\infty}] \\
    &\cong \oplus_{\mm} \Hom_G(\cInd^G_{KZ} V_i, \Pi(P))[\mm^{\infty}] \\
    &\cong \oplus_{\mm} \varinjlim_{n} \Hom_G(\cInd^G_{KZ} V_i, \Pi(P))[\mm^n] \\
    &\cong \oplus_{\mm} \varinjlim_{n} \Hom_G(A_i / \mm^n \otimes_{A_i} \cInd^G_{KZ} V_i, \Pi(P)) \\
    &\cong \oplus_{\mm} \varinjlim_{n} \cHom{G}(\Pi_{i,\mm, n}, \Pi(P)).
\end{align*}
Here the first isomorphism is due to the fact that any module $M$ over a commutative ring $A$, such that every finitely generated submodule is of finite length, admits a decomposition $M \cong \oplus_{\mm} M[\mm^{\infty}]$ with $\mm$ running through maximal ideals of $A$, the second isomorphism is given by (\ref{equa:dual}), and the last isomorphism is due to Corollary \ref{cor:completion}. Similarly, we have
\begin{align*}
    \cHom{\OO \br{K}} (M, V_i^*)_{\lfin} 
    &\cong \oplus_{\mm} \cHom{\OO \br{K}} (M, V_i^*)[\mm^{\infty}] \\
    &\cong \oplus_{\mm} \Hom_K(V_i, \Pi)[\mm^{\infty}] \\
    &\cong \oplus_{\mm} \Hom_G(\cInd^G_{KZ} V_i, \Pi)[\mm^{\infty}] \\
    &\cong \oplus_{\mm} \varinjlim_{n} \Hom_G(A_i / \mm^n \otimes_{A_i} \cInd^G_{KZ} V_i, \Pi) \\
    &\cong \oplus_{\mm} \varinjlim_{n} \cHom{G}(\Pi_{i,\mm, n}, \Pi) \\
    &= \oplus_{\mm} \varinjlim_{n} \cHom{G}(\Pi_{i,\mm, n}, \Pi(P)).
\end{align*}
Thus by Proposition \ref{prop:lfindense}, we have $\cHom{\OO \br{K}} (P, V_i^*) \cong \cHom{\OO \br{K}} (M, V_i^*)$ for each $i \in I$. Combining this with Lemma \ref{lem:capture}, we deduce the proposition.
\end{proof}

\begin{cor} \label{cor:density}
Let $\md_{i,\mm, n} := \Hom_{\CC(\OO)}(P, \Theta^d) \otimes_{\OO} L$, where $\Theta$ is an open bounded $G$-invariant lattice in $\Pi_{i,\mm, n}$. Then 
$$\bigcap_{i \in I} \bigcap_{\mm, n} \mathfrak{a}_{i,\mm, n} = 0,$$
where $\mathfrak{a}_{i,\mm, n} := \ann_E \bigl(\md_{i,\mm, n}\bigr)$.
\end{cor}

\begin{proof}
For $\Pi \in \aBan_{G, \zeta}(L)$ with a $G$-invariant $\OO$-lattice $\Theta$, we have 
$$\Hom_{\CC(\OO)}(P, \Theta^d) \otimes_{\OO} L \cong \cHom{G}(\Pi, \Pi(P)).$$ Thus the evaluation map in Proposition \ref{prop:density} induces an $E[1/p]$-homomorphism
\[
\bigoplus_{i \in I} \bigoplus_{\mm, n} \md_{i,\mm, n} \otimes_L \Pi_{i,\mm, n} \rightarrow \Pi(P)
\]
with a dense image. Since $E[1/p]$ acts faithfully on the RHS of the map, it acts faithfully on the LHS as well. This proves the corollary since the $E$-action on LHS factors through the quotient $E / \bigcap_{i \in I} \bigcap_{\mm, n} \mathfrak{a}_{i,\mm, n}$.
\end{proof}

\section{Main results}

Given a block $\BB$, we have defined $\pi_{\BB}$, $\PB$, $\EB$ in \S \ref{section:blocks} and $\PB'$, $\EB'$ in \S \ref{section:quotcat}. We assume that all irreducibles in $\BB$ are absolutely irreducible. This can be achieved by replacing $k$ with a finite extension. Let $\rhobar_{\BB}$ be the 2-dimensional semi-simple Galois representation of $\gal$ over $k$ defined by $\cV(\pi_{\BB}^{\vee})$ in cases (i), (ii), (iv), (v) and by a direct sum of two copies of $\cV(\pi_{\BB}^{\vee})$ in cases (iii), (vi), see section 
\ref{CMF} for an explicit description.  We write $R^{\ps, \zeta\varepsilon}_{\tr\rhobar_{\BB}}$ for the universal pseudodeformation ring of $\tr \rhobar_{\BB}$ with a fixed determinant $\zeta \varepsilon$. This ring is noetherian by \cite[Proposition F]{che_durham}. We let $T: \gal\rightarrow R^{\ps, \zeta\varepsilon}_{\tr\rhobar_{\BB}}$ be the universal object, see section \ref{section:pseudorepresentations}.

\subsection{Finiteness} 
Let $\{V_i \}_{i \in I}$ be a family of $K$-representations defined in Proposition \ref{prop:principalcapture} and let $A_i = \End_G(\cInd^G_{KZ} V_i)$. For each $i \in I$, a maximal ideal $\mm$ of $A_i$ and $n \in \NN$, we write $\Pi_{i,\mm, n}$ for the universal unitary completion of $A_i / \mm^n \otimes_{A_i} \cInd^G_{KZ} V_i$.

\begin{lem} \label{lem:infdefm}
Assume $\Pi_{i,\mm, n}$ is a subrepresentation of $\Pi(\PB')$. Then $\cV(\Pi_{i,\mm, n})$ is a finite free $A_{i} / \mm^n$-module of rank equal to $\dim_{A_i / \mm}(\cV(\Pi_{i,\mm, 1})) \leq 2$. Moreover,
\begin{enumerate}[label=(\roman*)]
    \item if $\rank_{A_i / \mm^n} \cV(\Pi_{i,\mm, n}) = 2$ then $\cV(\Pi_{i,\mm, n})$ is a deformation to $A_{i} / \mm^n$ of the absolutely irreducible $2$-dimensional $L$-representation $\cV(\Pi_{i,\mm, 1})$ of $\gal$;
    \item if $\rank_{A_{i} / \mm^n} \cV(\Pi_{i,\mm, n}) = 1$ then the action of $\gal$ on $\cV(\Pi_{i,\mm, n})$ is given by a $(A_{i} / \mm^n)^{\times}$-valued character lifting $\cV(\Pi_{i,\mm, 1})$.
\end{enumerate}
\end{lem}

\begin{proof}
Since $V_i$ is a principal series type, the $\gal$-module
\[
\cV(\Pi_{i,\mm, n}) / \mm \cV(\Pi_{i,\mm, n}) \xrightarrow{\sim} \cV(\Pi_{i,\mm, 1})
\]
has dimension $r \leq 2$ over $\kappa(\mm) := A_i / \mm$ by \cite[Theorem 4.3.1]{MR2642406} and \cite[Proposition 2.2.1]{MR2667890}. Nakayama's lemma implies that we have a surjection
\[
(A_i / \mm^n)^{\oplus r} \twoheadrightarrow \cV(\Pi_{i,\mm, n}).
\]
Proposition \ref{cor:completion} and exactness of $\cV$ imply that $\cV(\Pi_{i,\mm, n})$ has length $nr$ as $A_i / \mm^n$-module. Hence the surjection is an isomorphism and the lemma follows.
\end{proof}

\begin{lem}\label{wrong_lemma}
Under the same assumptions as in Lemma \ref{lem:infdefm}, there is a natural map $\theta_{i,\mm, n}: R^{\ps, \zeta\varepsilon}_{\tr\rhobar_{\BB}} \rightarrow A_i / \mm^n$, which induces a map
\[
R^{\ps, \zeta\varepsilon}_{\tr\rhobar_{\BB}} \br{\gal}/ J \rightarrow \End_{A_i / \mm^n}(\cV(\Pi_{i,\mm, n})).
\]
\end{lem}

\begin{proof}
If $\cV(\Pi_{i,\mm, n})$ is of rank $2$ over $A_i/\mm^n$, it follows from Lemma \ref{lem:infdefm} (i) that $\cV(\Pi_{i,\mm, n})$ is a deformation of the 2-dimensional Galois representation $\cV(\Pi_{i,\mm, 1})$ to $A_i/{\mm}^n$ with determinant $\zeta\varepsilon$. It follows from \cite[\S 4.1, Theorem 3.17]{che_durham} that there is an $\OO$-algebra map $\theta_{i,\mm, n}: R^{\ps, \zeta\varepsilon}_{\tr\rhobar_{\BB}} \rightarrow A_i / \mm^n$ such that the specialization of $T$ along $\theta_{i,\mm, n}$ is equal to $\tr_{A_i/{\mm}^n} \cV(\Pi_{i,\mm, n})$. This map induces a homomorphism of $\OO$-algebras 
\[
R^{\ps, \zeta\varepsilon}_{\tr\rhobar_{\BB}} \br{\gal}\rightarrow \End_{A_i / \mm^n}(\cV(\Pi_{i,\mm, n})),
\]
and Cayley--Hamilton for $M_2(A_i / \mm^n)$ implies that $J$ lies in the kernel of this map. 

If $\cV(\Pi_{i,\mm, n})$ is given by a character $\chi_{i,\mm, n}: \gal \rightarrow (A_i / \mm^n)^{\times}$, then the same argument applies to the representation 
 $\chi_{i,\mm, n} \oplus \chi_{i,\mm,n}^{-1} \zeta\varepsilon$, so that we get a map
\[
R^{\ps, \zeta\varepsilon}_{\tr\rhobar_{\BB}} \br{\gal}/J\rightarrow \End_{A_i / \mm^n}(\chi_{i,\mm, n} \oplus \chi_{i,\mm,n}^{-1}\zeta\varepsilon ).\]
Its image commutes with the idempotent which projects onto the direct summand $\chi_{i,\mm, n}$. Hence, the image is contained in 
$\End_{A_i / \mm^n}(\chi_{i,\mm, n})\times  \End_{A_i / \mm^n}(\chi_{i,\mm,n}^{-1}\zeta\varepsilon)$
and we may project to  $\End_{A_i / \mm^n}(\chi_{i,\mm, n})$ to obtain the required homomorphism.
\end{proof}

In Propositions \ref{prop:surjgalE} and \ref{prop:galsurj} we have established surjections 
$$\alpha: \OO\br{\gal}\twoheadrightarrow \End_{\EB'}^{\cont}(\cV(\PB')), \quad \beta:  \OO\br{\gal}\twoheadrightarrow R^{\ps, \zeta\varepsilon}_{\tr\rhobar_{\BB}}\br{\gal}/ J.$$ 
\begin{thm} \label{thm:main}
The above maps induce a surjection $$R^{\ps, \zeta\varepsilon}_{\tr\rhobar_{\BB}}\br{\gal}/ J \twoheadrightarrow \End_{\EB'}^{\cont}(\cV(\PB')).$$ In particular, $\EB'$ and $\ZB'$ are  finite over $R^{\ps, \zeta\varepsilon}_{\tr\rhobar_{\BB}}$ and hence noetherian.
\end{thm}

\begin{proof} For the first part we have to show that $\Ker \beta \subset \Ker \alpha$. 
Let $M = \cV(\PB')$ and $\md_{i,\mm, n} := \md(\Pi_{i,\mm, n})$ with $i,\mm, n$ as in Corollary \ref{cor:density}. Then the assumptions in Lemma \ref{lem:kergalend} are satisfied by Corollary \ref{cor:density}. It follows that the kernel of $\alpha$  is given by $\cap_{i \in I} \cap_{\mm, n} \mathfrak b_{i,\mm, n}$, where $\mathfrak b_{i,\mm, n}$ is the $\OO\br{\gal}$-annihilator of 
\[
\md_{i,\mm, n} \otimes_{\EB'} \cV(\PB') \cong \cV(\md_{i,\mm, n} \otimes_{\EB'} \PB') \cong \cV(\Pi_{i,\mm, n}).
\]
Since the action of $\OO\br{\gal}$ on $\cV(\Pi_{i,\mm, n})$ factors through $R^{\ps, \zeta\varepsilon}_{\tr\rhobar_{\BB}}\br{\gal}/ J$ by Lemma \ref{wrong_lemma}, $\mathfrak b_{i,\mm, n}$  contains the kernel of $\beta$ and hence $\Ker \beta \subset \Ker \alpha$.

 The second assertion is a consequence of the first assertion and the finiteness of $R^{\ps, \zeta\varepsilon}_{\tr\rhobar_{\BB}}\br{\gal}/ J$ over $R^{\ps, \zeta\varepsilon}_{\tr\rhobar_{\BB}}$ \cite[Proposition 3.6]{WE_alg}.
\end{proof}

\begin{cor}\label{cor:allfinite}
$\EB$  and $\ZB$ are finite over $R^{\ps, \zeta\varepsilon}_{\tr\rhobar_{\BB}}$ and hence noetherian. 
\end{cor}

\begin{proof}
By \cite[Lemma 10.26]{image}, we have
\begin{align*}
    \End_{\CC(\OO)}(\MB) &\cong \End_{\mathfrak Q(\OO)}(\cT \MB), \\
    \Hom_{\CC(\OO)}(\PB', \MB) &\cong \Hom_{\mathfrak Q(\OO)}(\cT \PB', \cT \MB),
\end{align*}
since $\MB^{\SL_2(\Qp)} = (\MB)_{\SL_2(\Qp)} = (\PB')_{\SL_2(\Qp)} = 0$. Let $\md := \Hom_{\CC(\OO)}(\PB', \MB)$, which is a finitely generated right $\EB'$-module by Corollary \ref{cor:MBfg}. Moreover, we have
\[
\End_{\EB'}(\md) \cong \End_{\mathfrak Q(\OO)}(\cT \MB)\cong \End_{\CC(\OO)}(\MB)
\]
by Proposition \ref{prop:CEeuqiv} (2) and the isomorphism above. Theorem \ref{thm:main} implies that the conditions of Lemma \ref{lem:Endfg} are satisfied with $E=\EB'$ and $Z=\ZB'$. Thus $\End_{\CC(\OO)}(\MB)$ and its centre are  finite over $\ZB'$, and thus finite  over $R^{\ps, \zeta\varepsilon}_{\tr\rhobar_{\BB}}$ by Theorem \ref{thm:main}, and hence noetherian. This implies the corollary since $\End_{\CC(\OO)}(\MB) \cong \EB$  and $Z(\End_{\CC(\OO)}(\MB)) \cong \ZB$ by Proposition \ref{prop:MBcentre}. 
\end{proof}

\begin{remar} 
Since $\ZB$ is noetherian the $\mm$-adic topology coincides with the linearly compact topology in Lemma \ref{ZB_local}.
\end{remar}

Let us spell out the properties of the map $R^{\ps, \zeta\varepsilon}_{\tr\rhobar_{\BB}}\rightarrow \ZB$ constructed in Corollary \ref{cor:allfinite}. Since $\ZB$ acts functorially on every object in $\dualcat(\OO)_{\BB}$, the homomorphism $R^{\ps, \zeta\varepsilon}_{\tr\rhobar_{\BB}}\rightarrow \ZB$ induces a functorial ring homomorphism 
$$c_M: R^{\ps, \zeta\varepsilon}_{\tr\rhobar_{\BB}} \rightarrow \End_{\dualcat(\OO)}(M),$$ 
for every object $M$ in $\dualcat(\OO)$. Since 
$\cV$ is a functor, it induces a ring homomorphism $$\End_{\dualcat(\OO)}(M)\rightarrow \End^{\cont}_{\gal}(\cV(M)),\quad  \varphi\mapsto \cV(\varphi).$$ 
We denote the action of $\gal$ on $\cV(M)$ by $\rho_{\cV(M)}$. Finally, for all $g\in \gal$ we may evaluate the universal pseudorepresentation $T: \gal \rightarrow R^{\ps, \zeta\varepsilon}_{\tr\rhobar_{\BB}}$ at $g\in \gal$ to obtain an element  $T(g)\in R^{\ps, \zeta\varepsilon}_{\tr\rhobar_{\BB}}$.

\begin{prop}\label{explain} For each $M\in \dualcat(\OO)_{\BB}$  and each $g\in \gal$ 
$$\cV(c_M(T(g)))= \rho_{\cV(M)}(g) + \rho_{\cV(M)}(g^{-1}) \zeta\varepsilon(g),$$
in $\End_{\gal}^{\cont}(\cV(M))$.
\end{prop}

\begin{proof} 
Since $g^2 -T(g)g + \zeta\varepsilon(g)=0$ in $R^{\ps, \zeta\varepsilon}_{\tr\rhobar_{\BB}}\br{\gal}/J$, the equality  $T(g)\id = g + \zeta\varepsilon(g) g^{-1}$ holds in that ring. The rest is just unravelling the definitions. 
\end{proof}

Since $P_{\BB}$ is a projective generator for $\dualcat(\OO)_{\BB}$, the functor 
$$N \mapsto \md(N):= \Hom_{\dualcat(\OO)}(P_{\BB}, N)$$
induces an equivalence of categories between $\dualcat(\OO)_{\BB}$ and the category of right pseudo-compact $E_{\BB}$-modules. The inverse functor is given by $\md \mapsto \md \wtimes_{E_{\BB}} P_{\BB}$.

\begin{cor}\label{new_fibre} 
For $N$ in $\dualcat(\OO)_{\BB}$ the following assertions are equivalent: 
\begin{enumerate}
\item there is a surjection $\PB^{\oplus n} \twoheadrightarrow N$ for some $n\ge 1$;
\item $\md(N)$ is a finitely generated $E_{\BB}$-module;
\item $\md(N)$ is a finitely generated $R^{\ps, \zeta\varepsilon}_{\tr\rhobar_{\BB}}$-module;
\item $k\wtimes_{R^{\ps, \zeta\varepsilon}_{\tr\rhobar_{\BB}}} N$ is of finite length in $\dualcat(\OO)$;
\item the cosocle of $N$ in $\dualcat(\OO)$ 
is of finite length.
\end{enumerate} 
The equivalent conditions  hold if $N$ is finitely generated over $\OO\br{H}$ for a compact open subgroup $H$ of $G$. 
\end{cor}

\begin{proof}  
(1) implies (2) since $\md$ is exact. (2) implies (3) by Corollary \ref{cor:allfinite}. Since 
$$ k\wtimes_{R^{\ps, \zeta\varepsilon}_{\tr\rhobar_{\BB}}}\md(N)\cong \md( k\wtimes_{R^{\ps, \zeta\varepsilon}_{\tr\rhobar_{\BB}}}N),$$
and the functor $\md$ is an anti-equivalence (3) implies (4). Let $N \twoheadrightarrow \cosoc(N)$ be the cosocle of $N$. Since the maximal ideal of $R^{\ps, \zeta\varepsilon}_{\tr\rhobar_{\BB}}$ acts trivially on every semi-simple object, the surjection factors through
$$ k\wtimes_{R^{\ps, \zeta\varepsilon}_{\tr\rhobar_{\BB}}} N \twoheadrightarrow \cosoc(N),$$
and so (4) implies (5). If $\cosoc(N)$ is of finite length then there is a surjection $\pi_{\BB}^{\oplus n} \twoheadrightarrow \cosoc(N)$ for some $n\ge1$. Since $\PB$ is projective there is a map $\varphi: \PB^{\oplus n}  \rightarrow N$ lifting $(\PB)^{\oplus n} \twoheadrightarrow \pi_{\BB}^{\oplus n} \twoheadrightarrow \cosoc(N)$. The cokernel of $\varphi$ will have zero cosocle and hence $\varphi$ is surjective, so that (5) implies (1). 

If $N$ is finitely generated over $\OO\br{H}$, which we may assume to be pro-$p$, then $((N/\varpi N)^{\vee})^H$ is a finite dimensional $k$-vector space, and hence the $G$-socle of $N^{\vee}$ is of finite length, which dually implies that (5) holds.  
\end{proof}

Since every irreducible in $\BB$ is admissible, its Pontryagin dual is finitely generated over
 $\OO\br{H}$ for any compact open subgroup $H$ of $G$. It follows from Corollary \ref{new_fibre} (4) 
 that  $k\wtimes_{R^{\ps, \zeta\varepsilon}_{\tr\rhobar_{\BB}}}P_{\BB}$ is also a finitely generated 
 $\OO\br{H}$-module. This implies that the assumptions made in section 4 of \cite{image} are 
 satisfied with the category $\dualcat(\OO)$ in \cite[\S 4]{image} equal to $\dualcat(\OO)_{\BB}$, 
 and we will record some consequences below. 

\subsection{Banach space representations}\label{main_banach}
The category $\aBan_{G, \zeta}(L)$ decomposes into a direct sum of categories \cite[Proposition 5.36]{image}:
\begin{equation}\label{sum0}
\aBan_{G, \zeta}(L) \cong \bigoplus_{\BB \in \Irr_{G, \zeta} / \sim} \aBan_{G, \zeta}(L)_{\BB}
\end{equation}
where the objects of $\aBan_{G, \zeta}(L)_{\BB}$ are those $\Pi$ in $\aBan_{G, \zeta}(L)$ such that for every open bounded $G$-invariant lattice $\Theta$ in $\Pi$ the irreducible subquotients of $\Theta \otimes_{\OO} k$ lie in $\BB$. This condition is equivalent to  require $\Theta^d$ to be an object of $\dualcat(\OO)_{\BB}$.

As in the previous subsection we fix a block $\BB$ consisting of absolutely irreducible representations. 
Let $\Mod^{\fg}_{\EB[1/p]}$ be the category of finitely generated right $\EB[1/p]$-modules. 
The functor 
$$\md: \Ban^{\adm}_{G, \zeta}(L)_{\BB} \rightarrow \Mod^{\fg}_{\EB[1/p]}, \quad \Pi \mapsto \md(\Pi):= \Hom_{\dualcat(\OO)}(\PB,\Theta^d)\otimes_{\OO} L,$$
where $\Theta$ is any open bounded $G$-invariant lattice in $\Pi$, is exact, contravariant and fully faithful by \cite[Lemma 4.45]{image}. Moreover, it induces an anti-equivalence of categories 
\begin{equation}\label{equivalence}
\md:\Ban^{\adm}_{G, \zeta}(L)_{\BB}^{\fl} \overset{\cong}{\longrightarrow} \Mod^{\fl}_{\EB[1/p]},
\end{equation}
where the superscript $\fl$ indicates the subcategories of objects of finite length in the respective categories, see \cite[Theorem 4.34]{image}. We write \textit{anti-equivalence} instead of \textit{equivalence} to indicate that $\md$ is contravariant. 

If $\mm$ is a maximal ideal of $R^{\ps, \zeta\varepsilon}_{\tr\rhobar_{\BB}}[1/p]$ then we let $\Ban^{\adm}_{G, \zeta}(L)_{\BB, \mm}^{\fl}$ be the full subcategory of 
$\Ban^{\adm}_{G, \zeta}(L)$ consisting of finite length Banach space representations, 
which are killed by some power of $\mm$. The functor $\md$ induces an anti-equivalence between 
this category and the category of $\EB[1/p]$-modules of finite length, which are killed by a power of 
$\mm$. Chinese remainder theorem, see \cite[Theorem 4.36]{image}, implies that 
we have an equivalence of categories 
\begin{equation}\label{equivalence2}
\Ban^{\adm}_{G, \zeta}(L)_{\BB}^{\fl}\cong \bigoplus_{\mm \in \mSpec R^{\ps, \zeta\varepsilon}_{\tr\rhobar_{\BB}}[1/p]} \Ban^{\adm}_{G, \zeta}(L)_{\BB, \mm}^{\fl}.
\end{equation}

\begin{cor} 
If $\Pi_1\in \Ban^{\adm}_{G, \zeta}(L)_{\BB, \mm_1}^{\fl}$ and  $\Pi_2\in \Ban^{\adm}_{G, \zeta}(L)_{\BB, \mm_2}^{\fl}$ for distinct maximal ideals $\mm_1$ and $\mm_2$ of $R^{\ps, \zeta\varepsilon}_{\tr\rhobar_{\BB}}[1/p]$, then the Yoneda $\Ext^i(\Pi_1, \Pi_2)$ computed in  $\Ban^{\adm}_{G, \zeta}(L)$ vanish for all $i\ge 0$. 
\end{cor}

\begin{proof} 
It follows from \eqref{equivalence2} that the assertion holds for the Yoneda $\Ext$ groups computed in  $\Ban^{\adm}_{G, \zeta}(L)_{\BB}^{\fl}$. It follows from \cite[Proposition 4.46, Corollary  4.48]{image} that these coincide with Yoneda $\Ext$ groups computed in $\Ban^{\adm}_{G, \zeta}(L)_{\BB}$, which is a direct summand of $\Ban^{\adm}_{G, \zeta}(L)$, see \eqref{sum0}.
\end{proof} 

We will determine the set of isomorphism classes $\Irr(\mm, L')$ of irreducible objects in $\Ban^{\adm}_{G, \zeta}(L')_{\BB, \mm}^{\fl}$ for a sufficiently large finite extension $L'$ of $L$. Recall that $\Pi\in \Ban^{\adm}_{G, \zeta}(L)$ is \textit{absolutely irreducible}, if $\Pi\otimes_L L'$ is irreducible in $\Ban^{\adm}_{G, \zeta}(L')$ for all finite extensions $L'$ of $L$. It follows from \eqref{equivalence} that for such $\Pi$ Schur's lemma holds, so that $\End^{\cont}_G(\Pi)=L$. This result is also proved in \cite{DS} in much more general setting. It follows from \eqref{equivalence} that irreducibles in $\Ban^{\adm}_{G, \zeta}(L)_{\BB, \mm}$ correspond to irreducible modules of the algebra $E_{\BB} \otimes_{R^{\ps, \zeta\varepsilon}_{\tr\rhobar_{\BB}}} \kappa(\mm)$. Corollary \ref{cor:allfinite} implies that this algebra is finite
dimensional over $\kappa(\mm)$, thus $\Irr(\mm, L')$ is finite for every finite extension $L'$ of $L$, and  there is a finite extension $L'$ of $L$ such that all $\Pi$ in   $\Irr(\mm, L')$ are absolutely irreducible. 

\begin{prop}\label{action_centre} 
Let $L'$ be a finite extension of $L$ and let $\Pi$ be absolutely irreducible in $\Ban^{\adm}_{G, \zeta}(L')_{\BB}$. Let $c_{\Pi}: R^{\ps, \zeta\varepsilon}_{\tr\rhobar_{\BB}}\rightarrow L'$ be the composition: 
$$ c_{\Pi}: R^{\ps, \zeta\varepsilon}_{\tr\rhobar_{\BB}}\rightarrow Z_{\BB}\rightarrow \End^{\cont}_G(\Pi)=L'.$$
Then one of the following holds:
\begin{enumerate} 
\item if $\Pi$ is a subquotient of $(\Ind_B^G \psi_1\otimes \psi_2 \varepsilon^{-1})_{\cont}$ for some 
unitary characters $\psi_1, \psi_2: \Qp^{\times} \rightarrow (L')^{\times}$ then $T_{c_{\Pi}}= \psi_1+\psi_2$;
\item otherwise, $\cV(\Pi)$ is a $2$-dimensional absolutely irreducible $L'$-representation of
$\gal$, $\det \cV(\Pi)= \zeta \varepsilon$ and $T_{c_{\Pi}}= \tr \cV(\Pi)$;
\end{enumerate}
where $T_{c_{\Pi}}$ is the specialization of the universal pseudorepresentation 
$T: \gal\rightarrow R^{\ps, \zeta\varepsilon}_{\tr\rhobar_{\BB}}$ along $c_{\Pi}$.
\end{prop}

\begin{proof} 
Let $\Psi$ be the unitary principal series representation in (1). If $\psi_1\neq \psi_2\varepsilon^{-1}$ then $\Psi^{\SL_2(\Qp)}=0$ and by looking at its reduction modulo $p$ one may conclude that $\Psi$ is absolutely irreducible. If $\psi_1= \psi_2\varepsilon$ then $\Psi$ is a non-split extension 
$$0 \rightarrow \psi_1\circ \det \rightarrow \Psi \rightarrow \widehat{\Sp} \otimes \psi_1\circ\det \rightarrow 0,$$
where $\widehat{\Sp}$ is the universal unitary completion of the smooth Steinberg representation. This representation is absolutely irreducible, since its mod $p$ reduction is. In both cases $\End^{\cont}_G(\Psi)=L'$, thus $R^{\ps, \zeta\varepsilon}_{\tr\rhobar_{\BB}}$ acts on all irreducible subquotients of $\Psi$ via the same homomorphism $c_{\Psi}$. Since $\cV(\Psi)= \psi_2$, regarded as representation of $\gal$ via the class field theory, $g + \varepsilon\zeta(g) g^{-1}$ acts on it via the scalar $\psi_2(g)+ \psi_2(g^{-1})\varepsilon\zeta(g) = \psi_2(g) + \psi_1(g)$ for all $g\in \gal$. Proposition \ref{explain} implies that the specialization of $T$ at $c_{\Psi}$ is is the character $\psi_1+\psi_2$. 

If we are not in part (1) then \cite[Corollary 1.2, Theorem 1.9]{CDP} imply that $\cV(\Pi)$ is absolutely irreducible $2$-dimensional and $\det \cV(\Pi)= \zeta \varepsilon$. A calculation with $2\times 2$-matrices implies that  $g+ \zeta\varepsilon(g) g^{-1}$ acts on $\cV(\Pi)$ by a scalar $(\tr \cV(\Pi))(g)$. Proposition \ref{explain} implies that the specialization of $T$ at $c_{\Psi}$ is is the character $\tr \cV(\Pi)$. 
\end{proof}

\begin{cor}\label{reducible_irr} 
Let $L'$ be a finite extension of $L$ and let $x: R^{\ps, \zeta\varepsilon}_{\tr\rhobar_{\BB}}\rightarrow L'$ be an $\OO$-algebra homomorphism. If $T_x= \psi_1+ \psi_2$ for characters $\psi_1, \psi_2:\gal \rightarrow (L')^{\times}$ then one of the following holds: 
\begin{enumerate} 
\item if $\psi_1\psi_2^{-1}=\Eins$ then $\Irr(\mm_x, L')= \{ (\Ind_B^G \Eins\otimes  \varepsilon^{-1})_{\cont}\otimes \psi_1\circ \det\}$;
\item if $\psi_1 \psi_2^{-1}= \varepsilon^{\pm 1}$ then $\Irr(\mm_x, L')= \{  \Eins, \widehat{\Sp}, (\Ind_B^G \varepsilon\otimes  \varepsilon^{-1})_{\cont}\} \otimes \psi\circ \det$;
\item if $\psi_1 \psi_2^{-1}\neq \varepsilon^{\pm 1}, \Eins$ then  
$$\Irr(\mm_x, L')= \{ (\Ind_B^G \psi_1\otimes \psi_2 \varepsilon^{-1})_{\cont}, (\Ind_B^G \psi_2\otimes \psi_1 \varepsilon^{-1})_{\cont}\},$$
\end{enumerate}
where we consider $\psi_1$ and $\psi_2$ as unitary characters of $\Qp^{\times}$ via the class field theory and $\psi$ in (2) is either $\psi_1$ or $\psi_2$.
\end{cor} 

\begin{proof} 
We have explained in the course of the proof of Proposition \ref{action_centre} that the representations listed above are absolutely irreducible and are contained in $\Irr(\mm_x, L')$. Moreover, using the functor of ordinary parts one may show that they are pairwise non-isomorphic. 

We will show that the list is exhaustive. We may enlarge $L'$ so that all $\Pi \in  \Irr(\mm_x, L')$ are absolutely irreducible. Since $c_{\Pi}=x$ we cannot 
be in part (2) of Proposition \ref{action_centre} thus we must be in part (1) and hence $\Pi$ is already in our list.  
\end{proof}

\begin{prop}\label{irreducible_irr} 
Let $L'$ be a finite extension of $L$ and let 
$x: R^{\ps, \zeta\varepsilon}_{\tr\rhobar_{\BB}}\rightarrow L'$ be an $\OO$-algebra homomorphism. 
If $T_x =\tr \rho$, where $\rho: \gal\rightarrow \GL_2(L')$ is absolutely irreducible, 
then $\Irr(\mm_x, L')=\{\Pi\}$ with $\Pi$ absolutely irreducible non-ordinary and $\cV(\Pi)\cong \rho$. 
\end{prop}

\begin{proof} It follows from \cite[Theorem 1.1]{CDP} that such $\Pi$ exists. We may enlarge $L'$ so that all $\Pi' \in  \Irr(\mm_x, L')$ are absolutely irreducible. Since $c_{\Pi'}=x$ we cannot be in part (1) of Proposition \ref{action_centre} thus we must be in part (2) and $\tr \cV(\Pi)=\tr \cV(\Pi')$. Since both $\cV(\Pi')$ and $\cV(\Pi)$ are absolutely irreducible, we deduce that $\cV(\Pi)\cong \cV(\Pi')$ and \cite[Theorem 1.8]{CDP} implies that $\Pi\cong \Pi'$.  
\end{proof}

\subsection{The centre} \label{sec_main_centre}
We fix a block $\BB$ as in the previous section and will explore the relation between  $R^{\ps, \zeta\varepsilon}_{\tr\rhobar_{\BB}}$ and $\ZB$. So far we have constructed a finite map
\begin{equation}\label{map_centre}
R^{\ps, \zeta\varepsilon}_{\tr\rhobar_{\BB}}\rightarrow \ZB'\twoheadrightarrow \ZB,
\end{equation}
Theorem \ref{thm:main}, Corollary \ref{surjectiveZB}. 
We have shown in Corollary \ref{normal} that $R^{\ps, \zeta\varepsilon}_{\tr\rhobar_{\BB}}[1/p]$ is normal and we know by \cite[Theorem 2.1]{che_unpublished} that $R^{\ps, \zeta\varepsilon}_{\tr\rhobar_{\BB}}[1/p]$ is equidimensional and the locus corresponding to absolutely irreducible pseudorepresentations is Zariski dense in $\Spec R^{\ps, \zeta\varepsilon}_{\tr\rhobar_{\BB}}[1/p]$. 

\begin{prop}\label{VS} Let $R^{\square, \zeta\varepsilon}_{\rhobar_{\BB}}$ be the 
universal framed deformation ring of $\rhobar_{\BB}$ with fixed determinant $\zeta\varepsilon$, let $S$ be its maximal $\OO$-torsion free quotient and let $V_S$ be a free $S$-module of rank $2$ with $\gal$-action induced by the universal deformation $\gal\rightarrow \GL_2(  R^{\square, \zeta\varepsilon}_{\rhobar_{\BB}})\twoheadrightarrow \GL_2(S)$. There is $N$ in $\dualcat(\OO)$ with a continuous action of $S$, which commutes with the action of $G$, such that we have an isomorphism of $S[\gal]$-modules $\cV(N)\cong V_{S}$.
\end{prop}

\begin{proof}
If $x\in \mSpec S[1/p]$ then the specialization of $V_S$ at $x$ lies in the image of $\cV$ by \cite[Theorem 10.1]{CDP2}. Since $S[1/p]$ is reduced (Propositions \ref{normal_framed}, \ref{go_char}) and Jacobson, such points will be dense, and the existence of such $N$ follows from \cite[Theorem II.3.3]{MR2642409}.
\end{proof}

The subscript $\tf$ will indicate the maximal $\OO$-torsion free quotient.

\begin{thm}\label{main2} 
The surjection $R^{\ps, \zeta\varepsilon}_{\tr\rhobar_{\BB}}\br{\gal}/ J \twoheadrightarrow \End_{\EB'}^{\cont}(\cV(\PB'))$ in Theorem \ref{thm:main} identifies $\End_{\EB'}^{\cont}(\cV(\PB'))$ with $(R^{\ps, \zeta\varepsilon}_{\tr\rhobar_{\BB}}\br{\gal}/ J)_{\tf}$. In particular,  \eqref{map_centre} induces an isomorphism 
\begin{equation}\label{invert_p} 
R^{\ps, \zeta\varepsilon}_{\tr\rhobar_{\BB}}[1/p]\overset{\cong}{\longrightarrow} \ZB'[1/p].
\end{equation}
Moreover, if $p\neq 2$ then $\ZB'= (R^{\ps, \zeta\varepsilon}_{\tr\rhobar_{\BB}})_{\tf}$ and if $p=2$ then the cokernel of \eqref{map_centre} is killed by $2$.
\end{thm}

\begin{proof} 
As already explained in the proof of Proposition \ref{HMP}, 
projective objects in $\dualcat(\OO)$ are also projective in the category 
of compact $\OO\br{K'}$-modules, where $K'$ is an open pro-$p$ subgroup of $\SL_2(\Qp)$ intersecting $Z$ trivially, and thus are $\OO$-torsion free. 
Hence, $\PB'$ is $\OO$-torsion free. Since $\EB'$ and $\ZB'$ act faithfully on $\PB'$ we deduce that both rings are $\OO$-torsion free. Since $\End_{\EB'}^{\cont}(\cV(\PB'))$ is either $(\EB')^{\op}$ or $M_2(\EB')$ by Proposition \ref{prop:surjgalE}, we deduce that the map in Theorem \ref{thm:main} factors through 
\begin{equation}\label{inj_tf}
(R^{\ps, \zeta\varepsilon}_{\tr\rhobar_{\BB}}\br{\gal}/ J)_{\tf}\twoheadrightarrow \End_{\EB'}^{\cont}(\cV(\PB')).
\end{equation}
If $a$ lies in the kernel of this map, then it will kill $\cV(\PB')$ and hence $\md\wtimes_{\EB'} \cV(\PB')$ for all compact right $\EB'$-modules. Thus $a$ will kill all objects in the essential image of $\cV$. Thus it will also kill the representation $V_S$ defined in Proposition \ref{VS}.

It follows from Propositions \ref{normal_framed}, \ref{go_char} that 
the ring $R^{\square, \zeta\varepsilon}_{\rhobar_{\BB}}[1/p]$ is normal 
and the absolutely irreducible locus is dense in $\Spec R^{\square, \zeta\varepsilon}_{\rhobar_{\BB}}[1/p]$. Corollary \ref{faithful} implies that $(R^{\ps, \zeta\varepsilon}_{\tr\rhobar_{\BB}}\br{\gal}/ J)_{\tf}$
acts faithfully on $V_S$, hence $a=0$ and \eqref{inj_tf} is injective. 

The assertions about the centre follow from Proposition \ref{Ztf}.
\end{proof}

We immediately obtain. 

\begin{cor}\label{tf_red} 
$\ZB'$ is a complete local noetherian $\OO$-algebra with residue field 
$k$. It is $\OO$-torsion free and $\ZB'[1/p]$ is normal. 
\end{cor}

\begin{cor}\label{Zequal} 
$\ZB= \ZB'$. 
\end{cor} 

\begin{proof} 
Since $\ZB$ acts faithfully on $\PB$ it is $\OO$-torsion free. Thus it is enough to show that the surjection $\ZB'\twoheadrightarrow \ZB$, see Corollary \ref{surjectiveZB}, induces an isomorphism after inverting $p$. Since $\ZB'[1/p]$ is reduced by Corollary \ref{tf_red}, it is enough to show that $\mSpec \ZB[1/p]$ contains a subset $\Sigma$ of $\mSpec \ZB'[1/p]$, which is dense in $\Spec \ZB'[1/p]$. We may take $\Sigma$ to be the absolutely irreducible locus in $\mSpec R^{\ps, \zeta \varepsilon}_{\tr \rhobar_{\BB}}[1/p]$, as it is dense in $\Spec R^{\ps, \zeta \varepsilon}_{\tr \rhobar_{\BB}}[1/p]$ by \cite[Theorem 2.1]{che_unpublished} and lies in $\mSpec \ZB[1/p]$ by the main result of \cite{CDP}.
\end{proof} 

\begin{cor}\label{irr_irr} 
Let $L'$ be a finite extension of $L$ and let $x: R^{\ps, \zeta\varepsilon}_{\tr\rhobar_{\BB}} \rightarrow L'$ be an $\OO$-algebra homomorphism. If the specialization of the universal pseudodeformation $T:\gal\rightarrow R^{\ps, \zeta\varepsilon}_{\tr\rhobar_{\BB}}$ 
at $x$ is not of the form $\psi+ \psi\varepsilon$ then $\Ban^{\adm}_{G, \zeta}(L')^{\fl}_{\BB, \mm_x}$ is equivalent to the category of modules of finite  length over the completion of $(R^{\ps, \zeta\varepsilon}_{\tr\rhobar_{\BB}}\br{\gal}/ J)\otimes_{\OO} L'$ at $\mm_x$. 

Moreover, if $T_x =\tr \rho$, where $\rho: \gal\rightarrow \GL_2(L')$ is absolutely irreducible, then $\Ban^{\adm}_{G, \zeta}(L')^{\fl}_{\BB, \mm_x}$ is equivalent to the category of modules of finite  length over the deformation ring $R^{\zeta\varepsilon}_{\rho}$, which parameterizes the deformations of $\rho$ with determinant $\zeta \varepsilon$ to local artinian $L'$-algebras. 

In particular, if $\Pi'\in \Ban^{\adm}_{G, \zeta}(L')$ is killed by $\mm_x$ then $\Pi'$ is isomorphic to a direct sum of finitely many copies of $\Pi$ in Proposition \ref{irreducible_irr}.
\end{cor} 

\begin{proof} 
After extending scalars we may assume that $L=L'$. If $T_x\neq \psi+\psi\varepsilon$ for any character $\psi$ then it follows from Corollaries \ref{reducible_irr}, \ref{irreducible_irr} that $\Irr(\mm_x, L')$ does not contain characters. We may apply \cite[Theorem 4.36]{image} to deduce that  $\Ban^{\adm}_{G, \zeta}(L')^{\fl}_{\BB, \mm_x}$ is anti-equivalent to the category of $\EB'\otimes_R \widehat{R}_{\mm_x}$-modules of finite length, where $R=R^{\ps, \zeta\varepsilon}_{\tr\rhobar_{\BB}}$. Theorem \ref{main2} implies that this ring coincides with the completion of $(R^{\ps, \zeta\varepsilon}_{\tr\rhobar_{\BB}}\br{\gal}/ J)[1/p]$ at $\mm_x$. 

Let us assume that $T_x= \tr \rho$ with $\rho$ absolutely irreducible. Then $\Irr(\mm_x, L')=\{\Pi\}$ with $\Pi$ absolutely irreducible by Proposition \ref{irreducible_irr} and $\cV(\Pi)\cong \rho$. It follows from \cite[\S 4.1, 4.2]{che_durham} that $(R\br{\gal}/ J)\otimes_R \widehat{R}_{\mm_x}$ is an Azumaya algebra over $\widehat{R}_{\mm_x}$. Since  $\rho$ is an absolutely irreducible $2$-dimensional module of $(R\br{\gal}/ J)\otimes_R \kappa(x)$, we conclude that $(R\br{\gal}/ J)\otimes_R \kappa(x) = M_2(\kappa(x))$ and thus $(R\br{\gal}/ J)\otimes_R \widehat{R}_{\mm_x}$ is isomorphic to the ring of $2\times 2$-matrices over $\widehat{R}_{\mm_x}$. Since  $M_2(\widehat{R}_{\mm_x})$ is  Morita equivalent to $\widehat{R}_{\mm_x}$, 
which is isomorphic to $R^{\zeta\varepsilon}_{\rho}$ by  \cite[Lemma 2.3.3 and Proposition 2.3.5]{kisin_moduli}, we obtain the first assertion. 
 
In particular, the full subcategory of $\Ban^{\adm}_{G, \zeta}(L')^{\fl}_{\BB, \mm_x}$ consisting of representations killed by $\mm_x$ is equivalent to the category of finite dimensional vector spaces over $L'$, and hence the last assertion follows. 
\end{proof} 
\subsection{Complements} We will prove Theorem \ref{intro_finite} stated in the introduction. 
Let $\BB$ be an arbitrary block, so that we do not assume that $\BB$ contains an absolutely irreducible 
representation. 

If  $\pi_1, \pi_2\in \Irr_{G, \zeta}(k)$ then it follows from \cite[Proposition 5.11]{image} that  there is a finite extension $k'$ of $k$ such that $\pi_1\otimes_k k'$ is a finite direct sum of absolutely irreducible representations; then $\pi_2\otimes_k k'$ is a finite direct sum of irreducible representations, each of them occurring with multiplicity one. Proposition 5.33 of \cite{image} implies that 
\begin{equation}\label{base_change_ext}
\Ext^1_{k[G], \zeta}(\pi_1, \pi_2)\otimes_k k'\cong \Ext^1_{k'[G],\zeta}(\pi_1\otimes_k k', \pi_2\otimes_k k').
\end{equation}

If $\Ext^1_{k[G], \zeta}(\pi_1, \pi_2)\neq 0$ then it follows from \eqref{base_change_ext} that 
there are irreducible summands $\pi_1'$ of $\pi_1\otimes_k k'$ and $\pi_2'$ of $\pi_2 \otimes_k k'$ 
such that $\Ext^1_{k'[G], \zeta}(\pi_1', \pi_2')\neq 0$. Since $\pi_1'$ is absolutely irreducible we conclude by inspecting the list of blocks in Section \ref{section:blocks} that $\pi_2'$ is absolutely irreducible,  thus if $\BB$ is the block containing $\pi_1$ then $\pi\otimes_k k'$ is a finite direct sum of absolutely irreducible representations for all $\pi\in \BB$.

Let $L'$ be a finite extension of $L$ with ring of integers $\OO'$ and residue field $k'$. 
If $\pi'_1, \pi_2'\in \Irr_{G,\zeta}(k')$ are absolutely irreducible then it follows from 
\cite[Proposition 5.11]{image} that there exist unique $\pi_1, \pi_2\in \Irr_{G, \zeta}(k)$ such that 
$\pi_1'$ is a direct summand of $\pi_1\otimes_k k'$ and  $\pi_2'$ is a direct summand 
of $\pi_2\otimes_k k'$. It follows from \eqref{base_change_ext} that if $\pi_1'$ and $\pi_2'$ 
lie in the same block in $\Mod^{\lfin}_{G, \zeta}(\OO')$ then $\pi_1$ and $\pi_2$ lie in the 
same block in $\Mod^{\lfin}_{G, \zeta}(\OO)$.

Thus if $\pi_1\in \BB$ and we let $\BB_1, \ldots, \BB_r$ be the blocks of irreducible subquotients 
of $\pi_1\otimes_k k'$ in $\Mod^{\lfin}_{G, \zeta}(\OO')$ and let $\BB\otimes_k k'$ be the set of isomorphism 
classes of irreducible subquotients of $\pi\otimes_k k'$ for all $\pi \in \BB$ then 
$$\BB\otimes_k k'= \bigcup_{i=1}^r \BB_i.$$
It follows from \cite[Corollary 5.40]{image} that   
$\PB\otimes_{\OO} \OO'\cong \prod_{i=1}^r P_{\BB_i}$ and 
$$\EB\otimes_{\OO} \OO'\cong \End_{\dualcat(\OO')}(\PB\otimes_{\OO} \OO')\cong \prod_{i=1}^r E_{\BB_i}.$$
Since the blocks $\BB_i$ contain only absolutely irreducible representations, it follows from
Corollary \ref{cor:allfinite} that $\EB\otimes_{\OO} \OO'$ is a finite module  over its centre 
$Z(\EB\otimes_{\OO} \OO')$ and 
$$Z(\EB)\otimes_{\OO} \OO'\cong Z(\EB\otimes_{\OO} \OO')\cong \prod_{i=1}^r Z_{\BB_i}$$
is noetherian, see the argument in the proof of \cite[Lemma 4.14]{DPS} for the first isomorphism.  Since $\OO'$ is a finite free $\OO$-module, this implies that 
$Z(\EB)$ is noetherian, and $\EB$ is a finitely generated $\ZB$-module, which finishes the proof 
of Theorem \ref{intro_finite}.

\section{Application to Hecke eigenspaces} 

Let $R$ be a linearly compact  local $R^{\ps, \zeta\varepsilon}_{\tr\rhobar_{\BB}}$-algebra with residue field $k$; we do not assume that $R$ is noetherian. 
If $x:R\rightarrow \Qpbar$ is an $\OO$-algebra homomorphism then we denote by $T_x$ the specialization of the universal pseudorepresentation $T: \gal \rightarrow R^{\ps, \zeta\varepsilon}_{\tr\rhobar_{\BB}}$ along $R^{\ps, \zeta\varepsilon}_{\tr\rhobar_{\BB}}\rightarrow R \overset{x}{\rightarrow} \Qpbar$. 

Let $M$ be an object of $\dualcat(\OO)_{\BB}$, which we assume to be $\OO$-torsion free. Then $\Pi(M):=\Hom_{\OO}^{\cont}(M, L)$ is a unitary $L$-Banach space representation of $G$.

We assume that we are given a continuous action of $R$ on $M$, which commutes with the action of $G$, such that the following hold: 
\begin{itemize}
\item the action of $R$ on $M$ is faithful;
\item the two actions of $R^{\ps, \zeta\varepsilon}_{\tr\rhobar_{\BB}}$ on $M$ induced by the maps 
$$R^{\ps, \zeta\varepsilon}_{\tr\rhobar_{\BB}}\rightarrow R, \quad R^{\ps, \zeta\varepsilon}_{\tr\rhobar_{\BB}}\rightarrow Z_{\BB}$$ 
coincide;
\item $M$ is a finitely generated $R\br{K}$-module. 
\end{itemize} 

\begin{thm}\label{lue_pan} 
Let $x: R\rightarrow \Qpbar$ be an $\OO$-algebra homomorphism and let $\Pi(M)[\mm_x]$ be the subspace of $\Pi(M)$ annihilated by the kernel of $x$. Then under the above assumptions $\Pi(M)[\mm_x]$ is non-zero and is of finite length in $\Ban^{\adm}_{G, \zeta}(L)$. Moreover, 
\begin{itemize}
\item if $T_x$ is the trace of an absolutely irreducible Galois representation defined over $\kappa(x)$ then $$\Pi(M)[\mm_x]\cong \Pi^{\oplus m},$$ for some multiplicity $m>0$, where $\Pi$ is an absolutely irreducible non-ordinary $\kappa(x)$-Banach space representation of $G$ satisfying $\tr \cV(\Pi)= T_x$. 
\item if $T_x$ is the trace of a reducible Galois representation then (after a possible extension of scalars) all the irreducible subquotients of $\Pi(M)[\mm_x]$ occur as subquotients of a direct sum of unitary parabolic induction:
$$ (\Ind_B^G \psi_1 \otimes \psi_2 \varepsilon^{-1})_{\cont} \oplus (\Ind_B^G \psi_2 \otimes \psi_1 \varepsilon^{-1})_{\cont},$$
where $\psi_1, \psi_2:\gal \rightarrow \kappa(x)^{\times}$ are characters such that $T_x= \psi_1 + \psi_2$. 
\end{itemize}
\end{thm} 

\begin{proof} 
Since $P_{\BB}$ is a projective generator for $\dualcat(\OO)_{\BB}$ the functor 
$N \mapsto \md(N):= \Hom_{\dualcat(\OO)}(P_{\BB}, N)$ induces an equivalence of 
categories between $\dualcat(\OO)_{\BB}$ and the category of right pseudo-compact
 $E_{\BB}$-modules. The inverse functor is given by $\md \mapsto \md \wtimes_{E_{\BB}} P_{\BB}$. 
 In particular, the assumption that $R$ acts faithfully on $M$ implies that $R$ acts faithfully on 
 $\md(M)$. 

We claim that $\md(M)$ is a finitely generated $R$-module. Topological Nakayama's lemma 
implies that it is enough to show that $k \wtimes_R \md(M)$ is a finite dimensional $k$-vector space. 
Since $k\wtimes_R \md(M)\cong \md(k \wtimes_R M)$, it is enough to show that $k\wtimes_R M$ is 
of finite length in $\dualcat(\OO)$.  Since by assumption $M$ is a finitely generated $R\br{K}$-module, 
$k\wtimes_R M$ is a finitely generated $k\br{K}$-module. Since by assumption the actions of 
$R^{\ps, \zeta\varepsilon}_{\tr\rhobar_{\BB}}$ on $M$ induced by $Z_{\BB}$ and by $R$ coincide 
we deduce that the maximal ideal of $R^{\ps, \zeta\varepsilon}_{\tr\rhobar_{\BB}}$ annihilates $k\wtimes_R M$. Corollary \ref{new_fibre} (4) applied to $N=k\wtimes_R M$ implies the claim.

Since $\md(M)$ is a finitely generated and faithful $R$-module, its localization $\md(M)_{\mm_x}$ is a finitely generated faithful $R_{\mm_x}$-module. If $\md(M)\otimes_R \kappa(x)=0$ then $\md(M)_{\mm_x}=0$ by Nakayama's lemma and, since $R_{\mm_x}$ acts faithfully, $R_{\mm_x}=0$, and hence $\kappa(x)=0$, giving a contradiction. In particular, $\md(M)\otimes_{R}\kappa(x)$ is a non-zero, finite dimensional $\kappa(x)$-vector space. Since $R$ is a compact $\OO$-module and $\kappa(x)$ is a subfield of $\Qpbar$, $\kappa(x)$ is a finite extension of $L$ and the image of $R$ is contained in the ring of integers of $\kappa(x)$.  

Let $Q$ be the maximal $\OO$-torsion free Hausdorff quotient of $M/\mm_x M$. It follows from \cite[Proposition 1.3]{st_iw} that $\Pi(Q)$ is a closed subspace of $\Pi(M)$, which then implies that $\Pi(Q)=\Pi(M)[\mm_x]$. It follows from the equivalence of categories explained above that $\md(Q)$ is isomorphic to the image of $\md(M)$ in $\md(M)\otimes_{R}\kappa(x)$. In particular, $Q$ and thus $\Pi(Q)$ are non-zero. 

The last two assertions follow from the anti-equivalence \eqref{equivalence}, Corollary \ref{irr_irr}, Corollary \ref{reducible_irr}.
\end{proof} 

\begin{remar} 
If $M$ is finitely generated as $\OO\br{K}$-module then the argument in the proof of Theorem \ref{lue_pan} shows that $\md(M)$ is finitely generated $R^{\ps, \zeta\varepsilon}_{\tr\rhobar_{\BB}}$-module and, since $R$ acts faithfully on $\md(M)$, $R$ is a finitely generated $R^{\ps, \zeta\varepsilon}_{\tr\rhobar_{\BB}}$-module, and hence is noetherian.
\end{remar}

The result allows to remove the restrictions imposed on the Galois representation 
$\rhobar_{\mm, p}$ in  \cite[Corollary 6.3.6]{luepan}, by taking $M$ to be the Pontryagin dual of representation denoted by $\tilde{H}^1(K^p, E/\OO)_{\mm, \zeta'}$ in \cite[Theorem 6.3.5]{luepan}, and taking $R$ to be the closure of the subring generated by the Hecke operators in $\End^{\cont}_{\OO}(M)$. Since \cite[Corollary 6.3.6]{luepan} is the only place, where the restriction on $p$ is used, the proof 
of \cite[Theorem 6.4.7]{luepan} goes through without a change to give the following result: 
\begin{thm}[Lue Pan +$\varepsilon$]\label{application} Let $\rho: \Gal(\overline{\mathbb{Q}}/\mathbb{Q})\rightarrow \GL_2(L)$ be 
pro-modular and absolutely irreducible. If $\rho$ is unramified outside finitely many places and $\rho|_{\gal}$ is Hodge--Tate with weights $0,0$ then 
$\rho$ is associated to a weight $1$ modular form. 
\end{thm}

The condition pro-modular means that the Hecke eigenvalues associated to $\rho$ appear in 
completed cohomology, see \cite[Definition 6.1.2]{luepan} for the precise definition. The original 
theorem in Lue Pan's paper had to additionally assume that if $p$ is $2$ or $3$ then 
$(\rhobar|_{\gal})^{\mathrm{ss}}$ is not isomorphic to $\chi \oplus \chi \omega$
 for any character $\chi: \gal\rightarrow k^{\times}$ .

\appendix

\section{Normality of $R^{\ps}[1/p]$}\label{appendix}
Let $\GG$ be a pro-finite group satisfying Mazur's finiteness condition at $p$: the group of continuous group homomorphisms $\Hom^{\cont}_{\mathrm{grp}}(\GG', \Fp)$ is finite for every open subgroup $\GG'$ of $\GG$. Let $\rhobar: \GG \rightarrow \GL_d(k)$ be a continuous semi-simple representation, such that all the irreducible summands of $\rhobar$ are absolutely irreducible. Let $\psi: \GG\rightarrow \OO^{\times}$ be a character lifting $\det \rhobar$. Let $\bar{D}: k[\GG]\rightarrow k$ be the pseudorepresentation associated to $\rhobar$ in \cite{che_durham}, so that $\bar{D}(1+t g)= \det(1 + t \rhobar(g))$, for all $g\in \GG$. We may consider the framed deformation ring $R^{\square}_{\rhobar}$, its quotient $R^{\square, \psi}_{\rhobar}$ parameterizing framed deformations of $\rhobar$ with the determinant equal to $\psi$, the universal deformation ring 
$R^{\ps}$ of $\bar{D}$, and its quotient $R^{\ps, \psi}$ parameterizing deformations of $\bar{D}$ with determinant $\psi$. This last ring is constructed as follows if $D^u:\GG\rightarrow R^{\ps}$ is the universal deformation of $\bar{D}$ then for each $g\in \GG$, $D^u(1 +t g)= a_0(g)+\ldots +a_d(g)t^d$, with $a_i(g)\in R^{\ps}$ and $R^{\ps, \psi}$ is the quotient of $R^{\ps}$ by the ideal generated $\psi(g) a_d(g)-1$ for all $g\in \GG$. The finiteness condition on $\GG$ ensures that all these rings are noetherian. The characteristic polynomial of the 
universal framed deformations of $\rhobar$ induces maps $R^{\ps}\rightarrow R^{\square}_{\rhobar}$, $R^{\ps, \psi}\rightarrow R^{\square, \psi}_{\rhobar}$.

\begin{thm}\label{main_appendix} If $R^{\square}_{\rhobar}[1/p]$ is normal then both $R^{\ps}[1/p]$ and the associated rigid space $(\Spf R^{\ps})^{\rig}$ are normal. 
\end{thm}

We also prove a version of the theorem with a fixed determinant. We apply this theorem to $\GG=\gal$
 to prove that the rings $R^{\ps, \psi}$, $R^{\ps}$ and associated rigid analytic spaces are normal for 
 all $2$-dimensional $\rhobar$. There is essentially one case that one needs to handle, namely $\rhobar= \Eins \oplus \omega$, where $\omega$ is the cyclotomic character modulo $p$, as in the other cases all the rings are regular.  The most tricky cases are when $p=2$ and $p=3$. The case $p=2$ is treated in \cite{CDP2}. We deal with the case $p=3$ using the work of B\"{o}ckle \cite{boeckle}.

The argument of \cite{CDP2} has been extended by Iyengar in \cite{Iy}, when $\rhobar$ is the trivial $d$-dimensional representation of a Galois group of a $p$-adic field $F$, under the assumption that $F$ contains a primitive $4$-th root of unity if $p=2$. Thus our theorem applies in that setting.\footnote{
This has been further generalized in \cite{BIP} for all $p$-adic fields $F$ and all $\rhobar$.}

We will split the proof into several steps. We start with commutative algebra lemmas and recall that all excellent rings are $G$-rings, \cite[\href{https://stacks.math.columbia.edu/tag/07QS}{Tag 07QS}]{stacks-project}.

\begin{lem}\label{commalg1} 
Let $A$ be a $G$-ring and let $\pp\in \Spec A$. Then $A_{\pp}$ satisfies Serre's condition $(R_i)$ (resp.~$(S_i)$) if and only if the completion $A_{\pp}$ at $\pp$ does. 
\end{lem}

\begin{proof} 
Let $B=A_{\pp}$ and let $\hat{B}$ be the completion of $A_{\pp}$ at $\pp$. Since $A$ is a $G$-ring the fibre rings $\kappa(\qq)\otimes_B \hat{B}$ are regular for all $\qq\in \Spec B$. The assertion follows from \cite[Theorem 23.9]{matsumura}. 
\end{proof}

\begin{lem}\label{commalg2} 
Let $A$ be a complete local noetherian $\OO$-algebra with residue field $k$, $B=A\br{x_1, \ldots, x_r}$, let $\qq \in \Spec B[1/p]$ and $\pp$ the image of $\qq$ in $\Spec A[1/p]$. Then $A_{\pp}$ satisfies Serre's condition $(R_i)$ (resp. $(S_i)$) if and only if  $B_{\qq}$ does. In particular, $A[1/p]$ is normal if and only if $B[1/p]$ is normal.
\end{lem} 

\begin{proof} 
The proof is a variation on \cite[Appendix A]{6auth}. We may assume that $A$ and hence $B$ are $\OO$-torsion free. Let $\pp'\in \Spec A_{\pp} \subset \Spec A$.  We claim that the ring $\kappa(\pp')\otimes_A B$ is regular. By Cohen's structure theorem there is a subring $C\subset A/\pp'$, such that $C$ is formally smooth over $\OO$ and $A/\pp'$ is finite over $C$. Then 
$$\kappa(\pp')\otimes_A B\cong \kappa(\pp') \otimes_{A/\pp'} B/\pp'B.$$ 
Since $A/\pp'$ is finite over $C$ we have
$$B/\pp' B= (A/\pp')\br{x_1,\ldots, x_r}\cong A/\pp' \otimes_C C\br{x_1, \ldots, x_r}.$$  
Thus
$$\kappa(\pp')\otimes_A B\cong \kappa(\pp)\otimes_{Q(C)} Q(C)\otimes_C C\br{x_1, \ldots, x_r},$$
where $Q(C)$ is the quotient field of $C$. Since $C$ is formally smooth over $\OO$, the ring $C\br{x_1, \ldots, x_r}$ is isomorphic to a ring of  formal power series over $\OO$, and thus is regular. Tensoring with $Q(C)$ over $C$ is just localization with respect to the multiplicative set $C\setminus \{0\}$, thus $Q(C)\otimes_C C\br{x_1, \ldots, x_r}$ is regular. Since $Q(C)$ is of characteristic zero the extension $\kappa(\pp)/ Q(C)$ is separable and it follows from \cite[Lemma A.3]{6auth} that $\kappa(\pp)\otimes_{Q(C)} Q(C)\otimes_C C\br{x_1, \ldots, x_r}$ is regular. We deduce that $\kappa(\pp')\otimes_{A_{\pp}} B_{\qq}$ is regular, since it is a localization of $\kappa(\pp')\otimes_A B$ at $\qq$.

It follows from \cite[Theorem 23.9]{matsumura} that $A_{\pp}$ satisfies $(R_i)$ (resp.~$(S_i)$) if and only if $B_{\qq}$ does. To conclude that $A[1/p]$ is normal if and only if $B[1/p]$ is, we only have to show that the map $\Spec B[1/p]\rightarrow \Spec A[1/p]$ is surjective, but this is clear as $(\pp, x_1, \ldots, x_r)$ maps to $\pp$. 
\end{proof}

\begin{lem}\label{commalg3} 
Let $A\rightarrow B$ be a finite \'{e}tale map of local rings. Then $A$ satisfies Serre's condition $(R_i)$ (resp.~$(S_i)$) if and only if $B$ does. 
\end{lem}

\begin{proof} 
If $\pp\in \Spec A$ then the fibre ring $\kappa(\pp)\otimes_A B$ is a finite 
\'{e}tale $\kappa(\pp)$-algebra and hence a product of fields, and thus is regular. The assertion follows from \cite[Theorem 23.9]{matsumura}. 
\end{proof}

\begin{prop}\label{framed1} 
Let $L'$ be a finite extension of $L$ and let $\rho: \GG \rightarrow \GL_n(L')$ 
be a continuous representation with mod $p$ semi-simplification isomorphic to $\rhobar$. If $R^{\square}_{\rhobar}[1/p]$ is normal then the ring $R^{\square}_{\rho}$, representing  the framed deformations of $\rho$ to artinian $L'$-algebras,  is also normal. 
\end{prop}

\begin{proof} 
We may choose a finite extension $L''$ of $L'$ with the ring of integers $\OO''$  and residue field $k''$ such that $\rho\otimes_{L'} L''$ has a $\GG$-invariant $\OO''$-lattice $\Theta$ with $\Theta\otimes_{\OO''} k''\cong \rhobar\otimes_k k''$, see the proof of \cite[Lemma 9.5]{CDP2}. Thus $\Theta$ is a deformation of $\rhobar\otimes_k k''$ to $\OO''$.

It follows from Lemma \ref{commalg3} that $R^{\square}_{\rho}$ is normal if and only if $L''\otimes_{L'} R^{\square}_{\rho}$ is normal. The same argument shows that $L''\otimes_{L} R^{\square}_{\rhobar}[1/p]$ is normal. Moreover, we may identify $L''\otimes_{L'} R^{\square}_{\rho}$ with the framed deformation ring of $\rho\otimes_{L'} L''$ to local artinian $L''$-algebras, and $\OO''\otimes_{\OO} R^{\square}_{\rhobar}$ with the framed deformation ring of $\rhobar\otimes_k k''$ to local artinian $\OO''$-algebras. After these identifications we may assume that $L=L'=L''$, and so $\Theta$ is a deformation of $\rhobar$ to $\OO$, and hence induces an $\OO$-algebra homomorphism $x: R^{\square}_{\rhobar} \rightarrow \OO$.

It follows from \cite[Lemma 2.3.3 and Proposition 2.3.5]{kisin_moduli} that $R^{\square}_{\rho}$ is isomorphic to the completion of $(R^{\square}_{\rhobar})_{\pp}$ at $\pp=\Ker x$. Since $R^{\square}_{\rhobar}[1/p]$ is normal $(R^{\square}_{\rhobar})_{\pp}$ will satisfy $(R_1)$ and $(S_2)$. Lemma \ref{commalg1} implies that the same holds for the completion. Thus $R^{\square}_{\rho}$ is normal by Serre's criterion for normality, \cite[Theorem 23.8]{matsumura}.
\end{proof}

\begin{lem}\label{normal_inv} 
Let $A$ be a normal noetherian ring and let $G$ be a group acting on $A$ by ring automorphisms. Then the subring of $G$-invariants $A^G$ is normal. 
\end{lem}

\begin{proof} 
If $A$ is a domain then the assertion is proved in \cite[Proposition 6.4.1]{BH}. The same proof works in our setting, as we will explain for the lack of a reference. Since $A$ is noetherian and normal, it is a finite product of normal domains. Let $\Frac(A)$ denote its total ring of fractions. Then $\Frac(A)$ is a finite product of fields. The group $G$ acts on $\Frac(A)$ and we have 
\begin{equation}\label{AGcap}
 A^G= A\cap \Frac(A)^G.
 \end{equation}
We claim that $\Frac(A)^G$ is a finite product of fields. The claim implies that $\Frac(A)^G$ is its own ring of fractions. Since $A$ is normal, \eqref{AGcap} implies that $A^G$ is reduced, integrally closed in its ring of fractions and has only finitely many minimal prime ideals, and hence is normal by Lemma 10.37.16 in \cite[\href{https://stacks.math.columbia.edu/tag/037B}{Tag 037B}]{stacks-project}.

To prove the claim, we note that $\Spec \Frac(A)$ consists of finitely many primes and is in bijection with  the set $\mathcal E$ of idempotents $e\in \Frac(A)$, such that $e \Frac(A) e$ is a field. We have $1=\sum_{e\in \mathcal E} e$ and $e e'=0$ if $e\neq e'$. If $e\in \Frac(A)$ is a $G$-invariant idempotent then
$$\Frac(A)^G= (e \Frac(A) e)^G \times ((1-e) \Frac(A)(1-e))^G,$$
thus we may assume that the action of $G$ on $\mathcal E$ is transitive. If $x\in \Frac(A)^G$ and $e x=0$ for some $e\in \mathcal E$ then using the transitivity of the action we obtain that $ex=0$ for all $e\in \mathcal E$ and so $x=\sum_{e\in \mathcal E} ex=0$. Hence, if $x\in \Frac(A)^G$ is non-zero then $ex$ is non-zero and we denote its inverse in the field $e \Frac(A) e$ by $(xe)^{-1}$. If we let $y= \sum_{e\in \mathcal E} (xe)^{-1} e \in \Frac(A)$ then $xy=1$. Thus $x$ is a unit in $\Frac(A)$ and its inverse $y$ is unique. Uniqueness implies that $y$ is $G$-invariant. Thus if $G$-acts transitively on $\mathcal E$ then $\Frac(A)^G$ is a field. 
\end{proof}

\begin{proof}[Proof of Theorem \ref{main_appendix}] 
Let $D^u: \GG\rightarrow R^{\ps}$ be the universal pseudorepresentation lifting 
$\overline{D}$. Let $\CH(D^u)$ be the closed two-sided ideal of $R^{\ps}\br{\GG}$ defined in \cite[\S 1.17]{che_durham}, so that  $E:= R^{\ps}\br{\GG}/ \CH(D^u)$ is the largest quotient of $R^{\ps}\br{\GG}$, where the Cayley--Hamilton theorem for $D^u$ holds. Following \cite[\S 1.17]{che_durham} we will call such algebras Cayley--Hamilton $R^{\ps}$-algebra of degree $d$. Then $E$ is a finitely generated $R^{\ps}$-module, \cite[Proposition 3.6]{WE_alg}. If $f: E\rightarrow M_d(B)$ is a homomorphism of $R^{\ps}$-algebras for a commutative $R^{\ps}$-algebra $B$ then we say $f$ is a
homomorphism of Cayley--Hamilton algebras 
if $\det\circ f: E\rightarrow B$ is equal to the specialization of $D^u$ along $R^{\ps}\rightarrow B$.

There is a commutative $R^{\ps}$-al\-ge\-bra $A^{\gen}$ together with a homomorphism of $R^{\ps}$-al\-ge\-bras, $j: E\rightarrow M_d(A^{\gen})$, satisfying the following universal property: if $f: E\rightarrow M_d(B)$ is a map of Cayley--Hamilton $R^{\ps}$-algebras for a commutative $R^{\ps}$-algebra $B$, then there is a unique map $\tilde{f}: A^{\gen}\rightarrow B$ of $R^{\ps}$-algebras such that $f= M_d(\tilde{f})\circ j$, see for example \cite[Theorem 3.8]{WE_alg} or \cite[Lemma 3.1]{BIP}.
Since $E$ is finitely generated as $R^{\ps}$-module, $A^{\gen}$ is of finite type over $R^{\ps}$.

Let $\Lambda_i: E\rightarrow R^{\ps}$, $0\le i\le d$ be the coefficients of the characteristic polynomial of $D^u$ - these are homogeneous polynomial laws satisfying 
$D^{u}(t-a)= \sum_{i=0}^n (-1)^i \Lambda_i(a) t^{d-i}$ in $R^{\ps}[t]$ for all $a\in E$, see \cite[Section 1.10]{che_durham}.
Now $E[1/p]$ is a $\QQ$-algebra and the pair $(E[1/p], \Lambda_1)$ is a trace algebra satisfying the $d$-dimensional Cayley--Hamilton identity in the sense of \cite[Definition 2.6]{Pro87}, see \cite[Footnote 10]{che_habil}. Moreover, for $\QQ$-algebras the homomorphisms of Cayley--Hamilton algebras coincide with the notion of maps of algebras with traces in \cite[Section 2.5]{Pro87}. Thus $j: E[1/p] \rightarrow M_d(A^{\gen}[1/p])$ is injective, its image is equal to the $\GL_d$-invariants, \cite[Theorem 2.6]{Pro87}. Moreover, $R^{\ps}[1/p]= (A^{\gen}[1/p])^{\GL_d}$, \cite[Proposition 2.3]{che_habil}, \cite[Theorem 2.20]{WE_alg}. By Lemma \ref{normal_inv} it is enough to show that $A^{\gen}[1/p]$ is normal. Further it is enough to show that the localization of $A^{\gen}[1/p]$ at every maximal ideal is normal, see Lemma 10.37.10 in \cite[\href{https://stacks.math.columbia.edu/tag/037B}{Tag 037B}]{stacks-project}. (The superscript \textit{gen} in $A^{\gen}$ stands for generic matrices in \cite[Section 1.1]{Pro87}.)

Let $\mm$ be a maximal ideal of $A^{\gen}[1/p]$. Its residue field $\kappa(\mm)$ is a finite extension of $L$, as $A^{\gen}[1/p]$ is finitely generated over $R^{\ps}[1/p]$. By specializing $j$ at $\mm$ we obtain a continuous representation $\rho:\GG\rightarrow \GL_d(\kappa(\mm))$ such that $$\det(1+ t \rho(g))=D^u\otimes_{R^{\ps}} \kappa(\mm)( 1 + tg),\quad  \forall g\in \GG.$$ 
This implies that if we choose a $\GG$-invariant $\OO_{\kappa(\mm)}$-lattice $\Theta$ in $\rho$ then the semi-simplification of $\Theta/\varpi_{\kappa(\mm)}\Theta$ is isomorphic to $\rhobar$, so that we are in the setup of Proposition \ref{framed1}. The universal property of $A^{\gen}$ implies that the completion of $A^{\gen}[1/p]$ at $\mm$ is the universal framed deformation ring $R^{\square}_{\rho}$, which is normal by Proposition \ref{framed1}. Since $A^{\gen}$ is finitely generated over $R^{\ps}$, which is a complete local noetherian ring, $A^{\gen}$ and hence its localization $(A^{\gen}[1/p])_{\mm}$ are excellent, and thus a $G$-ring. Lemma \ref{commalg1} implies that $(A^{\gen}[1/p])_{\mm}$ satisfies $(R_1)$ and $(S_2)$ and hence is normal. 

Let $\tilde{R}$ be the normalization of $R^{\ps}$ in $R^{\ps}[1/p]$. Then $(\Spf R^{\ps})^{\rig}=(\Spf \tilde{R})^{\rig}$, \cite[Lemma 7.2.2]{deJong}. Since $R^{\ps}[1/p]$ is normal, so is $\tilde{R}$ and thus $(\Spf \tilde{R})^{\rig}$ is normal, \cite[Proposition 7.2.4 (c)]{deJong}. Alternatively, one could use that the local rings of $(\Spf R^{\ps})^{\rig}$ are excellent, \cite[Theorem 1.1.3]{conrad}, and \cite[Lemma 7.1.9]{deJong} together with Lemma \ref{commalg1}.
\end{proof}

The following is a corollary to the proof; it does not require the assumption that $R^{\square}_{\rhobar}[1/p]$ is normal. 

\begin{cor}\label{faithful} 
Let $V$ be a free $R^{\square}_{\rhobar}[1/p]$-module of rank $d$ with $\GG$-action given by $\rho^{\square}: \GG \rightarrow \GL_d(R^{\square}_{\rhobar})$. Then $(R^{\ps}\br{\GG}/\CH(D^u))[1/p]$ acts faithfully on $V$. The same holds with the fixed determinant.
\end{cor}

\begin{proof} 
We use the notation of the proof of Theorem \ref{main_appendix}. So that $E= R^{\ps}\br{\GG}/ \CH(D^u)$ and there is a map $j: E\rightarrow M_d(A^{\gen})$, satisfying a universal property. This map is an injection after inverting $p$. Let $V^{\gen}$ be a free $A^{\gen}$-module of rank $d$, with $E$-action given by $j$. Thus $E[1/p]$ acts faithfully on $V^{\gen}[1/p]$. 

Suppose that $a\in E[1/p]$ kills off $V$. Since $E[1/p]$ acts faithfully on $V^{\gen}[1/p]$ there is a maximal ideal $\mm$ of $A^{\gen}[1/p]$, such that $a$ acts non-trivially on $V^{\gen}\otimes_{A^{\gen}} A^{\gen}_{\mm}$. Let $\hat{A}^{\gen}_{\mm}$ be the completion of $A^{\gen}_{\mm}$ with respect to the maximal ideal. Since $\hat{A}^{\gen}_{\mm}$ is faithfully flat over $A^{\gen}_{\mm}$, $a$ acts non-trivially on the completion of $V^{\gen}[1/p]$ at $\mm$, which we denote by $\hat{V}^{\gen}_{\mm}$. However, as explained in the proof of Theorem \ref{main_appendix}, $\hat{V}^{\gen}_{\mm}$ is isomorphic as an $E$-module to the completion of $V$ at a maximal ideal of $R^{\square}_{\rhobar}[1/p]$. Since $a$ annihilates $V$ it will also annihilate the completion giving a contradiction.

Let $E^{\psi}:= E\otimes_{R^{\ps}}R^{\ps, \psi}$ and let $A^{\gen, \psi}:=A^{\gen}\otimes_{R^{\ps}}R^{\ps, \psi}$. Then $j: E\rightarrow M_d(A^{\gen})$ induces a map $j: E^{\psi}\rightarrow M_d(A^{\gen, \psi})$, which satisfies the same universal property as $j$.   Then the same proof works with $A^{\gen, \psi}$ instead of $A^{\gen}$.
\end{proof}

\begin{lem}\label{centralizer} 
Let $R$ be a complete local noetherian $\OO$-algebra with residue field $k$ and 
let $\rho: \GG\rightarrow \GL_d(R)$ be a continuous representation. Assume that $R$ is $\OO$-torsion free and reduced, and the set of $x\in \mSpec R[1/p]$ such that $\rho_x$ is absolutely irreducible is dense in $\Spec R[1/p]$. Then 
$$\mathcal C:=\{ X\in M_d(R): X \rho(g)= \rho(g) X, \quad \forall g\in \GG\}$$
consists of scalar matrices. 
\end{lem}

\begin{proof}  
Let $X\in \mathcal C$ with matrix entries $x_{ij}$. Let $\pp_1, \ldots, \pp_n$ be the minimal primes of $R$. Since $R$ is reduced it embeds into $\prod_{s=1}^n \kappa(\pp_s)$. It is enough to show that the image of $X$ in $M_d(\kappa(\pp_s))$ for  $1\le s \le n$  is scalar, since if the images of $x_{ij}$ and $x_{ii}-x_{jj}$ for $i\neq j$ are zero in $\kappa(\pp_s)$ for $1\le s\le n$  then they are zero in $A$ and 
so $X$ is scalar. If $x\in \mSpec R[1/p]$ is such that $\rho_x: \GG\rightarrow \GL_d(\kappa(x))$ is absolutely irreducible then $\mathcal C\otimes_R \kappa(x)$ is $1$-dimensional. Since $\qq\mapsto \dim_{\kappa(\qq)} \mathcal C\otimes_R \kappa(\qq)$ is upper semi-continuous and such $x$ are dense, we deduce that $\dim_{\kappa(\pp_s)} \mathcal C\otimes_R \kappa(\pp_s)=1$ for all minimal primes $\pp_s$. This implies that $\mathcal C\otimes_R \kappa(\pp_s)$ consists of scalar matrices. 
\end{proof}

\begin{prop}\label{go_char} 
Assume that $R^{\square, \psi}_{\rhobar}[1/p]$ is non-zero. Then $R^{\square}_{\rhobar}[1/p]$ is normal (resp.~reduced) if and only if $R^{\square, \psi}_{\rhobar}[1/p]$ is normal (resp.~reduced).
\end{prop}

\begin{proof} 
Let $\Gamma$ be the pro-$p$ completion of the abelianization of $\GG$. As $\Hom^{\cont}_{\mathrm{grp}}(\GG, \Fp)$ is finite, $\Gamma \cong \Delta \times \Zp^r$, where $\Delta$ is a finite $p$-group. 
The map $\GG\rightarrow (R^{\square}_{\rhobar})^{\times}$, $g\mapsto \psi(g)^{-1} \det \rho^{\square}(g)$ factors through $\Gamma$ and thus induces an $\OO$-algebra homomorphism $\OO\br{\Gamma}\rightarrow R^{\square}_{\rhobar}$ and $R^{\square, \psi}_{\rhobar}$ is equal to the quotient of $R^{\square}_{\rhobar}$ by the augmentation ideal in $\OO\br{\Gamma}$.

Let $\mathcal X(\Gamma)$ be the functor which sends a local artinian $\OO$-algebra
$(A, \mm_A)$ to the group of continuous characters $\chi: \Gamma\rightarrow 1+\mm_A$. This functor is represented by $\Spf \OO\br{\Gamma}$.  For such $(A, \mm_A)$ the group $\mathcal X(\Gamma)(A)$ acts on $D^{\square}(A)$ by twisting. The action induces a homomorphism of local $\OO$-algebras $\gamma: R^{\square}_{\rhobar}\rightarrow R^{\square}_{\rhobar}\wtimes_{\OO} \OO\br{\Gamma}$. Let $R^{\inv}=\{ a\in R^{\square}_{\rhobar}: \gamma(a)= a\otimes 1\}$ be the subring of $\mathcal X(\Gamma)$-invariants in $R^{\square}_{\rhobar}$.  Analogously $\mathcal X(\Delta)$ acts on $D^{\square}$, the action induces
the map $\delta:  R^{\square}_{\rhobar}\rightarrow R^{\square}_{\rhobar}\wtimes_{\OO} \OO[\Delta]$ and we let $R^{\invt}=\{ a\in R^{\square}_{\rhobar}: \delta(a)= a\otimes 1\}$ be the subring of $\mathcal X(\Delta)$-invariants in $R^{\square}_{\rhobar}$.

The action of $\mathcal X(\Gamma)$ and $\mathcal X(\Delta)$ on $D^{\square}$ is free, since if $\rho_A: \gal\rightarrow \GL_d(A)$ is a framed deformation of $\rhobar$ then for each $g\in \gal$ at least one matrix entry of $\rho_A(g)$ will not lie in $\mm_A$, and thus is a unit. Hence, $\rho_A(g)=\rho_A(g) \chi_A(g)$ for all $g\in G$ implies that $\chi_A$ is the trivial character. 

The map $R^{\invt} \rightarrow R^{\square}_{\rhobar}$ is finite and becomes \'etale after inverting $p$ by \cite[Proposition 1.1.11 (2)]{allen}. Thus $R^{\square}_{\rhobar}[1/p]$ is normal if and only if $R^{\invt}[1/p]$ is normal by Lemma \ref{commalg3}. Since $R^{\inv}$ is the subring of $\mathcal X(\Gamma/\Delta)$-invariants in $R^{\invt}$ and $\Gamma/\Delta\cong \Zp^r$, 
we have $R^{\invt}\cong R^{\inv}\br{x_1, \ldots, x_r}$ by \cite[Proposition 1.1.11 (2)]{allen}. Thus $R^{\invt}[1/p]$ is normal if and only if $R^{\inv}[1/p]$ is normal by Lemma \ref{commalg2}. The map $R^{\inv}\rightarrow R^{\square, \psi}_{\rhobar}$ is finite and becomes \'etale after inverting $p$ by \cite[Proposition 1.1.11 (3)]{allen}.  Lemma \ref{commalg3} implies that $R^{\square, \psi}_{\rhobar}[1/p]$ is normal if and only if $R^{\inv}[1/p]$ is normal. Putting all the equivalences together proves the assertion. 

Since reducedness is equivalent to ($R_0$) and ($S_1$) the same proof works. 
\end{proof} 

\begin{cor}\label{fix_char_nor} 
If $R^{\square}_{\rhobar}[1/p]$ is normal then $R^{\ps, \psi}[1/p]$ and the associated rigid analytic space $(\Spf R^{\ps, \psi})^{\rig}$ are normal.
\end{cor}

\begin{proof} 
Proposition \ref{go_char} implies that $R^{\square, \psi}_{\rhobar}[1/p]$ is normal. The proof of Theorem \ref{main_appendix} with $R^{\square}_{\rhobar}[1/p]$ replaced by $R^{\square, \psi}_{\rhobar}[1/p]$ implies the assertion. 
\end{proof}

\begin{prop}\label{Ztf} 
Let $E=R^{\ps}\br{\GG}/\CH(D^u)$, $E_{\tf}$ the maximal $\OO$-torsion free quotient of $E$, let $Z(E_{\tf})$ be the centre of $E_{\tf}$ and let $R^{\ps}_{\tf}$ be the maximal $\OO$-torsion free quotient of $R^{\ps}$. Then $R^{\ps}_{\tf}$ is a subring of $Z(E_{\tf})$. 

If $R^{\square}_{\rhobar}[1/p]$ is reduced and the set $x\in \mSpec R^{\square}_{\rhobar}[1/p]$, such that $\rho^{\square}_x$ is absolutely irreducible, is dense in $\Spec R^{\square}_{\rhobar}[1/p]$ then $d \cdot Z(E_{\tf}) \subset R^{\ps}_{\tf}$. In particular, if $p\nmid d$ then $R^{\ps}_{\tf}=Z(E_{\tf})$. Moreover, the same holds for the rings with the fixed determinant.
\end{prop}

\begin{proof} 
As in the proof of Theorem \ref{main_appendix} there is an injection $$j: E[1/p] \hookrightarrow M_d(A^{\gen}[1/p]).$$ Moreover, $\tr \circ j$ induces a surjection $E[1/p]\twoheadrightarrow R^{\ps}[1/p]$. Thus $R^{\ps}_{\tf}$ is a subring of $Z(E_{\tf})$.  Corollary \ref{faithful} gives us an injection $E_{\tf} \hookrightarrow M_d(R^{\square}_{\rhobar}[1/p])$. If $a\in E_{\tf}$ then the characteristic polynomial of $j(a)$ has coefficients in $R^{\ps}_{\tf}$. Moreover, $Z(E_{\tf})$ is contained in the centralizer of $\rho^{\square}(\GG)$ in $M_d(R^{\square}_{\rhobar}[1/p])$. According to Lemma \ref{centralizer} the centralizer is equal to scalar matrices.  Since $j(z)$ is a scalar matrix, we deduce that $d z \in R^{\ps}_{\tf}$. 

It follows from Proposition \ref{go_char} that $R^{\square, \psi}_{\rhobar}[1/p]$ is reduced. Since twisting by characters does not change the property of being absolutely irreducible, the proof of Proposition \ref{go_char} shows that absolutely irreducible locus is dense in $\Spec R^{\square, \psi}_{\rhobar}[1/p]$. Then the same proof goes through.
\end{proof}
 
\begin{prop}\label{nailp3}\footnote{The statement is proved in \cite[Corollary 4.22]{BIP} without computing the equations for the deformation ring.}
If $p=3$, $\GG=\GG_{\QQ_3}$ and $\rhobar= \Eins\oplus \omega$ then $R^{\square}_{\rhobar}[1/p]$ is normal.
\end{prop}

\begin{proof} 
We will first relate the framed deformation ring $R^{\square}_{\rhobar}$ to the ring studied in \cite{boeckle}. Let $\mu_3$ be the group of $3$-rd roots of unity in $\Qtbar$, let $E=\QQ_3(\mu_3)$ and let $E(3)$ be the compositum of all extensions $E \subset E' \subset \Qtbar$ such that $[E':E]$ is a power of $3$. Then the Galois group $\Gal(E(3)/E)$ is the maximal pro-$3$ quotient of $\Gal(\Qtbar/E)$, and thus the map $\rho^{\square}: \GG_{\QQ_3} \rightarrow \GL_2(R^{\square}_{\rhobar})$ factors through the surjection $\GG_{\QQ_3}\twoheadrightarrow \Gal(E(3)/\QQ_3)$. Since $\Gal(E/ \QQ_3)$ has order $2$, Schur--Zassenhaus implies that the surjection $\Gal(E(3)/\QQ_3)\twoheadrightarrow \Gal(E/ \QQ_3)$ has a splitting, which gives us an isomorphism 
$$\Gal(E(3)/\QQ_3)\cong \Gal(E(3)/E) \rtimes G,$$
where $G=\{1, \sigma\}$ is a subgroup of $\Gal(E(3)/\QQ_3)$. One may define a closed 
subfunctor, denoted by $\mathrm{EH}_1$ in \cite{boeckle}, of the framed deformation functor 
$D^{\square}$, such that $\mathrm{EH}_1(A)$ consists of pairs $(V_A, \beta_A)$, where 
$V_A$ is a deformation of $\omega \oplus 1$ to $A$, and $\beta_A=(v_1, v_2)$ is an 
$A$-basis of $V_A$ lifting a fixed basis $\beta_k$ of $\omega \oplus 1$, such that 
$\sigma$ acts by $-1$ on $v_1$ by  and by $1$ on $v_2$. It follows from 
the Iwahori decomposition for the group $1+M_2(\mm_A)$ that a framed deformation 
$(V_A, \beta_A)\in D^{\square}(A)$ can be conjugated to a framed deformation in 
$\mathrm{EH}_1(A)$ by a unique element of the form 
$\bigl ( \begin{smallmatrix} 1 & b \\ 0 & 1 \end{smallmatrix} \bigr )
\bigl ( \begin{smallmatrix} 1 & 0\\ c & 1 \end{smallmatrix} \bigr )$, with $b, c \in \mm_A$. 
Hence, if $\mathrm{EH}_1$ is represented by $R$ then 
$$R^{\square}_{\rhobar}\cong R\br{x,y}.$$ 
Now $R$ is complete intersection by \cite[Theorem 1.1]{boeckle}, thus so is $R^{\square}_{\rhobar}$, 
and to show normality of $R^{\square}_{\rhobar}[1/p]$ it is enough to show that the singular locus in 
$R[1/p]$ has codimension at least $2$. 

Let $\rho: \Gal(E(3)/\QQ_3)\rightarrow \GL_2(R)$ be the representation obtained for the action of the Galois group on $V_R$ with respect to the basis $\beta_{R}$, so that $\rho(\sigma)= \bigl ( \begin{smallmatrix} -1 & 0 \\ 0 & 1\end{smallmatrix} \bigr )$. It follows from the argument of \cite[Lemma 4.1]{CDP2} that $x\in \mSpec R[1/p]$ is singular if and only if there is an exact sequence $0\rightarrow \delta\rightarrow \rho_x \rightarrow \delta \varepsilon \rightarrow 0$ for some character $\delta: \Gal(E(3)/\QQ_3)\rightarrow \kappa(x)^{\times}$. Thus the singular locus is contained in the reducible locus and it is enough to show that it has positive codimension inside the reducible locus: we know that $R$ is a domain by \cite[Theorem 1.1]{boeckle}, and there are absolutely irreducible lifts of $\rhobar$, so that the reducible locus has codimension $1$ inside $\Spec R$. 

We will now describe the ring $R$ as computed in \cite{boeckle} and will compute the reducible locus. Lemma 3.2 in \cite{boeckle} says that the representation  
$\rho: \Gal(E(3)/\QQ_3)\rightarrow \GL_2(R)$ factors through a quotient $\Gal(E(3)/\QQ_3)\twoheadrightarrow P \rtimes G$, where $P$ is a pro-$p$ group 
with generators $x_1, x_2, x_3, x_4$ and one relation 
$$r= x_1^3 [x_1, x_2][x_3, x_4][x_4, x_3^{-1}] [x_2, x_1^{-1}] x_1^3,$$
where $[g, h]=g hg^{-1} h^{-1}$. The action of $\sigma\in G$ on the generators is given by  
$$\sigma(x_1)= x_1^{-1}, \quad \sigma(x_2)= x_2, \quad \sigma(x_3)=x_3^{-1}, \quad \sigma(x_4)=x_4.$$
Let $S=\OO\br{a, a', b, b', c, c', d, d'}$ and let $A_i\in \GL_2(S)$ be the matrices 
$$A_1=\biggl(\begin{smallmatrix} \sqrt{1+bc} & b \\ c & \sqrt{1+bc}\end{smallmatrix}\biggr), 
\quad A_2= \sqrt{1+a}\biggl(\begin{smallmatrix} \sqrt{1+d} & 0 \\ 0 & \sqrt{1+d}^{-1}\end{smallmatrix}\biggr),$$
$$A_3=\biggl(\begin{smallmatrix} \sqrt{1+b'c'} & b' \\ c' & \sqrt{1+b'c'}\end{smallmatrix}\biggr), 
\quad A_4= \sqrt{1+a'}\biggl(\begin{smallmatrix} \sqrt{1+d'} & 0 \\ 0 & \sqrt{1+d'}^{-1}\end{smallmatrix}\biggr),$$
and let $$B= A_1^3 [A_1, A_2][A_3, A_4][A_4, A_3^{-1}] [A_2, A_1^{-1}] A_1^3.$$
Then \cite[Theorem 4.1 (c)]{boeckle} asserts that $R= S/I$, where $I$ is the ideal of $S$ generated by the matrix entries of $B - \bigl( \begin{smallmatrix}1 & 0 \\ 0 & 1\end{smallmatrix}\bigr)$ and $\rho: P\rtimes G\rightarrow \GL_2(R)$ is obtained by mapping $\sigma\mapsto \bigl ( \begin{smallmatrix} -1 & 0 \\ 0 & 1\end{smallmatrix} \bigr )$ and $x_i\mapsto A_i$ for $1\le i\le 4$. 
 
It follows from this description that the locus in $\Spec R$ parameterizing reducible representations,  where $\sigma$ acts on the rank $1$ subrepresentation by $-1$ (resp.~$1$)  is equal to $V( c, c')$  (resp. $V(b, b'))$. 
 
The images of $A_1$ and $A_3$ in $\GL_2(R/( c, c'))$ are unipotent upper-triangular matrices. It is easy to compute the commutator of a unipotent upper-triangular matrix with a diagonal matrix. One obtains that the image of $B$ in $\GL_2(S/( c, c'))$ is the matrix $( \begin{smallmatrix} 1 & 6b- 2bd -2 b'd' \\ 0 & 1\end{smallmatrix} \bigr )$. Thus $$R/( c, c')= S/(c, c', 3b -bd -b'd')$$ is an integral domain, as $S/(c, c')\cong \OO\br{a, a', b, b', d, d'}$ is factorial and $3b -bd -b'd'$ is an irreducible element in $S/(c, c')$. 
 
Let $X^{\mathrm{sing}}$ be the singular locus in $\Spec R[1/p]$. The point $x\in \Spec R/(c,c')$ corresponding to the representation $\bigl( \begin{smallmatrix} \varepsilon^3 & 0 \\ 0 & 1 \end{smallmatrix}\bigr)$ will not lie in $X^{\mathrm{sing}}$, since this representation is not an extension of $\delta \varepsilon$ by $\delta$. Thus $X^{\mathrm{sing}} \cap \Spec R/(c,c')[1/p]$ is at least of codimension $1$. In the same way we obtain that $X^{\mathrm{sing}} \cap \Spec R/(b,b')[1/p]$ is at least of codimension $1$ in $\Spec R/(b,b')[1/p]$. Thus $X^{\mathrm{sing}}$ is at least of codimension $1$ in the reducible locus in $\Spec R[1/p]$. Thus $X^{\mathrm{sing}}$ is at least of codimension $2$ in $\Spec R[1/p]$. 
\end{proof}

\begin{prop}\label{normal_framed} 
If $\GG=\gal$ then $R^{\square}_{\rhobar}[1/p]$ is normal and the absolutely irreducible locus is dense in $\Spec  R^{\square}_{\rhobar}[1/p]$ for all semi-simple $2$-dimensional $\rhobar$. 
\end{prop}

\begin{proof} 
Since $R^{\square}_{\rhobar}[1/p]$ is excellent the singular locus is closed in 
$R^{\square}_{\rhobar}[1/p]$. If non-empty then it will contain a maximal ideal $x$, such that 
$$\Hom_{\gal}(\rho^{\square}_x, \rho^{\square}_x(1))\neq 0,$$ 
see \cite[Lemma 4.1]{CDP2}. Thus $\rhobar$ is of the form $\bar{\chi} \oplus \bar{\chi} \omega$. After twisting by a character we may assume that $\rhobar= \Eins \oplus \omega$. If $p=2$ or $p=3$ then $R^{\square}_{\rhobar}[1/p]$ is normal by \cite[Proposition 4.3]{CDP2}, and Proposition \ref{nailp3}, respectively. If $p\ge 5$ it follows from the proof of \cite[Proposition B2, Theorem B.3]{image}, based on the work of B\"ockle \cite{boeckle2}, that $R^{\square}_{\rhobar}$ is formally smooth over $\OO\br{x,y, z,w}/(xy-zw)$. (The only change is that because in our setting $\rhobar$ is split, the generator $x_{p-2}$ maps to the matrix $\bigl( \begin{smallmatrix} 1 & x \\ 0 & 1\end{smallmatrix}\bigr)$, instead of $\bigl( \begin{smallmatrix} 1 & 1 \\ 0 & 1\end{smallmatrix}\bigr)$. This adds an extra variable, but does not change the relation coming from \cite[Equation (261)]{image}.) Thus $R^{\square}_{\rhobar}[1/p]$ is normal.

Hence, $R^{\square}_{\rhobar}[1/p]$ is a product of normal domains, and if absolutely irreducible locus was not dense there would be a component without absolutely irreducible points.
(Let $I$ be the ideal of $R^{\square}_{\rhobar}$ generated by the matrix entries of 
$(\rho^{\square}(gh)-\rho^{\square}(hg))^2$ for all $g, h\in \gal$. Then a specialization of $\rho^{\square}$ at  $x\in \mSpec R^{\square}_{\rhobar}[1/p]$  is absolutely irreducible over $\kappa(x)$ if and only if $x\not\in V(I)$.
Thus if an irreducible component of $R^{\square}_{\rhobar}[1/p]$ contains an absolutely irreducible point then 
such points are dense in the component.) In the course of the proof of Proposition \ref{Ztf} we have shown that $R^{\ps}[1/p]$ is a subring of $R^{\square}_{\rhobar}[1/p]$. Thus there would exist an irreducible component of $R^{\ps}[1/p]$ without absolutely irreducible points. This would contradict \cite[Theorem 2.1]{che_unpublished}.
\end{proof}

\begin{cor}\label{normal} 
If $\GG=\gal$ then $R^{\ps, \psi}[1/p]$, $R^{\ps}[1/p]$ and the corresponding rigid analytic spaces are normal for all semi-simple $2$-dimensional $\rhobar$. 
\end{cor}

\begin{proof} 
The assertion follows from Proposition \ref{normal_framed} and  Theorem \ref{main_appendix} and Corollary \ref{fix_char_nor}.
\end{proof}

\bibliographystyle{plain}
\bibliography{Ref}

\begin{thebibliography}{10}

\bibitem{allen}
Patrick Allen.
\newblock Deformations of {H}ilbert modular {G}alois representations and
  adjoint {S}elmer groups.
\newblock \url{https://faculty.math.illinois.edu/~pballen/research/smooth.pdf}.

\bibitem{MR1290194}
L.~Barthel and R.~Livn\'e.
\newblock Irreducible modular representations of {${\rm GL}_2$} of a local
  field.
\newblock {\em Duke Math. J.}, 75(2):261--292, 1994.

\bibitem{bel_che}
Jo\"{e}l Bella\"{\i}che and Ga\"{e}tan Chenevier.
\newblock Families of {G}alois representations and {S}elmer groups.
\newblock {\em Ast\'{e}risque}, (324):xii+314, 2009.

\bibitem{MR2642406}
Laurent Berger and Christophe Breuil.
\newblock Sur quelques repr\'esentations potentiellement cristallines de {${\rm
  GL}_2({\Bbb Q_p})$}.
\newblock {\em Ast\'erisque}, (330):155--211, 2010.

\bibitem{bernstein}
J.~N. Bernstein.
\newblock Le ``centre'' de {B}ernstein.
\newblock In {\em Representations of reductive groups over a local field},
  Travaux en Cours, pages 1--32. Hermann, Paris, 1984.
\newblock Edited by P. Deligne.

\bibitem{boeckle2}
Gebhard B\"{o}ckle.
\newblock Demu\v{s}kin groups with group actions and applications to
  deformations of {G}alois representations.
\newblock {\em Compositio Math.}, 121(2):109--154, 2000.

\bibitem{boeckle}
Gebhard B\"{o}ckle.
\newblock Deformation rings for some mod 3 {G}alois representations of the
  absolute {G}alois group of {$\bold Q_3$}.
\newblock {\em Ast\'{e}risque}, (330):529--542, 2010.

\bibitem{BIP}
Gebhard B\"ockle, Ashwin Iyengar, and Vytautas Pa\v{s}k\={u}nas.
\newblock On local {G}alois deformation rings, 2021.
\newblock \url{https://arxiv.org/abs/2110.01638}.

\bibitem{MR2018825}
Christophe Breuil.
\newblock Sur quelques repr\'esentations modulaires et {$p$}-adiques de {${\rm
  GL}_2(\mathbb Q_p)$}. {I}.
\newblock {\em Compositio Math.}, 138(2):165--188, 2003.

\bibitem{MR2667890}
Christophe Breuil and Matthew Emerton.
\newblock Repr\'{e}sentations {$p$}-adiques ordinaires de {${\rm GL}_2({\Bbb
  Q_p})$} et compatibilit\'{e} local-global.
\newblock {\em Ast\'{e}risque}, (331):255--315, 2010.

\bibitem{MR1944572}
Christophe Breuil and Ariane M\'{e}zard.
\newblock Multiplicit\'{e}s modulaires et repr\'{e}sentations de {${\rm
  GL}_2({\bf Z}_p)$} et de {${\rm Gal}(\overline{\bf Q}_p/{\bf Q}_p)$} en
  {$l=p$}.
\newblock {\em Duke Math. J.}, 115(2):205--310, 2002.
\newblock With an appendix by Guy Henniart.

\bibitem{brumer}
Armand Brumer.
\newblock Pseudocompact algebras, profinite groups and class formations.
\newblock {\em J. Algebra}, 4:442--470, 1966.

\bibitem{BH}
Winfried Bruns and J\"{u}rgen Herzog.
\newblock {\em Cohen--{M}acaulay rings}, volume~39 of {\em Cambridge Studies in
  Advanced Mathematics}.
\newblock Cambridge University Press, Cambridge, 1993.

\bibitem{MR1711578}
Colin~J. Bushnell and Philip~C. Kutzko.
\newblock Semisimple types in {${\rm GL}_n$}.
\newblock {\em Compositio Math.}, 119(1):53--97, 1999.

\bibitem{6auth}
Ana Caraiani, Matthew Emerton, Toby Gee, David Geraghty, Vytautas
  Pa\v{s}k\={u}nas, and Sug~Woo Shin.
\newblock Patching and the {$p$}-adic local {L}anglands correspondence.
\newblock {\em Camb. J. Math.}, 4(2):197--287, 2016.

\bibitem{che_habil}
Ga\"{e}tan Chenevier.
\newblock Repr\'esentations galoisiennes automorphes et cons\'equences
  arithm\'etiques des conjectures de {L}anglands et {A}rthur.
\newblock Habilitation thesis, 2013.

\bibitem{che_unpublished}
Ga\"{e}tan Chenevier.
\newblock Sur la vari\'et\'e des caract\`eres $p$-adique du groupe de {G}alois
  absolu de $\mathbb{Q}_p$.
\newblock
  \url{http://gaetan.chenevier.perso.math.cnrs.fr/articles/lieugalois.pdf},
  2010.

\bibitem{che_durham}
Ga\"{e}tan Chenevier.
\newblock The {$p$}-adic analytic space of pseudocharacters of a profinite
  group and pseudorepresentations over arbitrary rings.
\newblock In {\em Automorphic forms and {G}alois representations. {V}ol. 1},
  volume 414 of {\em London Math. Soc. Lecture Note Ser.}, pages 221--285.
  Cambridge Univ. Press, Cambridge, 2014.

\bibitem{MR2642409}
Pierre Colmez.
\newblock Repr\'esentations de {${\rm GL}_2(\mathbb Q_p)$} et
  {$(\phi,\Gamma)$}-modules.
\newblock {\em Ast\'erisque}, (330):281--509, 2010.

\bibitem{MR3267142}
Pierre Colmez and Gabriel Dospinescu.
\newblock Compl\'{e}t\'{e}s universels de repr\'{e}sentations de
  {$\text{GL}_2(\Bbb{Q}_p)$}.
\newblock {\em Algebra Number Theory}, 8(6):1447--1519, 2014.

\bibitem{CDP}
Pierre Colmez, Gabriel Dospinescu, and Vytautas Pa\v{s}k\={u}nas.
\newblock The {$p$}-adic local {L}anglands correspondence for {${\rm GL}_2(\Bbb
  Q_p)$}.
\newblock {\em Camb. J. Math.}, 2(1):1--47, 2014.

\bibitem{CDP2}
Pierre Colmez, Gabriel Dospinescu, and Vytautas Pa\v{s}k\={u}nas.
\newblock Irreducible components of deformation spaces: wild $2$-adic
  exercises.
\newblock {\em Int. Math. Res. Not. IMRN}, (14):5333--5356, 2015.

\bibitem{conrad}
Brian Conrad.
\newblock Irreducible components of rigid spaces.
\newblock {\em Ann. Inst. Fourier (Grenoble)}, 49(2):473--541, 1999.

\bibitem{deJong}
A.~J. de~Jong.
\newblock Crystalline {D}ieudonn\'{e} module theory via formal and rigid
  geometry.
\newblock {\em Inst. Hautes \'{E}tudes Sci. Publ. Math.}, (82):5--96 (1996),
  1995.

\bibitem{DPS}
Gabriel Dospinescu, Vytautas Pa\v{s}k\={u}nas, and Benjamin Schraen.
\newblock Infinitesimal characters in arithmetic families.
\newblock 2020.
\newblock \url{https://arxiv.org/abs/2012.01041}.

\bibitem{DS}
Gabriel Dospinescu and Benjamin Schraen.
\newblock Endomorphism algebras of admissible {$p$}-adic representations of
  {$p$}-adic {L}ie groups.
\newblock {\em Represent. Theory}, 17:237--246, 2013.

\bibitem{MR2181093}
Matthew Emerton.
\newblock {$p$}-adic {$L$}-functions and unitary completions of representations
  of {$p$}-adic reductive groups.
\newblock {\em Duke Math. J.}, 130(2):353--392, 2005.

\bibitem{MR2667882}
Matthew Emerton.
\newblock Ordinary parts of admissible representations of {$p$}-adic reductive
  groups {I}. {D}efinition and first properties.
\newblock {\em Ast\'erisque}, (331):355--402, 2010.

\bibitem{MR2667883}
Matthew Emerton.
\newblock Ordinary parts of admissible representations of {$p$}-adic reductive
  groups {II}. {D}erived functors.
\newblock {\em Ast\'{e}risque}, (331):403--459, 2010.

\bibitem{MR2667892}
Matthew Emerton and Vytautas Pa\v{s}k\={u}nas.
\newblock On the effaceability of certain {$\delta$}-functors.
\newblock {\em Ast\'{e}risque}, (331):461--469, 2010.

\bibitem{MR232821}
Pierre Gabriel.
\newblock Des cat\'{e}gories ab\'{e}liennes.
\newblock {\em Bull. Soc. Math. France}, 90:323--448, 1962.

\bibitem{MR2867622}
Philippe Gille and Patrick Polo, editors.
\newblock {\em Sch\'{e}mas en groupes ({SGA} 3). {T}ome {III}. {S}tructure des
  sch\'{e}mas en groupes r\'{e}ductifs}, volume~8 of {\em Documents
  Math\'{e}matiques (Paris) [Mathematical Documents (Paris)]}.
\newblock Soci\'{e}t\'{e} Math\'{e}matique de France, Paris, 2011.
\newblock S\'{e}minaire de G\'{e}om\'{e}trie Alg\'{e}brique du Bois Marie
  1962--64. [Algebraic Geometry Seminar of Bois Marie 1962--64], A seminar
  directed by M. Demazure and A. Grothendieck with the collaboration of M.
  Artin, J.-E. Bertin, P. Gabriel, M. Raynaud and J-P. Serre, Revised and
  annotated edition of the 1970 French original.

\bibitem{Iy}
Ashwin Iyengar.
\newblock Deformation theory of the trivial mod {$p$} {G}alois representation
  for {${\rm GL}_n$}.
\newblock {\em Int. Math. Res. Not. IMRN}, (22):8896--8935, 2020.

\bibitem{gabber}
Mark Kisin.
\newblock Modularity for some geometric {G}alois representations.
\newblock In {\em {$L$}-functions and {G}alois representations}, volume 320 of
  {\em London Math. Soc. Lecture Note Ser.}, pages 438--470. Cambridge Univ.
  Press, Cambridge, 2007.
\newblock With an appendix by Ofer Gabber.

\bibitem{kisin_moduli}
Mark Kisin.
\newblock Moduli of finite flat group schemes, and modularity.
\newblock {\em Ann. of Math. (2)}, 170(3):1085--1180, 2009.

\bibitem{matsumura}
Hideyuki Matsumura.
\newblock {\em Commutative ring theory}, volume~8 of {\em Cambridge Studies in
  Advanced Mathematics}.
\newblock Cambridge University Press, Cambridge, second edition, 1989.
\newblock Translated from the Japanese by M. Reid.

\bibitem{luepan2}
Lue Pan.
\newblock The {F}ontaine--{M}azur conjecture in the residually reducible case.
\newblock 2019.
\newblock \url{https://arxiv.org/abs/1901.07166}.

\bibitem{luepan}
Lue Pan.
\newblock On locally analytic vectors of the completed cohomology of modular
  curves.
\newblock 2020.
\newblock \url{https://arxiv.org/abs/2008.07099}.

\bibitem{image}
Vytautas Pa{\v{s}}k\=unas.
\newblock The image of {C}olmez's {M}ontreal functor.
\newblock {\em Publ. Math. Inst. Hautes \'Etudes Sci.}, 118:1--191, 2013.

\bibitem{MR3444235}
Vytautas Pa{\v{s}}k\=unas.
\newblock Blocks for {${\rm mod}\,p$} representations of {${\rm GL}_2(\Bbb
  Q_p)$}.
\newblock In {\em Automorphic forms and {G}alois representations. {V}ol. 2},
  volume 415 of {\em London Math. Soc. Lecture Note Ser.}, pages 231--247.
  Cambridge Univ. Press, Cambridge, 2014.

\bibitem{2adicANT}
Vytautas Pa{\v{s}}k\=unas.
\newblock On $2$-dimensional $2$-adic {G}alois representations of local and
  global fields.
\newblock {\em Algebra Number Theory}, 10(6):1301--1358, 2016.

\bibitem{ludwig}
Vytautas Pa{\v{s}}k\=unas.
\newblock On some consequences of a theorem of {J}.~{L}udwig.
\newblock {\em J. Math. Jussieu}, 2021.
\newblock \url{https://doi.org/10.1017/S1474748020000547}.

\bibitem{Pro87}
Claudio Procesi.
\newblock A formal inverse to the {C}ayley--{H}amilton theorem.
\newblock {\em J. Algebra}, 107(1):63--74, 1987.

\bibitem{st_iw}
P.~Schneider and J.~Teitelbaum.
\newblock {B}anach space representations and {I}wasawa theory.
\newblock {\em Israel J. Math.}, 127:359--380, 2002.

\bibitem{stacks-project}
The {Stacks Project Authors}.
\newblock \textit{Stacks Project}.
\newblock \url{https://stacks.math.columbia.edu}, 2018.

\bibitem{2018arXiv180307451T}
Shen-Ning {Tung}.
\newblock {On the automorphy of 2-dimensional potentially semi-stable
  deformation rings of $\mathrm{Gal}(\overline{\mathbb{Q}}_p/\mathbb{Q}_p)$}.
\newblock {\em arXiv e-prints}, page arXiv:1803.07451, Mar 2018.

\bibitem{Tung2020}
Shen-Ning Tung.
\newblock On the modularity of $2$-adic potentially semi-stable deformation
  rings.
\newblock {\em Math. Z.}, pages 1--53, 2020.

\bibitem{WE_alg}
Carl Wang-Erickson.
\newblock Algebraic families of {G}alois representations and potentially
  semi-stable pseudodeformation rings.
\newblock {\em Math. Ann.}, 371(3-4):1615--1681, 2018.

\end{thebibliography}

\end{document}